\documentclass[a4paper,12pt,twoside]{article}

\title{Conditional Joint Probability Distributions of First Exit Times to Overlapping Absorbing Sets of the Mixture of Markov Jump Processes}

 \author{B.A. Surya\footnote{School of Mathematics and Statistics, Victoria University of Wellington, Gate 6 Kelburn PDE, Wellington 6140, New Zealand. Email address: budhi.surya@vuw.ac.nz }\\ School of Mathematics and Statistics \\ Victoria University of Wellington, New Zealand }

\date{}

\usepackage{latexsym,epic,eepic,float}
\usepackage{color,colordvi}
\usepackage{shadow}
\usepackage{amsmath,amssymb,amsfonts,stackrel}
\usepackage{amsmath}
\usepackage{epsfig}
\usepackage{lscape}
\usepackage{rotating}
\usepackage{subfigure}
\usepackage{multicol}
\usepackage{tikz}
\usepackage{graphicx}
\usepackage{bbold}
\usetikzlibrary{automata,positioning}

\newtheorem{theorem}{Theorem}[section]
\newtheorem{theo}[theorem]{Theorem}
\newtheorem{lem}[theorem]{Lemma}

\newtheorem{cor}[theorem]{Corollary}
\newtheorem{prop}[theorem]{Proposition}
\newtheorem{Rem}[theorem]{Remark}
\newtheorem{Ex}[theorem]{Example}

\newcommand{\exit}{{\mbox{\, \vspace{3mm}}} \hfill\mbox{$\square$}}

\newcommand{\subf}[2]{%
  {\small\begin{tabular}[t]{@{}c@{}}
  #1\\#2
  \end{tabular}}%
}

\rm 

\numberwithin{equation}{section}

\date{12 May 2018}

\begin{document}

\maketitle \pagestyle{myheadings} \markboth{B.A. Surya}{Conditional Joint Probability Distributions of First Exit Times}

\begin{abstract}
New results on conditional joint probability distributions of first exit times are presented for a continuous-time stochastic process defined as the mixture of Markov jump processes moving at different speeds on the same finite state space, while the mixture occurs at a random time. Such mixture was first proposed by Frydman \cite{Frydman2005} and Frydman and Schuermann \cite{Frydman2008} as a generalization of the mover-stayer model of Blumen et at. \cite{Blumen}, and was recently extended by Surya \cite{Surya2018}, in which explicit distributional identities of the process are given, in particular in the presence of an absorbing state. We revisit \cite{Surya2018} for a finite mixture with overlapping absorbing sets. The contribution of this paper is two fold. First, we generalize distributional properties of the mixture discussed in \cite{Frydman2008} and \cite{Surya2018}. Secondly, we give distributional identities of the first exit times explicitly in terms of intensity matrices of the underlying Markov processes and the Bayesian updates of switching probability and of the probability distribution of states, despite the fact that the mixture itself is non-Markov. They form non-stationary functions of time and have the ability to capture heterogeneity and path dependence when conditioning on the available information (either full or partial) of the process. In particular, initial profile of the distributions forms of a generalized mixture of multivariate phase-type distributions of Assaf et al. \cite{Assaf1984}. When the underlying processes move at the same speed, in which case the mixture becomes a simple Markov process, these features are removed, and the initial distributions reduce to \cite{Assaf1984}. Some explicit and numerical examples are discussed to illustrate the main results.

\medskip

\textbf{MSC2010 Subject Classification:} 60J20, 60J27, 60J28, 62N99


\end{abstract}

\section{Introduction}
Markov process has been one of the most important probabilistic tools in modeling complex stochastic systems dynamics. It has been widely used in variety of applications across various fields such as, among others, in modeling vegetation dynamics (Balzter \cite{Balzter}), demography (Nowak \cite{Nowak}), in marketing to model consumer relationship (Berger and Nasr \cite{Berger} and Pfeifer and Carraway \cite{Pfeifer}) and to identify substitutions behavior of customers in assortment problem (Blanchet et al. \cite{Blanchet}), in describing credit rating transitions used in many credit risk and pricing applications (Jarrow and Turnbull \cite{Jarrow1995}, Jarrow et al. \cite{Jarrow1997}, Bielecki et al. \cite{Bielecki}), in queueing networks and performance engineering (Bolch et al. \cite{Bolch}).

One of the key variables in the analysis of stochastic systems is the time until an event occurs (the lifetime of systems), for example, the lifetime of a corporate bond \cite{Jarrow1997}, customer relationship (Ma et al. \cite{MingMa}), networks \cite{Bolch}, etc. It represents the first exit time to an absorbing set of the underlying Markov process. Its distribution is usually referred to as the phase-type distribution, which was first introduced in univariate form by Neuts \cite{Neuts1975} in 1975 as generalization of Erlang distribution. It has dense property, which can approximate any distribution of positive random variables arbitrarily well, and has closure property under finite convex mixtures and convolutions. When the jumps of compound Poisson process has phase-type distribution, it results in a dense class of L\'evy processes, see Asmussen \cite{Asmussen2003}. The advantage of working under phase-type distribution is that it allows some analytically tractable results in applications. To mention some, in option pricing (Asmussen et al. \cite{Asmussen2004}), actuarial science (Albrecher and Asmussen \cite{Asmussen2010}, Rolski et al. \cite{Rolski}, Zadeh et al. \cite{Zadeh}), in survival analysis (Aalen \cite{Aalen1995}, Aalen and Gjessing \cite{Aalen2001}), in queueing theory (Chakravarthy and Neuts \cite{Chakravarthy}, Asmussen \cite{Asmussen2003}), in reliability theory (Assaf and Levikson \cite{Assaf1982}, Okamura and Dohi \cite{Okamura}).

The phase-type distribution $\overline{F}$ is expressed in terms of a Markov jump process $\{X_t\}_{t\geq 0}$ with a finite state space $\mathbb{S}=E\cup \{\Delta\}$, where for some integer $n\geq 1$, $E=\{i: i=1,...,n\}$ and $\Delta$ represent respectively the transient and absorbing states. We also refer to $\Delta$ as the $(n+1)$th element of $\mathbb{S}$, i.e., $\Delta=n+1$. The first exit time of $X$ to the absorbing state and its distribution are defined by
\begin{equation}\label{eq:DefTime}
\tau=\inf\{t\geq 0: X_t=\Delta\} \quad \textrm{and} \quad \overline{F}(t)=\mathbb{P}\{\tau > t\}.
\end{equation}

In view of credit risk applications, the state space $\mathbb{S}$ represents the possible credit classes, with $1$ being the highest (\textrm{Aaa} in Moody's rankings) and $n$ being the lowest (\textrm{C} in Moody's rankings), whilst the absorbing state $\Delta$ represents bankruptcy, \textrm{D}. The distribution $\pi_k$ represents the proportion of homogeneous bonds in the rating $k$. We refer to \cite{Jarrow1995} and \cite{Jarrow1997} and literature therein for details.

Unless stated otherwise, we denote by $\widetilde{\boldsymbol{\pi}}=(\boldsymbol{\pi},\pi_{\Delta})$ the initial probability of starting $X$ in any of the $n + 1$ phases. For simplicity, we assume that $\pi_{\Delta}=0$, so that $\mathbb{P}\{\tau>0\}=1$. The speed at which the Markov process moves along the state space $\mathbb{S}$ is described by an intensity matrix $\mathbf{Q}$. This matrix has block partition according to the process moving in the transient state $E$ and in the absorbing state $\Delta$, which admits the following block-partitioned form:
\begin{equation}\label{eq:MatQ}
\mathbf{Q} = \left(\begin{array}{cc}
  \mathbf{B} & -\mathbf{B}\mathbb{1} \\
  \mathbf{0} & 0 \\
\end{array}\right),
\end{equation}
with $\mathbb{1}=(1,...,1)^{\top}$, as the rows of the intensity matrix $\mathbf{Q}$ sums to zero. That is to say that the entry $q_{ij}$ of the matrix $\mathbf{Q}$ satisfies the following properties:
\begin{equation}\label{eq:matq}
q_{ii}\leq 0, \; \; q_{ij}\geq 0, \; \; \sum_{j\neq i} q_{ij}=-q_{ii}=q_i, \quad (i,j)\in \mathbb{S}.
\end{equation}
See Chapter II of Asmussen \cite{Asmussen2003} for more details on the Markov jump processes. Since  the states $E$ is transient and that $-\mathbf{B}\mathbb{1}$ is a non-negative vector and $\mathbb{1}^{\top}\mathbf{B}\mathbb{1}<0$, the condition (\ref{eq:matq}) implies that $\mathbf{B}$ is a negative definite matrix. See Section II4d of \cite{Asmussen2003}. The matrix $\mathbf{B}$ is known as the phase generator matrix of $\mathbf{Q}$. The absorption is certain if and only if $\mathbf{B}$ is nonsingular, see Neuts \cite{Neuts1981}.

Following Theorem 3.4 and Corollary 3.5 in \cite{Asmussen2003} and by the homogeneity of $X$, the transition probability matrix $\mathbf{P}(t)$ of $X$ over the period of time $(0,t)$ is
\begin{equation}\label{eq:transprob}
\mathbf{P}(t)= \exp(\mathbf{Q} t), \quad t\geq 0.
\end{equation}
The entry $q_{ij}$ has probabilistic interpretation: $1/(-q_{ii})$ is the expected length of time that $X$ remains in state $i\in E$, and $q_{ij}/q_{i}$ is the probability that when a transition out of state $i$ occurs, it is to state $j\in\mathbb{S}$, $j\neq i$. The representation of the distribution $\overline{F}$ is uniquely specified by $(\boldsymbol{\pi},\mathbf{B})$. We refer among others to Neuts \cite{Neuts1981} and Asmussen \cite{Asmussen2003} for details. Following \cite{Neuts1981} and Proposition 4.1 \cite{Asmussen2003},
\begin{equation}\label{eq:DistrDefTime}
 \overline{F}(t)=\boldsymbol{\pi}^{\top} e^{\mathbf{B} t} \mathbb{1} \quad \textrm{and} \quad f(t)=-\boldsymbol{\pi}^{\top} e^{\mathbf{B} t} \mathbf{B}\mathbb{1}.
 \end{equation}

The extension of (\ref{eq:DistrDefTime}) to multivariate form was proposed by Assaf et al. \cite{Assaf1984} and later by Kulkarni \cite{Kulkarni}. Following \cite{Assaf1984}, let $\Gamma_1,...,\Gamma_p$ be nonempty stochastically closed subsets of $\mathbb{S}$ such that $\cap_{k=1}^p \Gamma_k$ is a proper subset of $\mathbb{S}$. ($\Gamma_i\subset \mathbb{S}$ is said to be stochastically closed if once $X$ enters $\Gamma_i$, it never leaves.) We assume without loss of generality that $\cap_{k=1}^p \Gamma_k$ consists of only the absorbing state $\Delta$, i.e., $\cap_{k=1}^p \Gamma_k=\Delta$. Since $\Gamma_k$ is stochastically closed, necessarily $q_{ij}=0$ if $i\in\Gamma_k$ and $j\in\Gamma_k^c$.

The first exit time of $X$ to the stochastically closed set $\Gamma_k$ is defined by
\begin{equation}\label{eq:MultiPH}
\tau_k:=\inf\{t\geq 0: X_t \in \Gamma_k\}.
\end{equation}
The joint distribution $\overline{F}$ of $\{\tau_k\}$ is called the multivariate phase type distribution, see \cite{Assaf1984}. Let $t_{i_p}\geq\dots\geq t_{i_1}\geq 0$ be the ordering of $(t_1,...,t_p)\in\mathbb{R}_+^p$. Following \cite{Assaf1984},
\begin{equation}\label{eq:MPH}
\begin{split}
\overline{F}(t_1,...,t_p)=&\mathbb{P}\{\tau_1 > t_1,..., \tau_p > t_p)\\
=&\boldsymbol{\pi}^{\top} \prod_{k=1}^p \exp\big(\mathbf{B}(t_{i_k}-t_{i_{k-1}})\big)\mathbf{H}_{i_k}\mathbb{1},
\end{split}
\end{equation}
where $\mathbf{H}_{i_k}$ is $(n\times n)$ diagonal matrix whose $i$th diagonal element, for $i=1,...,n$, equals $1$ when $i\in\Gamma_{i_k}^c$ and is zero otherwise. As before, we assume that $\widetilde{\boldsymbol{\pi}}$ has zero mass on $\Delta$ and $\pi_i\neq 0$ for $i\in\bigcap_{k=1}^p \Gamma_k^c$ implying that $\mathbb{P}\{\tau_1>0,..., \tau_p>0)=1$.

\pagebreak

The multivariate distribution (\ref{eq:MPH}) has found various applications, e.g., in modeling credit default contagion (Herbertsson \cite{Herbertsson}, Bielecki et al. \cite{Bielecki}), in modeling aggregate loss distribution in insurance (Berdel and Hipp \cite{Berdel}, Asimit and Jones \cite{Asimit} and Willmot and Woo \cite{Willmot}), and in Queueing theory (Badila et al. \cite{Badila}).

Due to spatial homogeneity of the Markov process, the distributions (\ref{eq:DistrDefTime}) and (\ref{eq:MPH}) have stationary property and are unable to capture heterogeneity and available information of its past. In their empirical works, Frydman \cite{Frydman2005}, Frydman and Schuermann \cite{Frydman2008} found that bonds of the same credit rating, represented by the state space of the Markov process, can move at different speeds to other ratings. In addition to this observation, the inclusion of past credit ratings improves out-of-sample prediction of the Nelson-Aalen estimate of credit default intensity. These empirical findings suggest that the credit rating dynamics \cite{Jarrow1997} can be represented by a mixture of Markov jump processes moving at different speeds, where the mixture itself is non-Markov. However, the analyses performed in \cite{Frydman2005}, \cite{Frydman2008} were based on knowing the initial and current state of the process. Surya \cite{Surya2018} revisited the mixture model \cite{Frydman2005}, \cite{Frydman2008} and gave explicit distributional identities of the mixture, in particular in the presence of an absorbing state.

This paper attempts to extend \cite{Surya2018} by relaxing the assumptions \cite{Frydman2005}, \cite{Frydman2008} for a finite mixture of Markov jump processes with overlapping absorbing sets moving at different speeds. By doing so, we give distributional properties of the mixture process $X$ in general case and derive the joint probability distributions of the first exit times $\{\tau_k\}$ (\ref{eq:MultiPH}) of $X$, conditional on the available (either full or partial) information $\mathcal{F}_{t,i}=\mathcal{F}_{t-}\cup\{X_t=i\}$, with $\mathcal{F}_{t-}=\{X_s: 0\leq s \leq t-\}$, of the process. Using the results, we derive the joint probability distributions of $\{\tau_k\}$ conditional on the information $\mathcal{G}_t:=\mathcal{F}_{t-}\cup\{X_t\neq \Delta\}$ knowing all previous observations of the process and given that it is still ''alive'' at a given time $t\geq 0$, i.e., $\mathcal{G}_{t}=\bigcup_{i\in E}\mathcal{F}_{t,i}$. We write $\mathcal{G}_t=\mathcal{F}_{t-}$ if the only available information is the past observation $\mathcal{F}_{t-}$. Conditional on $\mathcal{F}_{t,i}$ and $\mathcal{G}_t$, we derive explicit formula for
\begin{equation}\label{eq:MPHnew}
\begin{split}
\overline{F}_{i,t}(t_1,...,t_p)=&\mathbb{P}\big\{\tau_1> t_1,..., \tau_p > t_p \big\vert \mathcal{F}_{t,i}\big\}\\
\overline{F}_t(t_1,...,t_n)=&\mathbb{P}\big\{\tau_1> t_1,..., \tau_p > t_p \big\vert \mathcal{G}_{t}\big\},
\end{split}
\end{equation}
for the mixture process $X$, with $n\geq 1$, $i\in E\subseteq\mathbb{S}$ and $0\leq t\leq \min\{t_1,...,t_p\}$. Unless the underlying Markov processes move at the same speed, we show that the initial profile of the joint distributions (\ref{eq:MPHnew}) forms a generalized mixture of (\ref{eq:MPH}). Under partial information, given the process is still alive in the long run, we give the corresponding limiting (stationary) distributions of (\ref{eq:MPHnew}) as $t\rightarrow \infty$.

From the credit risk point of view (see for e.g.  \cite{Jarrow1997}, \cite{Bielecki2002}, \cite{Bielecki}, \cite{Herbertsson}), the quantity $\overline{F}_{i,t}(t_1,...,t_p)$ describes the joint probability distribution of first exit times $\{\tau_k\}$ of $i-$rated bonds, due to cause-specific of exits (default, prepayment, calling back, debt retirement, etc), conditional on the credit rating history up to a given time $t$, whilst the function $\overline{F}_{t}(t_1,...,t_p)$ determines the joint probability distribution of the bonds' exit times $\{\tau_k\}$ across credit ratings viewed at time $t$. In the framework of competing risks (see for e.g. Pintilie \cite{Pintilie}), for the observed exit time $\tau:=\min\{\tau_1,...,\tau_p\}$ and the reason of exit $\boldsymbol{\xi}=\textrm{argmin}\{\tau_1,...,\tau_p\}$, the probability $\mathbb{P}\{t\leq \tau\leq s,\boldsymbol{\xi}=1\vert \mathcal{F}_{t,i}\}$ determines the proportion of $i-$rated bonds exiting by type $1$ from the credit portfolio within time interval $[t,s]$, whilst $\mathbb{P}\{t\leq \tau\leq s,\boldsymbol{\xi}=1\vert \mathcal{G}_{t}\}$ represents the percentage of bonds exiting by type $1$.

The organization of this paper is as follows. Section 2 discusses distributional properties of the Markov mixture process which extend the results of \cite{Frydman2008} and \cite{Surya2018}, in particular on the Bayesian update on the probability of starting the process in any state at given time $t\geq 0$. The main contributions of this paper are given in Section 3, where explicit forms of the conditional probability distributions and their Laplace transforms are presented. Some explicit examples are discussed in Section 4, in which we show that the exit times $\{\tau_k\}$ are independent under the Markov model, but not necessarily for the mixture model. Also in this section, we discuss numerical examples of the main results for bivariate distributions of birth-death mixture processes. Section 5 concludes this paper.

\begin{figure}
\begin{center}
  \begin{tikzpicture}[font=\sffamily]

        \tikzset{node style/.style={state,
                                    minimum width=1.5cm,
                                    line width=1mm,
                                    fill=gray!20!white}}

          \tikzset{My Rectangle/.style={rectangle, draw=brown, fill=yellow, thick,
    prefix after command= {\pgfextra{\tikzset{every label/.style={blue}}, label=below}}
    }
}

          \tikzset{My Rectangle2/.style={rectangle,draw=brown,  fill=yellow, thick,
    prefix after command= {\pgfextra{\tikzset{every label/.style={blue}}, label=below}}
    }
}

          \tikzset{My Rectangle3/.style={rectangle, draw=brown, fill=yellow, thick,
    prefix after command= {\pgfextra{\tikzset{every label/.style={blue}}, label=below}}
    }
}

        \node[node style] at (2, 0)     (s1)     {$J_1$};
        \node[node style] at (7, 0)     (s2)     {$J_2$};

      \node [My Rectangle3, label={initial}] at  ([shift={(-5em,0em)}]s1.west) (p0) {$\boldsymbol{\pi}$};

        \node [My Rectangle, label={regime} ] at  ([shift={(3em,-4em)}]s1.south) (g1) {$\phi=2, X^{(2)}$};
          \node [My Rectangle, label={regime} ] at  ([shift={(-3em,-4em)}]s1.south)  (g2) {$\phi=1, X^{(1)}$};

          \node [My Rectangle2] at  ([shift={(3em,-6em)}]s1.south) {$1-s_{j_1}$};
          \node [My Rectangle2] at  ([shift={(-3em,-6em)}]s1.south) {$s_{j_1}$};

         \node [My Rectangle, label={regime} ] at  ([shift={(3em,-4em)}]s2.south) (g3) {$\phi=2, X^{(2)}$};
          \node [My Rectangle, label={regime} ] at  ([shift={(-3em,-4em)}]s2.south)  (g4) {$\phi=1, X^{(1)}$};

           \node [My Rectangle2] at  ([shift={(3em,-6em)}]s2.south) {$1-s_{j_2}$};
          \node [My Rectangle2] at  ([shift={(-3em,-6em)}]s2.south) {$s_{j_2}$};

        \draw[every loop,
              auto=right,
              line width=1mm,
              >=latex,
              draw=orange,
              fill=orange]

            (s1)  edge[bend right=20, auto=left] node {$q_{12}^{(1)}/q_{12}^{(2)}$} (s2)
            (s1)  edge[loop above]                     node {$q_{11}^{(1)}/q_{11}^{(2)}$} (s1)
            (s2)  edge[loop above]                     node {$q_{22}^{(1)}/q_{22}^{(2)}$} (s2)
            (s2)  edge[bend right=20]                node {$q_{21}^{(1)}/q_{21}^{(2)}$}  (s1)

            (s1) edge node {} (g1)
            (s1) edge node {} (g2)

            (s2) edge node {} (g3)
            (s2) edge node {} (g4)

            (p0) edge node {} (s1);

 \end{tikzpicture}
 \caption{State diagram of the Markov mixture process (\ref{eq:mixture}) with $m=2$.}\label{fig:mixture}
\end{center}
\end{figure}
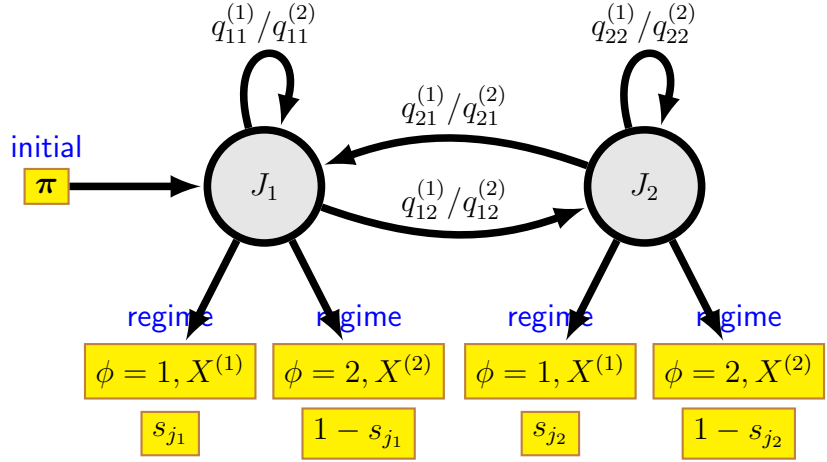


\section{Mixture of Markov jump processes}

Throughout the remaining of this paper we denote by $X=\{X_t^{(\phi)}, t\geq0\}$ the Markov mixture process, which is a continuous-time stochastic process defined as a finite mixture of Markov jump processes $X^{(k)}=\{X_t^{(k)}: t\geq 0\}$, with $k=1,\dots,m$, whose intensity matrices are given by $\{\mathbf{Q}^{(k)}\}$. We assume that the underlying Markov processes $\{X^{(k)}\}$ have right-continuous sample paths, and are defined on the state space $\mathbb{S}=\{1,\dots,n+1\}$. It is defined following \cite{Surya2018} by
\begin{equation}\label{eq:mixture}
X=
\begin{cases}
X^{(1)}, & \phi=1\\
\vdots \\
X^{(m)}, & \phi=m
\end{cases}
\end{equation}
where the variable $\phi$ represents the speed regimes, assumed to be unobservable.

\pagebreak

For a given initial state $i_0\in\mathbb{S}$, there is a separate mixing probability
\begin{equation}\label{eq:portion}
s_{i_0}^{(k)}=\mathbb{P}\{\phi=k \vert X_0=i_0\} \quad \textrm{with}\quad \sum_{k=1}^m s_{i_0}^{(k)}=1,
\end{equation}
and $0\leq s_{i_0}^{(k)} \leq 1$. The quantity $s_{i_0}^{(k)}$ has the interpretation as the proportion of population (e.g. bonds) with initial state $i_0$ evolving w.r.t to $X^{(k)}$. In general, $X^{(k)}$ and $X^{(l)}$, $k\neq l$, have different expected length of occupation time of a state $i$, i.e., $1/q_{i}^{(k)}\neq 1/q_{i}^{(l)}$, and have different probability of leaving the state $i\in E$ to state $j\in\mathbb{S}$, $j\neq i$, i.e. $q_{ij}^{(k)}/q_{i}^{(k)}\neq q_{ij}^{(l)}/q_{i}^{(l)}$. Note that we have used $q_i^{(k)}$ and $q_{ij}^{(k)}$ to denote the negative of the $i$th diagonal element and the $(i,j)$ entry of $\mathbf{Q}^{(k)}$.

Markov mixture process is a generalization of mover-stayer model, a mixture of two discrete-time Markov chains proposed by Blumen et al \cite{Blumen} in 1955 to model population heterogeneity in jobs labor market. In the mover-stayer model \cite{Blumen}, the population of workers consists of stayers (workers who always stay in the same job category, $\mathbf{Q}^{(1)}=\mathbf{0}$) and movers (workers who move to other job according to a stationary Markov chain with intensity matrix $\mathbf{Q}^{(2)}$). Estimation of the mover-stayer model was discussed in Frydman \cite{Frydman1984}. Frydman \cite{Frydman2005} generalized the model to a finite mixture of Markov chains moving with different speeds. Frydman and Schuermann \cite{Frydman2008} later on used the result for the mixture of two Markov jump processes moving with intensity matrices $\mathbf{Q}^{(1)}$ and $\mathbf{Q}^{(2)}=\boldsymbol{\Psi}\mathbf{Q}^{(1)}$, where $\boldsymbol{\Psi}$ is a diagonal matrix, to model the dynamics of firms' credit ratings. Depending on whether $0=\psi_i:=[\boldsymbol{\Psi}]_{i,i}$,  $0<\psi_i<1$, $\psi_i>1$ or $\psi_i=1$, $X^{(2)}$ never moves out of state $i$ (the mover-stayer model), moves out of state $i$ at lower rate, higher rate or at the same rate, subsequently, than that of $X^{(1)}$. If $\psi_i=1$, for all $i\in \mathbb{S}$, the mixture process $X$ reduces to a simple Markov jump process $X^{(1)}$.

Figure \ref{fig:mixture} illustrates the transition of $X$ between states $J_1$ and $J_2$. When $X$ is observed in state $J_1$, it would stay in the state for an exponential period of time with intensity $q_{j_1}^{(1)}$ or $q_{j_1}^{(2)}$ before moving to $J_2$ with probability $q_{j_1,j_2}^{(1)}/q_{j_1}^{(1)}$ or $q_{j_1,j_2}^{(2)}/q_{j_1}^{(2)}$, depending on whether it is driven by the Markov process $X^{(1)}$ or $X^{(2)}$.

\subsection{Distributional properties}

Recall that the process $X$ (\ref{eq:mixture}) repeatedly changes its speed randomly in time according to the speed rate $\mathbf{Q}^{(k)}$. The speed regime, represented by the variable $\phi$, is however not directly observable; we can not classify from which regime the observed process $X$ came from. However, it can be identified based on available information of the process. We denote by $\mathcal{F}_{t-}$ all previous information about $X$ prior to time $t\geq0$, and by $\mathcal{F}_{t,i}=\mathcal{F}_{t-}\cup\{X_t=i\}$, $i\in \mathbb{S}$. The set $\mathcal{F}_{t-}$ may contain full, partial information or maybe nothing about the past of $X$.

The likelihood of observing the past realization $\mathcal{F}_{t,j}$ of $X$ moving according to the process $X^{(k)}$ conditional on knowing its initial state $i$ is defined by
\begin{equation}\label{eq:likelihood}
\begin{split}
\mathbf{L}_{i,j}^{Q^{(k)}}(t):=\mathbb{P}\{\mathcal{F}_{t,j} \vert \phi=k, X_0=i\}= \prod_{l\in \mathbb{S}} \exp\big( -q_{l}^{(k)} T_l\big) \prod_{j\neq l, j\in \mathbb{S}} (q_{lj}^{(k)})^{N_{lj}},
\end{split}
\end{equation}
where in the both expressions we have denoted subsequently by $T_l$ and $N_{lj}$ the total time the observed process $X$ spent in state $l\in \mathbb{S}$ for $\mathcal{F}_{t,j}$, and the number of transitions from state $l$ to state $j$, with $j\neq l$, observed in the information set $\mathcal{F}_{t,j}$; whereas $q_{lj}^{(k)}$ represents the $(l,j)-$entry of the intensity matrix $\mathbf{Q}^{(k)}$.

\subsubsection{Bayesian updates of switching probability}

The Bayesian updates of switching probability $s_j(t)$ of $X$ (\ref{eq:mixture}) is defined by
\begin{equation}\label{eq:bayesianupdates}
s_j^{(k)}(t)=\mathbb{P}\{\phi=k\vert \mathcal{F}_{t,j}\}, \; \textrm{with} \; \sum_{k=1}^ms_j^{(k)}(t)=1, \; \textrm{for $j\in\mathbb{S}, \; t\geq0$}.
\end{equation}
It represents the proportion of those in state $j$ moving according to $X^{(k)}$. Note that $s_j^{(k)}(0)=s_j^{(k)}$ (\ref{eq:portion}). Denote by $\widetilde{\mathbf{S}}^{(k)}(t)$, $t\geq 0$, a diagonal matrix defined by
\begin{equation}\label{eq:St}
\widetilde{\mathbf{S}}^{(k)}(t) =
 \left(\begin{array}{cc}
 \mathbf{S}^{(k)}(t) & \mathbf{0} \\
  \mathbf{0} & s_{n+1}^{(k)}(t) \\
\end{array}\right), \;\; \textrm{s.t.} \;\; \sum_{k=1}^m \widetilde{\mathbf{S}}^{(k)}(t)=\mathbf{I},
\;\;  \textrm{for} \;\; t\geq 0,
\end{equation}
where we have denoted by $\mathbf{I}$ an $(n+1)\times (n+1)-$identity matrix, with $\mathbf{S}^{(k)}(t) =\mathrm{diag}(s_1^{(k)}(t), s_2^{(k)}(t),...,s_n^{(k)}(t))$, representing switching probability matrix of $X$.

For $t=0$, in which case $\mathcal{F}_{t,j}=\{X_0=j\}$, we write $\widetilde{\mathbf{S}}^{(k)}:= \widetilde{\mathbf{S}}^{(k)}(0)$, $\mathbf{S}^{(k)}:= \mathbf{S}^{(k)}(0)$. The element $s_j^{(k)}(t)$, $j\in \mathbb{S}$, of the intensity matrix $\widetilde{\mathbf{S}}^{(k)}(t)$ is given below.
\begin{prop}\label{prop:lem1}
Let $\widetilde{\boldsymbol{\pi}}$ be the initial probability of starting the Markov mixture process $X$ (\ref{eq:mixture}) on a finite state space $\mathbb{S}$. Define by $\mathbf{L}^{Q^{(k)}}(t)$ the likelihood matrix whose $(i,j)$ element $\mathbf{L}_{i,j}^{Q^{(k)}}(t)$ is defined in (\ref{eq:likelihood}). Then, for $j\in\mathbb{S}$ and $t\geq0$,
\begin{equation}\label{eq:likelihood2}
s_j^{(k)}(t)=\frac{\widetilde{\boldsymbol{\pi}}^{\top}\widetilde{\mathbf{S}}^{(k)} \mathbf{L}^{Q^{(k)}}(t)\mathbf{e}_j}{ \sum_{k=1}^m \widetilde{\boldsymbol{\pi}}^{\top}\widetilde{\mathbf{S}}^{(k)} \mathbf{L}^{Q^{(k)}}(t) \mathbf{e}_j}, \;\; k=1,...,m.
\end{equation}
To be more precise, depending on availability of information set $\mathcal{F}_{t-},$ we have:
\begin{enumerate}
\item[(i)] Under full information $\mathcal{F}_{t,j}=\{X_s, 0\leq s\leq t-\}\cup\{X_t=j\}$ that
\begin{equation*}
s_j^{(k)}(t)=\frac{s_{i_0}^{(k)}\mathbf{L}_{i_0,j}^{Q^{(k)}}(t)}{ \sum_{k=1}^m s_{i_0}^{(k)} \mathbf{L}_{i_0,j}^{Q^{(k)}}(t)}.
\end{equation*}
\item[(ii)] Under partial information $\mathcal{F}_{t,j}=\{X_t=j\}$, $s_j(t)$ is defined by
\begin{equation*}
s_j^{(k)}(t)=\frac{\widetilde{\boldsymbol{\pi}}^{\top} \widetilde{\mathbf{S}}^{(k)} \exp\big(\mathbf{Q}^{(k)}t\big)\mathbf{e}_j}{\sum_{k=1}^m \widetilde{\boldsymbol{\pi}}^{\top} \widetilde{\mathbf{S}}^{(k)} \exp\big(\mathbf{Q^{(k)}}t\big) \mathbf{e}_j}.
\end{equation*}
\item[(iii)] Under partial information $\mathcal{F}_{t,j}=\{X_0=i_0\}\cup\{X_t=j\}$, $s_j(t)$ is given by,
\begin{equation*}
s_j^{(k)}(t)=\frac{\mathbf{e}_{i_0}^{\top} \widetilde{\mathbf{S}}^{(k)} \exp\big(\mathbf{Q}^{(k)}t\big)\mathbf{e}_j}{ \sum_{k=1}^m \mathbf{e}_{i_0}^{\top} \widetilde{\mathbf{S}}^{(k)} \exp\big(\mathbf{Q}^{(k)}t\big)\mathbf{e}_j}.
\end{equation*}
\end{enumerate}
\end{prop}

The expression (\ref{eq:likelihood2}) generalizes the result of \cite{Frydman2008} and Lemma 3.1 in \cite{Surya2018}.

\pagebreak

Note that we have used slightly different notations for the likelihood function (\ref{eq:likelihood}) and the switching probability (\ref{eq:likelihood2}) from that of used in \cite{Frydman2008} and \cite{Surya2018}.

\medskip

\begin{proof}[Proposition \ref{prop:lem1}]
By the law of total probability and the Bayes' formula,
\begin{align*}
\mathbb{P}\{\mathcal{F}_{t,j},\phi=k\}=&\sum_{i\in\mathbb{S}}\mathbb{P}\{X_0=i\}\mathbb{P}\{\phi=k\vert X_0=i\}\mathbb{P}\{\mathcal{F}_{t,j}\vert \phi=k,X_0=i\}\\
=&\sum_{i\in\mathbb{S}}\pi_i\times s_i^{(k)} \times \mathbf{L}_{i,j}^{Q^{(k)}}(t)=\sum_{i\in\mathbb{S}} \pi_i \mathbf{e}_i^{\top}\widetilde{\mathbf{S}}^{(k)}\mathbf{L}^{Q^{(k)}}(t)\mathbf{e}_j\\
=&\widetilde{\boldsymbol{\pi}}^{\top} \widetilde{\mathbf{S}}^{(k)}\mathbf{L}^{Q^{(k)}}(t)\mathbf{e}_j.
\end{align*}
The claim in (\ref{eq:likelihood2}) is finally established on account of the Bayes' formula:
\begin{align*}
s_j^{(k)}(t)=\mathbb{P}\{\phi=k\vert\mathcal{F}_{t,j}\}=\frac{\mathbb{P}\{\mathcal{F}_{t,j},\phi=k\}}{\sum_{k=1}^m \mathbb{P}\{\mathcal{F}_{t,j},\phi=k\}}.  \exit
\end{align*}
\end{proof}

If $\{\mathbf{Q}^{(k)}\}$ have distinct eigenvalues $\{\lambda_j^{(k)}:j=1,\dots,n+1\}$, it can be proved similar to the Proposition 3.2 in \cite{Surya2018} using the Lagrange-Sylvester formula
\begin{equation}
\exp\big(\mathbf{Q}^{(k)}t\big)=\sum_{l=1}^{n+1}\exp\big(\lambda_l^{(k)}t\big)\prod_{j=1,j\neq l}^{n+1} \Big(\frac{\mathbf{Q}^{(k)}-\lambda_j^{(k)}\mathbf{I}}{\lambda_l^{(k)}-\lambda_j^{(k)}}\Big),
\end{equation}
see Theorem 2 of Apostol \cite{Apostol}, that, under partial information, the probability $s_j^{(k)}(t)\rightarrow 1$ in the long-run, as $t\rightarrow \infty$, implying that $X$ moves according to $X^{(k)}$. The result can be used to deduce the stationary distribution of (\ref{eq:MPHnew}) as $t\rightarrow \infty.$
\begin{prop}\label{prop:limitsjt}
Let $\{\mathbf{Q}^{(k)}\}$ have distinct eigenvalues $\{\lambda_j^{(k)}: j\in\mathbb{S}\},$ with $\lambda_{i_k}^{(k)}=\max\{\lambda_j^{(k)}, j\in\mathbb{S}\},$ $i_k=\textrm{argmax}_j\{\lambda_j^{(k)}\}$. Define $\overline{\lambda}=\max\{\lambda_{i_k}^{(k)}\}.$ For $j\in\mathbb{S}$,
\begin{equation}\label{eq:limitsjt}
  \lim_{t\rightarrow \infty} s_j^{(k)}(t)=
    \begin{cases}
     1, & \text{if  $\overline{\lambda}=\lambda_{i_k}^{(k)}$} \\
     \frac{\widetilde{\boldsymbol{\pi}}^{\top}\widetilde{\mathbf{S}}^{(k)}\mathcal{L}[\mathbf{Q}^{(k)}]\mathbf{e}_j}{\widetilde{\boldsymbol{\pi}}^{\top} \big( \widetilde{\mathbf{S}}^{(k)}\mathcal{L}[\mathbf{Q}^{(k)}] +  \widetilde{\mathbf{S}}^{(l)} \mathcal{L}[\mathbf{Q}^{(l)}] \big)\mathbf{e}_j},
        & \text{if  $\lambda_{i_k}^{(k)}=\lambda_{i_l}^{(l)}=\overline{\lambda}, l\neq k$} \\
       \frac{\widetilde{\boldsymbol{\pi}}^{\top}\widetilde{\mathbf{S}}^{(k)}\mathcal{L}[\mathbf{Q}^{(k)}]\mathbf{e}_j}{ \sum_{k=1}^m \widetilde{\boldsymbol{\pi}}^{\top} \widetilde{\mathbf{S}}^{(k)}\mathcal{L}[\mathbf{Q}^{(k)}]\mathbf{e}_j},
        & \text{if $\lambda_{i_k}^{(k)}=\lambda_{i_l}^{(l)}=\overline{\lambda},\forall l\neq k$,}
    \end{cases}
\end{equation}
where $\mathcal{L}[\mathbf{Q}^{(k)}]=\prod\limits_{j=1,j\neq i_k}^{n+1} \Big(\frac{\mathbf{Q}^{(k)}-\lambda_j^{(k)}\mathbf{I}}{\lambda_{i_k}^{(k)}-\lambda_j^{(k)}}\Big)$ is the Lagrange interpolation coefficient.
\end{prop}

It is clear following the above that when the intensity matrices $\{\mathbf{Q}^{(k)}\}$ take the form of (\ref{eq:intensity}), (\ref{eq:limitsjt}) reduces to the results of Proposition 3.2 of \cite{Surya2018}.

In the section below we derive the Bayesian updates $\boldsymbol{\pi}(t)$ on the probability of starting $X$ at a given time $t\geq 0$ and available information of the process.

\subsubsection{Bayesian updates of probability distribution $\boldsymbol{\pi}$}
The following proposition and its corollary provide Bayesian updates $\widetilde{\pi}_j(t)$ on finding $X$ in any state $j\in\mathbb{S}$ at a given time $t\geq 0$ based on all previous observations $\mathcal{F}_{t-}$ of the process and knowing that it is still ''alive'' at time $t$.

\pagebreak

\begin{prop}\label{prop:Prop1}
Let $\mathcal{G}_t=\mathcal{F}_{t-}$. Define $\pi_j(t)=\mathbb{P}\{X_t=j\vert \mathcal{G}_{t}\}$ for $j\in \mathbb{S}, t\geq 0$.
\begin{align}\label{eq:piatt}
\widetilde{\pi}_j(t)=\frac{\sum_{k=1}^m \widetilde{\boldsymbol{\pi}}^{\top} \widetilde{\mathbf{S}}^{(k)} \mathbf{L}^{Q^{(k)}}(t)\mathbf{e}_j}{\sum_{k=1}^m \widetilde{\boldsymbol{\pi}}^{\top} \widetilde{\mathbf{S}}^{(k)} \mathbf{L}^{Q^{(k)}}(t)\mathbb{1}}.
\end{align}
\begin{enumerate}
\item[(i)] Given all previous observations $\mathcal{F}_{t-}=\{X_s, 0\leq s\leq t-\}$, we have
\begin{equation*}
\widetilde{\pi}_j(t)=\frac{\sum_{k=1}^m s_{i_0}^{(k)} \mathbf{L}_{i_0,j}^{Q^{(k)}}(t)}{\sum_{j\in\mathbb{S}}\sum_{k =1}^m s_{i_0}^{(k)} \mathbf{L}_{i_0,j}^{Q^{(k)}}(t)}.
\end{equation*}
\item[(ii)] If $\mathcal{F}_{t-}=\emptyset$, it follows from (\ref{eq:likelihood}) that $\mathbf{L}^{Q^{(k)}}(t) =\exp\big(\mathbf{Q}^{(k)}t\big)$. Then,
\begin{align*}
\widetilde{\pi}_j(t)=\sum_{k=1}^m \widetilde{\boldsymbol{\pi}}^{\top} \widetilde{\mathbf{S}}^{(k)}\exp\big(\mathbf{Q}^{(k)}t\big) \mathbf{e}_j.
\end{align*}
Furthermore, let $\{\mathbf{Q}^{(k)}\}$ have the representation (\ref{eq:intensity}). Then,
\begin{equation}\label{eq:piupdatet1a}
\widetilde{\pi}_j(t)=
\begin{cases}
\sum_{k=1}^m \boldsymbol{\pi}^{\top} \mathbf{S}^{(k)} \exp\big(\mathbf{B}^{(k)}t\big)\mathbf{e}_j, & \textrm{for $j\in E$}\\[8pt]
\sum_{k=1}^m \boldsymbol{\pi}^{\top} \mathbf{S}^{(k)}\big[\mathbf{I}-\exp\big(\mathbf{B}^{(k)}t\big)\big]\mathbb{1},  & \textrm{for $j=\Delta$}.
\end{cases}
\end{equation}

\item[(iii)]  If $\mathcal{F}_{t-}=\{X_0=i_0\}$, it follows from the above that $\pi_j(t)$ is given by
\begin{align*}
\widetilde{\pi}_j(t)=\sum_{k=1}^m \mathbf{e}_{i_0}^{\top} \widetilde{\mathbf{S}}^{(k)}\exp\big(\mathbf{Q}^{(k)}t\big) \mathbf{e}_j.
\end{align*}
Moreover, if $\{\mathbf{Q}^{(k)}\}$ have the representation (\ref{eq:intensity}), then we have
\begin{equation}\label{eq:piupdatet1b}
\widetilde{\pi}_j(t)=
\begin{cases}
\sum_{k=1}^m\mathbf{e}_{i_0}^{\top} \mathbf{S}^{(k)} \exp\big(\mathbf{B}^{(k)}t\big)\mathbf{e}_j, & \textrm{for $j\in E$}\\[8pt]
\sum_{k=1}^m\mathbf{e}_{i_0}^{\top}\mathbf{S}^{(k)}\big[\mathbf{I}-\exp\big(\mathbf{B}^{(k)}t\big)\big]\mathbb{1},  & \textrm{for $j=\Delta$}.
\end{cases}
\end{equation}
\end{enumerate}
Notice that $0<\widetilde{\pi}_E(t)<1$, $\widetilde{\pi}_{\Delta}(t)>0$, $\sum_{j\in \mathbb{S}}\widetilde{\pi}_j(t)=1$ for $t\geq 0$, and $\widetilde{\boldsymbol{\pi}}=\widetilde{\boldsymbol{\pi}}(0)$.
\end{prop}
\begin{proof}
The proof follows from applying the law of total probability and the Bayes' formula for conditional probability. By applying the latter, we have that
\begin{align*}
\mathbb{P}\{\mathcal{F}_{t,j},\phi=k,X_0=i\}=&\mathbb{P}\{X_0=i\}\mathbb{P}\{\phi=k\vert X_0=i\}\mathbb{P}\{\mathcal{F}_{t,j}\vert \phi=1,X_0=i\}\\
=&\pi_i\times s_i^{(k)}\times \mathbf{L}_{i,j}^{Q^{(k)}}(t) =\pi_i \mathbf{e}_i^{\top} \widetilde{\mathbf{S}}^{(k)}\mathbf{L}^{Q^{(k)}}(t) \mathbf{e}_j.
\end{align*}
Therefore, we have by the above and applying the law of total probability that
\begin{align*}
\mathbb{P}\{\mathcal{F}_{t,j}\}=&\sum_{i\in \mathbb{S}}\sum_{k=1}^m \mathbb{P}\{\mathcal{F}_{t,j}, \phi=k,X_0=i\}
= \sum_{k=1}^m \widetilde{\boldsymbol{\pi}}^{\top} \widetilde{\mathbf{S}}^{(k)}\mathbf{L}^{Q^{(k)}}(t)\mathbf{e}_j.
\end{align*}
The result (\ref{eq:piatt}) is established by the Bayes' rule and the law of total probability,
\begin{align*}
\widetilde{\pi}_j(t)=\mathbb{P}\{X_t=j\vert \mathcal{G}_{t}\}=\frac{\mathbb{P}\{\mathcal{F}_{t,j}\}}{\sum_{k\in \mathbb{S}}\mathbb{P}\{\mathcal{F}_{t,k}\}},
\end{align*}
while $(ii)$ and $(iii)$ follow on account of $e^{\mathbf{Q}^{(k)}t}\mathbb{1}=\mathbb{1}$, $\widetilde{\boldsymbol{\pi}}^{\top}(t)\mathbb{1}=1$ and (\ref{eq:blockpartisi}). \exit
\end{proof}

\pagebreak

\begin{cor}\label{cor:Cor1}
Suppose that the process is still alive at time $t\geq 0$. Then,
\begin{enumerate}
\item[(i)] Under full information $\mathcal{G}_{t}=\{X_s, 0\leq s\leq t-\}\cup \{X_t\neq \Delta\}$, we have
\begin{equation*}
\pi_j(t)=\frac{\sum_{k=1}^m s_{i_0}^{(k)} \mathbf{L}_{i_0,j}^{Q^{(k)}}(t)}{\sum_{j\in E}\sum_{k=1}^m s_{i_0}^{(k)} \mathbf{L}_{i_0,j}^{Q^{(k)}}(t)}, \quad \textrm{for $j\in E$}.
\end{equation*}
\item[(ii)] If $\mathcal{G}_{t}=\{X_t\neq \Delta\}$, it follows from (\ref{eq:likelihood}) and the matrix partition (\ref{eq:blockpartisi}),
\begin{equation}\label{eq:piupdatet2a}
\pi_j(t)=
\frac{ \sum_{k=1}^m \boldsymbol{\pi}^{\top}\mathbf{S}^{(k)} e^{\mathbf{B}^{(k)}t} \mathbf{e}_j}{ \sum_{k=1}^m \boldsymbol{\pi}^{\top} \mathbf{S}^{(k)} e^{\mathbf{B}^{(k)}t}  \mathbb{1}}, \quad \textrm{for $j\in E$}.
\end{equation}

\item[(iii)]  If $\mathcal{G}_{t}=\{X_0=i_0\}\cup\{X_t\neq \Delta\}$, following the above, $\pi_j(t)$ is given by
\begin{equation}\label{eq:piupdatet2b}
\pi_j(t)=\frac{\sum_{k=1}^m \mathbf{e}_{i_0}^{\top} \mathbf{S}^{(k)} e^{\mathbf{B}^{(k)}t} \mathbf{e}_j}{ \sum_{k=1}^m \mathbf{e}_{i_0}^{\top} \mathbf{S}^{(k)} e^{\mathbf{B}^{(k)}t}  \mathbb{1}}, \quad \textrm{for $j\in E$}.
\end{equation}
\end{enumerate}
\end{cor}
It follows that $0<\pi_E(t)<1$, $\pi_{\Delta}(t)=0$, $\sum_{j\in E}\pi_j(t)=1$ for $t\geq 0$, and $\widetilde{\boldsymbol{\pi}}=\widetilde{\boldsymbol{\pi}}(0)$.

\medskip

Notice that the Bayesian update $\pi_j(t)$ (\ref{eq:piupdatet2a}) and (\ref{eq:piupdatet2b}) form the normalization of the probability $\pi_j(t)$ (\ref{eq:piupdatet1a}) and (\ref{eq:piupdatet1b}), respectively, as such that $\pi_{\Delta}(t)=0$. The results of Proposition \ref{prop:Prop1} and Corollary \ref{cor:Cor1} give additional features to the distributional properties of the mixture process \cite{Surya2018} and \cite{Frydman2008}.

Below we give the value of $\pi_j(t)$ as $t\rightarrow \infty$ under partial information. The result can be used to deduce the stationary distribution of (\ref{eq:MPHnew}) as $t\rightarrow \infty$.
\begin{prop}\label{prop:limitpijt}
Let $\{\mathbf{B}^{(k)}\}$ have distinct eigenvalues $\{\lambda_j^{(k)}: j\in E\},$ with $\lambda_{i_k}^{(k)}=\max\{\lambda_j^{(k)}, j\in E\},$ $i_k=\textrm{argmax}_j\{\lambda_j^{(k)}\}$. Define $\overline{\lambda}=\max\{\lambda_{i_k}^{(k)}\}.$ For $j\in E$,
\begin{equation}\label{eq:limitpijt}
  \lim_{t\rightarrow \infty} \pi_j(t)=
    \begin{cases}
     \frac{\boldsymbol{\pi}^{\top}\mathbf{S}^{(k)}\mathcal{L}[\mathbf{B}^{(k)}]\mathbf{e}_j }{ \boldsymbol{\pi}^{\top}\mathbf{S}^{(k)}\mathcal{L}[\mathbf{B}^{(k)}]\mathbb{1}  }, & \text{if  $\lambda_{i_k}^{(k)}=\overline{\lambda}$} \\[8pt]
      \frac{\boldsymbol{\pi}^{\top}\big(\mathbf{S}^{(k)}\mathcal{L}[\mathbf{B}^{(k)}]+ \mathbf{S}^{(l)}\mathcal{L}[\mathbf{B}^{(l)}]\big)\mathbf{e}_j}{\boldsymbol{\pi}^{\top}\big(\mathbf{S}^{(k)}\mathcal{L}[\mathbf{B}^{(k)}]+\mathbf{S}^{(l)}\mathcal{L}[\mathbf{B}^{(l)}]\big)\mathbb{1}},
        & \text{if  $\lambda_{i_k}^{(k)}=\lambda_{i_l}^{(l)}=\overline{\lambda}, l\neq k$}\\[8pt]
              \frac{\sum_{k=1}^m \boldsymbol{\pi}^{\top} \mathbf{S}^{(k)}\mathcal{L}[\mathbf{B}^{(k)}]\mathbf{e}_j}{\sum_{k=1}^m \boldsymbol{\pi}^{\top} \mathbf{S}^{(k)}\mathcal{L}[\mathbf{B}^{(k)}]\mathbb{1}},
        & \text{if  $\lambda_{i_k}^{(k)}=\lambda_{i_l}^{(l)}=\overline{\lambda}, \forall l\neq k$},
    \end{cases}
\end{equation}
where $\mathcal{L}[\mathbf{B}^{(k)}]=\prod\limits_{j=1,j\neq i_k}^{n} \Big(\frac{\mathbf{B}^{(k)}-\lambda_j^{(k)}\mathbf{I}}{\lambda_{i_k}^{(k)}-\lambda_j^{(k)}}\Big)$ is the Lagrange interpolation coefficient.
\end{prop}

In contrary to (\ref{eq:piupdatet1a}) and (\ref{eq:piupdatet1b}), we see from the above proposition that given the process still alive in the long run, the stationary distribution $\pi_j(\infty):=\lim_{t\rightarrow \infty} \pi_j(t)$ of $X$ does not have zero mass on the state $E$ with $\sum_{j\in E} \pi_j(\infty)=1.$

\subsubsection{$\mathcal{F}_t-$conditional transition probability matrix}
The main feature of the process $X$ (\ref{eq:mixture}) is that unlike its component $X^{(0)}$ and $X^{(1)}$, $X$ does not have the Markov property; future development of its state depends on its past information. The following theorem summarizes this property.

\pagebreak

\begin{theo}\label{theo:theo1}
For any $s\geq t\geq 0$, the conditional transition probability matrix $[\mathbf{P}(t,s)]_{i,j}:=\mathbb{P}\{X_s=j\vert\mathcal{F}_{t,i}\}$, $i,j\in\mathbb{S}$, of the mixture process $X$ (\ref{eq:mixture}) is given by
\begin{equation}\label{eq:transM}
\mathbf{P}(t,s)= \sum_{k=1}^m \widetilde{\mathbf{S}}^{(k)}(t)e^{\mathbf{Q}^{(k)}(s-t)} \;\; \mathrm{with} \;\; \sum_{k=1}^m \widetilde{\mathbf{S}}^{(k)}(t) =\mathbf{I}.
\end{equation}
\end{theo}
Theorem \ref{theo:theo1} generalizes the result of a lemma in \cite{Frydman2008} and Theorem 3.4 in \cite{Surya2018}.

\medskip

\begin{proof}
Similar to the proof of Theorem 3.4 in \cite{Surya2018}, (\ref{eq:transM}) is established by applying the law of total probability and Bayes' rule for conditional probability:
\begin{align*}
\mathbf{P}_{i,j}(t,s)=&\mathbb{P}\{X_s=j \vert X_t=i,\mathcal{F}_{t-}\}=\sum_{k=1}^m \mathbb{P}\big\{X_s=j, \phi=k \vert X_t=i, \mathcal{F}_{t-}\big\}\\
=& \sum_{k=1}^m \mathbb{P}\{\phi=k\vert X_t=i,\mathcal{F}_{t-}\} \mathbb{P}\{X_s=j \vert \phi=k, X_t=i,\mathcal{F}_{t-}\}\\
=&\sum_{k=1}^m s_i^{(k)}(t) \mathbf{P}_{i,j}^{Q^{(k)}}(t,s) =\sum_{k=1}^m \mathbf{e}_i^{\top} \widetilde{\mathbf{S}}^{(k)}(t)  \mathbf{P}^{Q^{(k)}}(t,s)  \mathbf{e}_j,
\end{align*}
where on the second last equality we used the fact that $X^{(k)}$ is Markovian. \exit
\end{proof}

\medskip

It is clear from (\ref{eq:transM}) that, unless the underlying Markov process $X^{(k)}$ moves at the same speed $\mathbf{Q}$, i.e., $\mathbf{Q}^{(k)}=\mathbf{Q}$ for $k=1,\dots,m$, $X$ does not inherit the Markov property of $X^{(k)}$, i.e., future development of $X$ is determined by its past information $\mathcal{F}_{t,i}$ through its likelihood functions (\ref{eq:likelihood}). To be more precise, when $\mathbf{Q}^{(k)}=\mathbf{Q}$, it follows from the transition probability matrix (\ref{eq:transM}) that $\mathbf{P}(t,s)=e^{\mathbf{Q}(s-t)},$ by which $X$ reduces to a simple Markov jump process.

\section{Probability distributions of first exit times}\label{sec:mainsection}

This section presents the main results of this paper on the joint probability distributions of the first exit times $\{\tau_k\}$ (\ref{eq:MultiPH}) of the Markov mixture process $X$ (\ref{eq:mixture}), conditional on the available information sets $\mathcal{F}_{t,i}$ and $\mathcal{G}_{t}$. We first derive conditional univariate distribution (\ref{eq:DefTime}). To motivate the main results on the conditional multivariate distributions (\ref{eq:MPHnew}), we consider the bivariate case in some details. Throughout the remaining, we define intensity matrix $\mathbf{Q}^{(k)}$ by
\begin{equation}\label{eq:intensity}
\mathbf{Q}^{(k)} = \left(\begin{array}{cc}
  \mathbf{B}^{(k)} & -\mathbf{B}^{(k)}\mathbb{1} \\
  \mathbf{0} & 0 \\
\end{array}\right).
\end{equation}

The following results on block partition of the transition probability matrix $\mathbf{P}(t,s)$ (\ref{eq:transM}) and exponential matrix $e^{\mathbf{Q}^{(k)}t}$ will be used to derive the conditional probability distributions (\ref{eq:MPHnew}). We refer to Proposition 3.7 in \cite{Surya2018} for details.
\begin{lem}
Let the phase generator matrix $\mathbf{B}^{(k)}$ be nonsingular. Then,
\begin{align}\label{eq:blockpartisi}
e^{\mathbf{Q}^{(k)}t}=
\left(\begin{array}{cc}
e^{\mathbf{B}^{(k)}t} & \mathbb{1}-e^{\mathbf{B}^{(k)}t}\mathbb{1}\\
  \mathbf{0} & 1 \\
\end{array}\right).
\end{align}
\end{lem}

\begin{prop}
The transition probability matrix (\ref{eq:transM}) has block partition:
\begin{equation}\label{eq:blockPts}
\mathbf{P}(t,s)= \left(\begin{array}{cc}
 \sum_{k=1}^m \mathbf{S}^{(k)}(t)e^{\mathbf{B}^{(k)}(s-t)} &  \sum_{k=1}^m \mathbf{S}^{(k)}(t)\big(\mathbf{I}-e^{\mathbf{B}^{(k)}(s-t)}\big)\mathbb{1}\\
  \mathbf{0} & 1 \\
\end{array}\right).
\end{equation}
\end{prop}

\subsection{Conditional univariate distributions}

This section presents explicit identity for the probability distribution $\overline{F}_t(s)=\mathbb{P}\{\tau > s \vert \mathcal{G}_{t}\}$, $s\geq t\geq 0$, of the first exit time $\tau$ (\ref{eq:DefTime}) given the information $\mathcal{G}_{t}$.
\begin{lem}
The $\mathcal{G}_{t}-$conditional distribution $\overline{F}_t(s)$ is given for $s\geq t\geq 0$ by
\begin{equation}\label{eq:PH1}
\overline{F}_t(s) = \sum_{k=1}^m \boldsymbol{\pi}^{\top}(t) \mathbf{S}^{(k)} (t) e^{\mathbf{B}^{(k)}(s-t)}\mathbb{1}, \;\; \mathrm{with} \;\; \sum_{k=1}^m \mathbf{S}^{(k)}(t)=\mathbf{I}.
\end{equation}
\end{lem}

\begin{proof}
Without loss of generality, let $\mathcal{G}_t=\mathcal{F}_{t-}$. As $\tau$ is the first exit time of $X$ to the absorbing state $\Delta$, by applying the law of total probability we have
\begin{equation}\label{eq:pers1}
\begin{split}
\overline{F}_t(s):=\mathbb{P}\{\tau>s \big\vert \mathcal{G}_t\}=&\sum_{j\in E}\sum_{i\in\mathbb{S}} \mathbb{P}\big\{X_s=j, X_t= i \big\vert \mathcal{F}_{t-}\big\}.
\end{split}
\end{equation}
Again, by the law of total probability and Bayes' formula, we obtain
\begin{align*}
\mathbb{P}\big\{X_s=j, X_t= i \big\vert \mathcal{F}_{t-}\big\}=&
\mathbb{P}\big\{X_t=i \big\vert \mathcal{F}_{t-}\big\}\mathbb{P}\big\{X_s=j \big\vert X_t=i,\mathcal{F}_{t-}\big\}\\
=&\pi_i(t) \mathbf{P}_{i,j}(t,s) \\
=&\pi_i(t) \mathbf{e}_i^{\top} \mathbf{P}(t,s) \mathbf{e}_j\mathbf{e}_j^{\top}\mathbb{1}.
\end{align*}
Starting from equation (\ref{eq:pers1}), we have following the above expression that
\begin{align*}
\overline{F}_t(s)=& \widetilde{\boldsymbol{\pi}}^{\top}(t) \mathbf{P}(t,s)\Big[\sum_{j\in E} \mathbf{e}_j\mathbf{e}_j^{\top}\Big]\mathbb{1}.
\end{align*}
We arrive at the probability distribution (\ref{eq:PH1}) on account of $\sum\limits_{j\in E} \mathbf{e}_j\mathbf{e}_j^{\top}=\textrm{diag}(\mathbf{I},0)$ and the block partition (\ref{eq:blockPts}) of the transition probability matrix $\mathbf{P}(t,s)$.  \exit
\end{proof}

\medskip

Applying similar steps of derivation to the proof of (\ref{eq:PH1}), one can show that
\begin{equation}\label{eq:PH0}
\overline{F}_{i,t}(s):=\mathbb{P}\{\tau>s \big\vert \mathcal{F}_{t,i}\}= \sum_{k=1}^m \mathbf{e}_i^{\top} \mathbf{S}^{(k)} (t) e^{\mathbf{B}^{(k)}(s-t)}\mathbb{1}.
\end{equation}

\begin{lem}
Following the two identities (\ref{eq:PH1}) and (\ref{eq:PH0}), we deduce that
\begin{align}\label{eq:relation}
\mathbb{P}\{\tau > s \vert \mathcal{G}_{t}\}=\sum_{i\in E} \pi_i(t)\mathbb{P}\{\tau > s \vert \mathcal{F}_{t,i}\}.
\end{align}
\end{lem}

Note that the measure $-d\overline{F}_t(s)$ has probability mass $f_t(t)=1-\boldsymbol{\pi}^{\top}(t)\mathbb{1}$ at the point $s=t$ when conditioning on $\mathcal{G}_{t}=\mathcal{F}_{t-}$, and no mass given $\mathcal{G}_t=\mathcal{F}_{t-}\cup \{X_t\neq \Delta\}$. Given that $\boldsymbol{\pi}^{\top}\mathbb{1}=1$, it has zero mass at $t=0$. It is absolutely continuous w.r.t Lebesgue measure $ds$ with density $f_t(s)$ on $\{s>t\}$. Following (\ref{eq:PH1}), the density function $f_t(s)$, its Laplace transform and $n$th moment are given below.

\pagebreak

\begin{theo}
The $\mathcal{G}_{t}-$conditional density function $f_t(s)$ is given for $s>t$ by
\begin{equation}\label{eq:PHD1}
f_t(s) = - \sum_{k=1}^m \boldsymbol{\pi}^{\top}(t) \mathbf{S}^{(k)}(t) e^{\mathbf{B}^{(k)}(s-t)}\mathbf{B}^{(k)}\mathbb{1}, \;\; \mathrm{with} \;\; \sum_{k=1}^m \mathbf{S}^{(k)}(t)=\mathbf{I}.
\end{equation}
\begin{enumerate}
\item[(i)] The Laplace transform $\Psi_t(\lambda)=\int_0^{\infty} e^{-\lambda u} f_t(t+u) du$ is given by
\begin{equation*}
\Psi_t(\lambda)=- \sum_{k=1}^m \boldsymbol{\pi}^{\top}(t) \mathbf{S}^{(k)}(t)\big(\lambda\mathbf{I}-\mathbf{B}^{(k)}\big)^{-1}\mathbf{B}^{(k)}\mathbb{1} + f_t(t).
\end{equation*}
\item[(ii)] The $\mathcal{G}_{t}-$conditional $n$th moment, \textrm{for $n=0,1,...$}, of $\tau$ is given by
\begin{equation*}
\mathbb{E}\{\tau^n \vert \mathcal{G}_{t}\}=(-1)^n n! \sum_{k=1}^m \boldsymbol{\pi}^{\top}(t) \mathbf{S}^{(k)}(t) \big[\mathbf{B}^{(k)}\big]^{-n}\mathbb{1}.
\end{equation*}
\end{enumerate}
\end{theo}

\noindent Setting $\mathbf{B}^{(k)}=\mathbf{B}$ in (\ref{eq:PHD1}), in which case $X$ never changes the speed, the above results coincide with that of given in \cite{Neuts1975} and Proposition 4.1 in \cite{Asmussen2003} for $t=0$.

The following theorem summarizes the dense and closure properties under finite convex mixtures and convolutions of $\overline{F}_t(s)$ (\ref{eq:PH1}). They can be established using matrix analytic approach \cite{Neuts1981}. See for e.g. Theorems 4.12 and 4.13 in \cite{Surya2018}.

\begin{theo}
The phase-type distribution $\overline{F}_t(s)$ (\ref{eq:PH1}) is closed under finite convex mixtures and convolutions, and forms a dense class of distributions on $\mathbb{R}_+$.
\end{theo}

\subsection{Conditional bivariate distributions}

As in the univariate case, we consider the mixture process $X$ (\ref{eq:mixture}) on the finite state space $\mathbb{S}=E\cup\{\Delta\}$. Following \cite{Assaf1984}, let $\boldsymbol{\Gamma}_1$ and $\boldsymbol{\Gamma}_2$ be two nonempty stochastically closed subsets of $\mathbb{S}$ such that $\boldsymbol{\Gamma}_1\cap \boldsymbol{\Gamma}_2$ is a proper subset of $\mathbb{S}$. We assume without loss of generality that $\boldsymbol{\Gamma}_1\cap \boldsymbol{\Gamma}_2=\Delta$ and the absorption into $\Delta$ is certain, i.e., the generator matrices $\{\mathbf{B}^{(k)}\}$ need to be nonsingular. As $\boldsymbol{\Gamma}_l$, $l=1,2$, are stochastically closed sets, necessarily we have $[\mathbf{Q}^{(k)}]_{i,j}=0$ if $i\in\boldsymbol{\Gamma}_l$ and $j\in \boldsymbol{\Gamma}_l^c$.

We denote by $\widetilde{\boldsymbol{\pi}}$ the initial probability vector on $\mathbb{S}$ such that $\pi_{\Delta}=0$. We shall assume that $\boldsymbol{\pi}_i\neq 0$ if $i\in \boldsymbol{\Gamma}_1^c\cap \boldsymbol{\Gamma}_2^c$ implying $\mathbb{P}\{\tau_1>0, \tau_2>0\}=1$. As before, $\mathcal{F}_{t,i}=\mathcal{F}_{t-}\cup\{X_t=i\}$ defines all previous and current information of $X$.

\subsubsection{Conditional joint survival function of $\tau_1$ and $\tau_2$}

The joint distribution of $\tau_k$ (\ref{eq:MPHnew}), for $k=1,2$, are given by the following.
\begin{lem}\label{lem:lemjointCDF}
The identity for $\mathcal{F}_{t,i}-$conditional joint distribution $\overline{F}_{i,t}(t_1,t_2)=\mathbb{P}\{\tau_1>t_1, \tau_2>t_2 \vert \mathcal{F}_{t,i}\}$ of $\tau_1$ and $\tau_2$ is given for $t_1,t_2\geq t\geq 0$ and $i\in E$ by
\begin{align*}
\overline{F}_{i,t}(t_1,t_2)=
\begin{cases}
 \sum_{k=1}^m\mathbf{e}_i^{\top} \mathbf{S}^{(k)}(t)e^{\mathbf{B}^{(k)}(t_{2}-t)}\mathbf{H}_2e^{\mathbf{B}^{(k)}(t_1-t_2)}\mathbf{H}_1\mathbb{1}, \;\; \textrm{if $t_1\geq t_2\geq t \geq 0$} \\[8pt]
\sum_{k=1}^m \mathbf{e}_i^{\top} \mathbf{S}^{(k)}(t)e^{\mathbf{B}^{(k)}(t_{1}-t)}\mathbf{H}_1e^{\mathbf{B}^{(k)}(t_2-t_1)}\mathbf{H}_2\mathbb{1}, \; \textrm{if $t_2\geq t_1\geq t \geq 0$}.
\end{cases}
\end{align*}
with $\sum_{k=1}^m \mathbf{S}^{(k)}(t)=\mathbf{I}.$
Note that we have used $\mathbf{H}_k$ to denote a $(n\times n)-$diagonal matrix whose $i$th diagonal element for $i\in E$ equals $1$ if $i\in\Gamma_k^c$ and is $0$ otherwise.
\end{lem}

\pagebreak

\begin{proof}
\noindent To begin with, let $(t_{i_1},t_{i_2})$, with $t_{i_2}\geq t_{i_1}$ be the ordering of $(t_1,t_2)$, with $t_{i_1}\geq t_{i_0}=t$. Since $\tau_{i_k}$, $k=1,2$, is the first exit time of $X$ (\ref{eq:mixture}) to $\Gamma_{i_k}$,
\begin{align}
\mathbb{P}\{\tau_1>t_1,\tau_2>t_2 \big\vert \mathcal{F}_{t_{i_0},i}\}=&\mathbb{P}\{\tau_{i_1}>t_{i_1}, \tau_{i_2}>t_{i_2} \big\vert \mathcal{F}_{t_{i_0},i}\} \nonumber\\
=& \mathbb{P}\{X_{t_{i_1}}\in\boldsymbol{\Gamma}_{i_1}^c, X_{t_{i_2}}\in\boldsymbol{\Gamma}_{i_2}^c \big\vert \mathcal{F}_{t_{i_0},i} \nonumber\}\\
=&\sum_{J_{i_1}\in\Gamma_{i_1}^c}\sum_{J_{i_2}\in\Gamma_{i_2}^c} \mathbb{P}\{X_{t_{i_1}}=J_{i_1}, X_{t_{i_2}}=J_{i_2} \big\vert  \mathcal{F}_{t_{i_0},i}\}. \label{eq:turunan1}
\end{align}
The probability on the r.h.s of the last equality can be worked out as follows.
\begin{align*}
&\mathbb{P}\big\{X_{t_{i_1}}=J_{i_1}, X_{t_{i_2}}=J_{i_2} \big\vert  \mathcal{F}_{t_{i_0},i}\big\}=\sum_{k=1}^m \mathbb{P}\big\{X_{t_{i_1}}=J_{i_1}, X_{t_{i_2}}=J_{i_2}, \phi=k \big\vert  \mathcal{F}_{t_{i_0},i}\big\}\\
&\hspace{2cm}= \sum_{k=1}^m \mathbb{P}\big\{X_{t_{i_0}}=J_{i_0}\vert \mathcal{F}_{t_{i_0},i}\big\}\mathbb{P}\big\{\phi=k \big\vert X_{t_{i_0}}=J_{i_0},\mathcal{F}_{t_{i_0},i}\big\} \\
&\hspace{3.5cm}\times \mathbb{P}\big\{X_{t_{i_1}}=J_{i_1} \big\vert \phi=k, X_{t_{i_0}}=J_{i_0},\mathcal{F}_{t_{i_0},i}\big\}\\
&\hspace{4.5cm}\times \mathbb{P}\big\{X_{t_{i_2}}=J_{i_2} \big\vert \phi=k, X_{t_{i_1}}=J_{i_1}, X_{t_{i_0}}=J_{i_0},\mathcal{F}_{t_{i_0},i}\big\}\\[6pt]
&\hspace{2cm}= \sum_{k=1}^m \mathbb{1}_{\{J_{i_0}=i\}} s_{J_{i_0}}^{(k)}(t_{i_0}) \mathbf{P}_{J_{i_0},J_{i_1}}^{Q^{(k)}}(t_{i_0},t_{i_1})\mathbf{P}_{J_{i_1},J_{i_2}}^{Q^{(k)}}(t_{i_1},t_{i_2}) \\
&\hspace{2cm}=\sum_{k=1}^m \mathbf{e}_i^{\top}\widetilde{\mathbf{S}}^{(k)}(t) e^{\mathbf{Q}^{(k)}(t_{i_1}-t_{i_0})}\mathbf{e}_{J_{i_1}} \mathbf{e}_{J_{i_1}}^{\top} e^{\mathbf{Q}^{(k)}(t_{i_2}-t_{i_1})} \mathbf{e}_{J_{i_2}}  \mathbf{e}_{J_{i_2}}^{\top}\mathbb{1}.
\end{align*}
Note that we have applied the law of total probability and Bayes' rule for conditional probability in the above equality. Recall that $\mathbb{P}\big\{X_{t_{i_0}}=J_{i_0}\vert \mathcal{F}_{t_{i_0},i}\big\}=1$ iff $J_{i_0}=i$ and zero otherwise. Therefore, starting from eqn. (\ref{eq:turunan1}), we have
\begin{align*}
&\mathbb{P}\{\tau_{i_1}>t_{i_1}, \tau_{i_2}>t_{i_2} \big\vert \mathcal{F}_{t_{i_0},i}\}=\sum_{J_{i_1}\in\Gamma_{i_1}^c}\sum_{J_{i_2}\in\Gamma_{i_2}^c} \mathbb{P}\{X_{t_{i_1}}=J_{i_1}, X_{t_{i_2}}=J_{i_2} \big\vert  \mathcal{F}_{t_{i_0},i}\}\\
&\hspace{0.5cm}= \sum_{k=1}^m \mathbf{e}_i^{\top}\widetilde{\mathbf{S}}^{(k)}(t) e^{\mathbf{Q}^{(k)}(t_{i_1}-t_{i_0})}\Big(\sum_{J_{i_1}\in\Gamma_{i_1}^c} \mathbf{e}_{J_{i_1}}\mathbf{e}_{J_{i_1}}^{\top} \Big) e^{\mathbf{Q}^{(k)}(t_{i_2}-t_{i_1})} \Big(\sum_{J_{i_2}\in\Gamma_{i_2}^c}  \mathbf{e}_{J_{i_2}}  \mathbf{e}_{J_{i_2}}^{\top}\Big)\mathbb{1},
\end{align*}
leading to $\overline{F}_{i,t}(t_1,t_2)$ on account of $\mathbf{H}_{i_k}=\sum\limits_{J_{i_k}\in\Gamma_{i_k}^c} \mathbf{e}_{J_{i_k}}\mathbf{e}_{J_{i_k}}^{\top}$, (\ref{eq:St}) and (\ref{eq:blockpartisi}). \exit
\end{proof}
\begin{prop}
The distribution $\overline{F}_{t}(t_1,t_2)=\mathbb{P}\{\tau_1>t_1, \tau_2>t_2 \vert \mathcal{G}_{t}\}$ is given by
\begin{align*}
\overline{F}_{t}(t_1,t_2)=
\begin{cases}
\sum_{k=1}^m \boldsymbol{\pi}^{\top}(t) \mathbf{S}^{(k)}(t)e^{\mathbf{B}^{(k)}(t_{2}-t)}\mathbf{H}_2e^{\mathbf{B}^{(k)}(t_1-t_2)}\mathbf{H}_1 \mathbb{1}, \; \textrm{if $t_1\geq t_2\geq t\geq 0$} \\[8pt]
\sum_{k=1}^m \boldsymbol{\pi}^{\top}(t) \mathbf{S}^{(k)}(t)e^{\mathbf{B}^{(k)}(t_{1}-t)}\mathbf{H}_1e^{\mathbf{B}^{(k)}(t_2-t_1)}\mathbf{H}_2 \mathbb{1},  \; \textrm{if $t_2\geq t_1\geq t \geq 0$}.
\end{cases}
\end{align*}
\end{prop}
\begin{proof}
By (\ref{eq:relation}) and law of total probability, $F_t(t_1,t_2)=\sum\limits_{i\in E} \pi_i(t)F_{i,t}(t_1,t_2)$. \exit
\end{proof}

\pagebreak

\begin{Rem}\label{rem:probabmass}
As $\mathbf{H}_2\mathbf{H}_1=\mathbf{H}_1\mathbf{H}_2$, the measures $d\overline{F}_{i,t}(t_1,t_2)$ and $d\overline{F}_{t}(t_1,t_2)$ have probability mass $1-\mathbf{e}_i^{\top}\mathbf{H}_2\mathbf{H}_1\mathbb{1}$ and $1-\boldsymbol{\pi}^{\top}(t)\mathbf{H}_2\mathbf{H}_1\mathbb{1}$, respectively, at the point $(t_1=t,t_2=t)$. They are absolutely continuous w.r.t Lebesgue measure $dt_1dt_2$ with density $f_{i,t}(t_1,t_2)$ and $f_t(t_1,t_2)$, subsequently, on $\{(t_1,t_2)\in\mathbb{R}_+^2: t_1,t_2>t\}.$
\end{Rem}

\subsubsection{Conditional joint probability density function}

In general, the joint distribution $\overline{F}_{i,t}(t_1,t_2)$ (resp. $\overline{F}_{t}(t_1,t_2)$) has a singular component $\overline{F}_{i,t}^{(0)}(t_1,t_2)$ (resp. $\overline{F}_{t}^{(0)}(t_1,t_2)$) on the set $\{(t_1,t_2): t_2=t_1\}$. The singular component can be obtained by deriving the joint density of $\tau_1$ and $\tau_2$ and deduce the absolutely continuous and singular parts of the pdf, such as discussed in the theorem below. For non-matrix based bivariate function, see for instance \cite{Sarhan}.
\begin{theo}\label{theo:maintheo}
Given the joint distribution $\overline{F}_{i,t}(t_1,t_2)$ of $(\tau_1,\tau_2)$ as specified in Lemma \ref{lem:lemjointCDF}, the joint probability density $f_{i,t}(t_1,t_2)$ of $(\tau_1,\tau_2)$ is given by
\begin{eqnarray} \label{eq:jointpdfit}
f_{i,t}(t_1,t_2)=
\begin{cases}
f_{i,t}^{(1)}(t_1,t_2), &\; \textrm{if $t_1\geq t_2 > t \geq 0$} \\[4pt]
f_{i,t}^{(2)}(t_1,t_2), &\; \textrm{if $t_2\geq t_1 > t \geq 0$}\\[4pt]
f_{i,t}^{(0)}(t_1,t_1), &\; \textrm{if $t_1= t_2 >  t \geq 0$}, \\[4pt]
1-\mathbf{e}_i^{\top}\mathbf{H}_2\mathbf{H}_1\mathbb{1},  &\; \textrm{if $t_1= t_2 = t \geq 0$},
\end{cases}
\end{eqnarray}
where the absolutely continuous components $f_{i,t}^{(1)}(t_1,t_2)$ and $f_{i,t}^{(2)}(t_1,t_2)$ are
\begin{align*}
f_{i,t}^{(1)}(t_1,t_2)=& \sum_{k=1}^m \mathbf{e}_i^{\top} \mathbf{S}^{(k)}(t)e^{\mathbf{B}^{(k)}(t_{2}-t)}\big[\mathbf{B}^{(k)},\mathbf{H}_2\big]e^{\mathbf{B}^{(k)}(t_1-t_2)}\mathbf{B}^{(k)}\mathbf{H}_1\mathbb{1},\\
f_{i,t}^{(2)}(t_1,t_2)=& \sum_{k=1}^m \mathbf{e}_i^{\top} \mathbf{S}^{(k)}(t)e^{\mathbf{B}^{(k)}(t_1-t)}\big[\mathbf{B}^{(k)},\mathbf{H}_1\big]e^{\mathbf{B}^{(k)}(t_2-t_1)}\mathbf{B}^{(k)}\mathbf{H}_2 \mathbb{1},
\end{align*}
where the matrix operator $[A,B]=AB-BA$ defines the commutator of $A$ and $B$, whilst the singular component part $f_{i,t}^{(0)}(t_1,t_2)$ is defined by the function
\begin{align*}
f_{i,t}^{(0)}(t_1,t_1)=& \sum_{k=1}^m \mathbf{e}_i^{\top} \mathbf{S}^{(k)}(t)e^{\mathbf{B}^{(k)}(t_1-t)}\Big(\big[\mathbf{B}^{(k)},\mathbf{H}_2\big]\mathbf{H}_1 + \big[\mathbf{B}^{(k)},\mathbf{H}_1\big]\mathbf{H}_2 -\mathbf{B}^{(k)}\mathbf{H}_2\mathbf{H}_1\Big)\mathbb{1}.
\end{align*}
\end{theo}
\begin{proof}
The expressions for $f_{i,t}^{(1)}(t_1,t_2)$ and $f_{i,t}^{(2)}(t_1,t_2)$ follow from taking partial derivative $\frac{\partial^2}{\partial t_2\partial t_1}\overline{F}_{i,t}(t_1,t_2)$ (see derivation of Theorem \ref{theo:maintheomph}) taking account
\begin{align}\label{eq:dervexpm}
\frac{d}{dt} \big( \mathbf{S}e^{\mathbf{B}t}\mathbf{A} \big)=\mathbf{S}\mathbf{B}e^{\mathbf{B}t}\mathbf{A}=\mathbf{S}e^{\mathbf{B}t}\mathbf{B}\mathbf{A}.
\end{align}

To get $f_{i,t}^{(0)}(t_1,t_2)$, recall that $ \int_0^{\infty}e^{\mathbf{B}t} dt = -\mathbf{B}^{-1}$, due to the phase-generator matrix $\mathbf{B}$ being negative definite (see Section II4d in \cite{Asmussen2003}). Following Remark \ref{rem:probabmass},
\begin{equation}\label{eq:total}
\begin{split}
\mathbf{e}_i^{\top}\mathbf{H}_2\mathbf{H}_1\mathbb{1}=&\int_t^{\infty} \int_t^{t_1} f_{i,t}^{(1)}(t_1,t_2) dt_2 dt_1 +
\int_t^{\infty} \int_t^{t_2} f_{i,t}^{(2)}(t_1,t_2) dt_1 dt_2 \\&+ \int_t^{\infty} f_{i,t}^{(0)}(t_1,t_1) dt_1.
\end{split}
\end{equation}

\pagebreak

\noindent Applying Fubini's theorem, the first integral is given after some calculations by
\begin{align*}
&\int_t^{\infty} \int_t^{t_1} f_{i,t}^{(1)}(t_1,t_2) dt_2 dt_1 = \int_t^{\infty} \int_{t_2}^{\infty} f_{i,t}^{(1)}(t_1,t_2) dt_1 dt_2 \\
& \hspace{1cm} = \sum_{k=1}^m \mathbf{e}_i^{\top} \mathbf{S}^{(k)}(t) \int_t^{\infty} dt_2 e^{\mathbf{B}^{(k)}(t_{2}-t)}\big[\mathbf{B}^{(k)},\mathbf{H}_2\big] \int_{t_2}^{\infty} dt_1 e^{\mathbf{B}^{(k)}(t_1-t_2)}\mathbf{B}^{(k)}\mathbf{H}_1\mathbb{1} \\
& \hspace{3cm}= \sum_{k=1}^m \mathbf{e}_i^{\top} \mathbf{S}^{(k)}(t)\big[\mathbf{B}^{(k)}\big]^{-1}\big[\mathbf{B}^{(k)},\mathbf{H}_2\big]\mathbf{H}_1\mathbb{1}\\
& \hspace{5cm} = -\int_t^{\infty} \sum_{k=1}^m \mathbf{e}_i^{\top} \mathbf{S}^{(k)}(t)e^{\mathbf{B}^{(k)}(t_1-t)}\big[\mathbf{B}^{(k)},\mathbf{H}_2\big]\mathbf{H}_1 \mathbb{1}dt_1.
\end{align*}

Following the same approach, one can show after some calculations that
\begin{align*}
\int_t^{\infty} \int_t^{t_2} f_{i,t}^{(2)}(t_1,t_2) dt_1 dt_2=&\int_t^{\infty} \int_{t_1}^{\infty} f_{i,t}^{(2)}(t_1,t_2) dt_2 dt_1\\=& \sum_{k=1}^m \mathbf{e}_i^{\top} \mathbf{S}^{(k)}(t)\big[\mathbf{B}^{(k)}\big]^{-1}\big[\mathbf{B}^{(k)},\mathbf{H}_1\big]\mathbf{H}_2 \mathbb{1}\\
=& -\int_t^{\infty} \sum_{k=1}^m \mathbf{e}_i^{\top} \mathbf{S}^{(k)}(t)e^{\mathbf{B}^{(k)}(t_1-t)}\big[\mathbf{B}^{(k)},\mathbf{H}_1\big]\mathbf{H}_2 \mathbb{1} dt_1.
\end{align*}
The proof is established on account of (\ref{eq:total}), the two identities above and
\begin{align*}
 \mathbf{e}_i^{\top}\mathbf{H}_2\mathbf{H}_1\mathbb{1}=-\int_t^{\infty} \sum_{k=1}^m \mathbf{e}_i^{\top} \mathbf{S}^{(k)}(t)e^{\mathbf{B}^{(k)}(t_1-t)}\mathbf{B}^{(k)}\mathbf{H}_2\mathbf{H}_1\mathbb{1}dt_1. \quad \exit
\end{align*}
\end{proof}
\begin{theo}
For $t\geq 0$, the $\mathcal{G}_{t}-$conditional density $f_t(t_1,t_2)$ is given by
\begin{align}\label{eq:jointpdft}
f_{t}(t_1,t_2)=
\begin{cases}
f_{t}^{(1)}(t_1,t_2), &\; \textrm{if $t_1\geq t_2 >  t \geq 0$} \\
f_{t}^{(2)}(t_1,t_2), &\; \textrm{if $t_2\geq t_1 > t \geq 0$} \\
f_{t}^{(0)}(t_1,t_2), &\; \textrm{if $t_1= t_2 >  t \geq 0$}, \\
1-\boldsymbol{\pi}^{\top}(t)\mathbf{H}_2\mathbf{H}_1\mathbb{1}, &\; \textrm{if $t_1=t_2=t$,}
\end{cases}
\end{align}
where the absolutely continuous components $f_{t}^{(1)}(t_1,t_2)$ and $f_{t}^{(2)}(t_1,t_2)$ are
\begin{align*}
f_{t}^{(1)}(t_1,t_2)=& \sum_{k=1}^m \boldsymbol{\pi}^{\top}(t) \mathbf{S}^{(k)}(t)e^{\mathbf{B}^{(k)}(t_{2}-t)}\big[\mathbf{B}^{(k)},\mathbf{H}_2\big]e^{\mathbf{B}^{(k)}(t_1-t_2)}\mathbf{B}^{(k)}\mathbf{H}_1\mathbb{1},\\
f_{t}^{(2)}(t_1,t_2)=& \sum_{k=1}^m  \boldsymbol{\pi}^{\top}(t) \mathbf{S}^{(k)}(t)e^{\mathbf{B}^{(k)}(t_1-t)}\big[\mathbf{B}^{(k)},\mathbf{H}_1\big]e^{\mathbf{B}^{(k)}(t_2-t_1)}\mathbf{B}^{(k)}\mathbf{H}_2 \mathbb{1},
\end{align*}
whilst the singular component $f_{t}^{(0)}(t_1,t_2)$ is defined by the function
\begin{align*}
f_{t}^{(0)}(t_1,t_1)=& \sum_{k=1}^m \boldsymbol{\pi}^{\top}(t) \mathbf{S}^{(k)}(t)e^{\mathbf{B}^{(k)}(t_1-t)}\Big(\big[\mathbf{B}^{(k)},\mathbf{H}_2\big]\mathbf{H}_1 + \big[\mathbf{B}^{(k)},\mathbf{H}_1\big]\mathbf{H}_2 -\mathbf{B}^{(k)}\mathbf{H}_2\mathbf{H}_1\Big)\mathbb{1}.
\end{align*}
\end{theo}
\begin{proof}
It follows from identity (\ref{eq:relation}) that $f_t(t_1,t_2)=\sum_{i\in E} \pi_i(t) f_{i,t}(t_1,t_2).$ \exit
\end{proof}

\pagebreak

\begin{cor}\label{cor:jointCDFit}
The singular component of $\overline{F}_{i,t}(t_1,t_2)$ and $\overline{F}_{t}(t_1,t_2)$ are
\begin{align*}
\overline{F}_{i,t}^{(0)}(t_1,t_1)=&
\sum_{k=1}^m \mathbf{e}_i^{\top} \mathbf{S}^{(k)}(t)e^{\mathbf{B}^{(k)}(t_1-t)}\big[\mathbf{B}^{(k)}\big]^{-1} \\
&\times\Big(\mathbf{B}^{(k)}\mathbf{H}_2\mathbf{H}_1 - \big[\mathbf{B}^{(k)},\mathbf{H}_2\big]\mathbf{H}_1 - \big[\mathbf{B}^{(k)},\mathbf{H}_1\big]\mathbf{H}_2 \Big)\mathbb{1}\\
\overline{F}_{t}^{(0)}(t_1,t_1)=&
\sum_{k=1}^m \boldsymbol{\pi}^{\top}(t) \mathbf{S}^{(k)}(t)e^{\mathbf{B}^{(k)}(t_1-t)}\big[\mathbf{B}^{(k)}\big]^{-1} \\
&\times \Big(\mathbf{B}^{(k)}\mathbf{H}_2\mathbf{H}_1 - \big[\mathbf{B}^{(k)},\mathbf{H}_2\big]\mathbf{H}_1 - \big[\mathbf{B}^{(k)},\mathbf{H}_1\big]\mathbf{H}_2 \Big)\mathbb{1}.
\end{align*}
\end{cor}
Hence, the singular component of $\overline{F}_{it}(t_1,t_2)$ and $\overline{F}_{t}(t_1,t_2)$ is zero if and only if, for $k=1,\dots,m$, $[\mathbf{B}^{(k)}]_{i,j}=0$ for $i\in \boldsymbol{\Gamma}_1^c \cap \boldsymbol{\Gamma}_2^c$ and $j=\Delta$, which is equivalent to
\begin{equation}\label{eq:singularcond}
\begin{split}
\mathbf{B}^{(k)}\mathbf{H}_2\mathbf{H}_1 -\big[\mathbf{B}^{(k)},\mathbf{H}_2\big]\mathbf{H}_1 - \big[\mathbf{B}^{(k)},\mathbf{H}_1\big]\mathbf{H}_2=0.
\end{split}
\end{equation}
\begin{Rem}
Consider the representation (\ref{eq:intensityex}) for the matrices $\{\mathbf{B}^{(k)}\}$. It is clear following (\ref{eq:jointpdft}) that the joint probability density function $f_t(t_1,t_2)$ coincides with the bivariate phase-type distribution \cite{Assaf1984} when we set each $\mathbf{B}^{(k)}=\mathbf{B}$ and $t=0$ taking into account the fact that $[\mathbf{B},\mathbf{H}_1]\mathbf{H}_2= [\mathbf{B},\mathbf{H}_1]$ and  $[\mathbf{B},\mathbf{H}_2]\mathbf{H}_1= [\mathbf{B},\mathbf{H}_2]$.
\end{Rem}

\subsubsection{Conditional joint Laplace transform of $\tau_1$ and $\tau_2$}

In order to compute the $\mathcal{F}_{t,i}-$conditional moment $\mathbb{E}\big\{\tau_1^n\tau_2^m\big\vert \mathcal{F}_{t,i}\big\}$, it is therefore convenient to study the $\mathcal{F}_{t,i}-$conditional joint Laplace transform of $\tau_1$ and $\tau_2$:
\begin{align}\label{eq:jointMGF}
\Psi_{i,t}(\lambda_1,\lambda_2):=\mathbb{E}\big\{e^{-\lambda_1 \tau_1 - \lambda_2 \tau_2}\big\vert \mathcal{F}_{t,i}\big\} \quad \textrm{for $i\in E$}.
\end{align}

\begin{theo}
The $\mathcal{F}_{t,i}-$conditional joint Laplace transform $\Psi_{i,t}(\lambda_1,\lambda_2)$ of the first exit times $\tau_1$ and $\tau_2$ of $X$ (\ref{eq:mixture}) is given for $\lambda_1,\lambda_2\geq 0$, $t\geq 0$ and $i\in E$ by
\begin{align*}
\Psi_{i,t}(\lambda_1,\lambda_2)=&\sum_{k=1}^m \mathbf{e}_i^{\top} \mathbf{S}^{(k)}(t)\big((\lambda_1+\lambda_2)\mathbf{I}-\mathbf{B}^{(k)}\big)^{-1}
\Big( [\mathbf{B}^{(k)},\mathbf{H}_2]\big(\lambda_1\mathbf{I}-\mathbf{B}^{(k)}\big)^{-1}\mathbf{B}^{(k)}\mathbf{H}_1 \\
 & + [\mathbf{B}^{(k)},\mathbf{H}_1]\big(\lambda_2\mathbf{I}-\mathbf{B}^{(k)}\big)^{-1}\mathbf{B}^{(k)}\mathbf{H}_2  + [\mathbf{B}^{(k)},\mathbf{H}_2] \mathbf{H}_1 + [\mathbf{B}^{(k)},\mathbf{H}_1] \mathbf{H}_2 \\ & -\mathbf{B}^{(k)}\mathbf{H}_2\mathbf{H}_1     \Big) \mathbb{1}  + \big(1-\mathbf{e}_i^{\top}\mathbf{H}_2\mathbf{H}_1\mathbb{1}\big) .
\end{align*}
\end{theo}
\begin{proof}
Recall that for $i\in E$, $f_{i,t}(t_1,t_2)=0$ for $t_1,t_2<t$. Following Remark \ref{rem:probabmass},
\begin{align*}
\Psi_{i,t}(\lambda_1,\lambda_2)=&\big(1-\mathbf{e}_i^{\top}\mathbf{H}_2\mathbf{H}_1\mathbb{1}\big)+ \int_0^{\infty}\int_0^{\infty} e^{-\lambda_1 u_1} e^{-\lambda_2 u_2} f_{it}(t+u_1,t+u_2)du_1du_2 \\
=& \big(1-\mathbf{e}_i^{\top}\mathbf{H}_2\mathbf{H}_1\mathbb{1}\big)+ \int_0^{\infty}\int_0^{u_1} e^{-\lambda_1 u_1} e^{-\lambda_2 u_2} f_{it}^{(1)}(t+u_1,t+u_2)du_2du_1 \\
&\hspace{2.85cm}+ \int_0^{\infty}\int_0^{u_2} e^{-\lambda_1 u_1} e^{-\lambda_2 u_2} f_{it}^{(2)}(t+u_1,t+u_2)du_1du_2 \\
&\hspace{2.85cm}+ \int_0^{\infty} e^{-(\lambda_1 + \lambda_2)u_1} f_{it}^{(0)}(t+u_1,t+u_1)du_1.
\end{align*}
The proof is established by applying Fubini's theorem to double integrals. \exit
\end{proof}

\pagebreak

By the law of total probability and Bayes' rule we have the following result.
\begin{theo}
The $\mathcal{G}_{t}-$conditional joint Laplace transform $\Psi_{t}(\lambda_1,\lambda_2):=\mathbb{E}\big\{e^{-\lambda_1 \tau_1 - \lambda_2 \tau_2}\big\vert \mathcal{G}_{t}\big\}$ of the first exit times $\tau_1$ and $\tau_2$ is given for $\lambda_1,\lambda_2, t\geq 0$ by
\begin{align*}
\Psi_{t}(\lambda_1,\lambda_2)=&\sum_{k=1}^m \boldsymbol{\pi}^{\top}(t) \mathbf{S}^{(k)}(t)\big((\lambda_1+\lambda_2)\mathbf{I}-\mathbf{B}^{(k)}\big)^{-1}
\Big( [\mathbf{B}^{(k)},\mathbf{H}_2]\big(\lambda_1\mathbf{I}-\mathbf{B}^{(k)}\big)^{-1}\mathbf{B}^{(k)}\mathbf{H}_1 \\
  & + [\mathbf{B}^{(k)},\mathbf{H}_1]\big(\lambda_2\mathbf{I}-\mathbf{B}^{(k)}\big)^{-1}\mathbf{B}^{(k)}\mathbf{H}_2  + [\mathbf{B}^{(k)},\mathbf{H}_2] \mathbf{H}_1 + [\mathbf{B}^{(k)},\mathbf{H}_1] \mathbf{H}_2  \\ & -\mathbf{B}^{(k)}\mathbf{H}_2\mathbf{H}_1  \Big) \mathbb{1}  + \big(1-\boldsymbol{\pi}^{\top}(t)\mathbf{H}_2\mathbf{H}_1\mathbb{1}\big) .
\end{align*}
\end{theo}

\medskip

\noindent Following the joint Laplace transform (\ref{eq:jointMGF}), we obtain the joint moments:
\begin{align*}
\mathbb{E}\big\{\tau_1^n \tau_2^m \big\vert \mathcal{G}_{t}\big\}=& (-1)^{m+n} \frac{\partial^{m+n}}{\partial \lambda_1^m \partial \lambda_2^n}\Psi_{t}(\lambda_1,\lambda_2)\Big\vert_{\lambda_1=0,\lambda_2=0}.
\end{align*}
\begin{Ex}
The conditional joint moments $\mathbb{E}\{\tau_1\tau_2\vert \mathcal{G}_{t}\}$ is given by
\begin{align*}
\mathbb{E}\{\tau_1\tau_2\vert \mathcal{G}_{t}\}=&\sum_{k=1}^m \boldsymbol{\pi}^{\top}(t) \mathbf{S}^{(k)}(t)\Big(\big[\mathbf{B}^{(k)}\big]^{-1}\mathbf{H}_1\big[\mathbf{B}^{(k)}\big]^{-1}\mathbf{H}_2 + \big[\mathbf{B}^{(k)}\big]^{-1}\mathbf{H}_2\big[\mathbf{B}^{(k)}\big]^{-1}\mathbf{H}_1\Big)\mathbb{1}.
\end{align*}
\end{Ex}

\subsection{Conditional multivariate distributions}

The extension to multivariate case follows similar approach to the bivariate one. Let $\Gamma_1,...,\Gamma_p$ be nonempty stochastically closed subsets of $\mathbb{S}$ such that $\cap_{l=1}^p \Gamma_l$ is a proper subset of $\mathbb{S}$. Without loss of generality, we assume that $\cap_{l=1}^p \Gamma_l=\Delta$. Since $\Gamma_l$ is stochastically closed, we necessarily assume that $q_{ij}^{(k)}=0$, $k=1,\dots,m$, if $i\in\Gamma_l$ and $j\in\Gamma_l^c$, for $l\in\{1,...,p\}$, and $\boldsymbol{\pi}_i\neq 0$ whenever $i\in \cap_{l=1}^p \boldsymbol{\Gamma}_l^c$.

Furthermore, denote by $\tau_k$ the first entry time of $X$ in the set $\boldsymbol{\Gamma}_k$ defined in (\ref{eq:MultiPH}). To formulate the joint distribution of $\{\tau_k\}$, let $(t_{i_1},...,t_{i_p})$ be the time ordering of $(t_1,...,t_p)\in\mathbb{R}_+^p$, where $(i_1,...,i_p)$ is a permutation of $(1,2,...,p)$. Subsequently, we define by $j_{i_k}\in \boldsymbol{\Gamma}_{i_k}^c$ the state that $X$ occupies at time $t=t_{i_k}$.

\begin{lem}\label{lem:main}
Let $t_{i_p}\geq \dots \geq t_{i_1}\geq t_{i_0}=t\geq 0$ be the time ordering of $(t_1,...,t_p)\in\mathbb{R}_+^p$. The joint distribution of the first exit times $\{\tau_k\}$ is given by
\begin{equation}\label{eq:maincdf1}
\begin{split}
\overline{F}_{j,t}(t_{i_1},...,t_{i_p})=&\mathbb{P}\big\{\tau_{i_1}>t_{i_1},...,\tau_{i_p}>t_{i_p} \big\vert \mathcal{F}_{t,j}\big\}\\=&\sum_{k=1}^m \mathbf{e}_j^{\top} \mathbf{S}^{(k)}(t)\prod_{l=1}^p e^{\mathbf{B}^{(k)}(t_{i_l}-t_{i_{l-1}})}\mathbf{H}_{i_l}\mathbb{1},\\
&\hspace{-0.3cm}\mathrm{with} \quad \sum\limits_{k=1}^m \mathbf{S}^{(k)}(t)=\mathbf{I},
\end{split}
\end{equation}
where $\mathbf{H}_{i_k}$ is an $(n\times n)-$ diagonal matrix whose $i$th element $[\mathbf{H}_{i_k}]_{i,i}=\mathbb{1}_{\{i\in\boldsymbol{\Gamma}_{i_k}^c\}}.$
\end{lem}
\begin{proof}
Following similar arguments of the proof in bivariate case, we obtain
\begin{align}
\mathbb{P}\big\{\tau_1>t_1,...,\tau_n>t_p \big \vert \mathcal{F}_{t,j}\big\}=&\mathbb{P}\big\{\tau_{i_1}>t_{i_1},...,\tau_{i_p}>t_{i_p} \big\vert \mathcal{F}_{t_{i_0},j}\big\} \nonumber \\
&\hspace{-2cm}=\mathbb{P}\big\{X_{t_{i_0}}=J_{i_0},X_{t_{i_1}}\in\boldsymbol{\Gamma}_{i_1}^c,...,X_{t_{i_p}}\in\boldsymbol{\Gamma}_{i_p}^c  \big\vert \mathcal{F}_{t_{i_0},j} \big\} \label{eq:Derivation}\\
&\hspace{-2cm}=\sum_{J_{i_1}\in\boldsymbol{\Gamma}_{i_1}^c}...\sum_{J_{i_p}\in\boldsymbol{\Gamma}_{i_p}^c} \mathbb{P}\big\{X_{t_{i_0}}=J_{i_0},X_{t_{i_1}}=J_{i_1},..., X_{t_{i_p}}=J_{i_p}  \big\vert \mathcal{F}_{t_{i_0},j} \big\}.  \nonumber
\end{align}
By Bayes' theorem for conditional probability and the law of total probability,
\begin{align*}
&\mathbb{P}\big\{X_{t_{i_0}}=J_{i_0},X_{t_{i_1}}=J_{i_1},..., X_{t_{i_p}}=J_{i_p}  \big\vert \mathcal{F}_{t_{i_0},j} \big\}\\
&\hspace{1.5cm}= \sum_{k=1}^m \mathbb{P}\big\{X_{t_{i_0}}=J_{i_0},X_{t_{i_1}}=J_{i_1},..., X_{t_{i_p}}=J_{i_p}, \phi=k  \big\vert \mathcal{F}_{t_{i_0},j} \big\}\\
&\hspace{1.5cm}= \sum_{k=1}^m \mathbb{P}\big\{X_{t_{i_0}}=J_{i_0} \big \vert \mathcal{F}_{t_{i_0},j}\big\}
\times \mathbb{P}\big\{\phi=k \big\vert X_{t_{i_0}}=J_{i_0},\mathcal{F}_{t_{i_0},j}\big\}\\
&\hspace{3cm}\times \mathbb{P}\big\{X_{t_{i_1}}=J_{i_1} \big\vert \phi=k, X_{t_{i_0}}=J_{i_0},\mathcal{F}_{t_{i_0},j} \big\} \\
&\hspace{3.2cm}\vdots \\
&\hspace{3cm}
\times \mathbb{P}\big\{X_{t_{i_p}}=J_{i_p} \big\vert \phi=k, X_{t_{i_{p-1}}}=J_{i_{p-1}},\dots,X_{t_{i_0}}=J_{i_0},\mathcal{F}_{t_{i_0},j} \big\} \\
&\hspace{1.5cm}=\sum_{k=1}^m \mathbb{1}_{\{J_{i_0}=j\}} s_{j}^{(k)}(t_{i_0})\times
\mathbf{P}_{J_{i_0},J_{i_1}}^{Q^{(k)}}(t_{i_0},t_{i_1})\times \dots \times
\mathbf{P}_{J_{i_{p-1}},J_{i_p}}^{Q^{(k)}}(t_{i_{p-1}},t_{i_p}).
 \end{align*}
Note that $\mathbb{P}\big\{X_{t_{i_0}}=J_{i_0}\vert \mathcal{F}_{t_{i_0},j}\big\}=1$ iff $J_{i_0}=j$ and $0$ otherwise. In terms of (\ref{eq:bayesianupdates}),
\begin{align*}
&\hspace{3.15cm}\mathbb{P}\big\{X_{t_{i_0}}=J_{i_0},X_{t_{i_1}}=J_{i_1},..., X_{t_{i_p}}=J_{i_p}  \big\vert \mathcal{F}_{t_{i_0},j} \big\}\\
&\hspace{0cm}=\sum_{k=1}^m \mathbf{e}_j^{\top} \widetilde{\mathbf{S}}^{(k)}(t) e^{\mathbf{Q}^{(k)}(t_{i_1}-t_{i_0})} \mathbf{e}_{J_{i_1}} \mathbf{e}_{J_{i_1}}^{\top} e^{\mathbf{Q}^{(k)}(t_{i_2}-t_{i_1})} \dots \mathbf{e}_{J_{i_{p-1}}} \mathbf{e}_{J_{i_{p-1}}}^{\top} e^{\mathbf{Q}^{(k)}(t_{i_p}-t_{i_{p-1}})} \mathbf{e}_{J_{i_p}}  \mathbf{e}_{J_{i_p}}^{\top}\mathbb{1}.
\end{align*}
Therefore, starting from equation (\ref{eq:Derivation}) we have following the above that
\begin{align*}
&\mathbb{P}\big\{\tau_{i_1}>t_{i_1},...,\tau_{i_p}>t_{i_p} \big\vert \mathcal{F}_{t_{i_0},j}\big\}= \sum_{k=1}^m \mathbf{e}_j^{\top} \widetilde{\mathbf{S}}^{(k)}(t) e^{\mathbf{Q}^{(k)}(t_{i_1}-t_{i_0})} \Big(\sum_{J_{i_1}\in\Gamma_{i_1}^c} \mathbf{e}_{J_{i_1}} \mathbf{e}_{J_{i_1}}^{\top}\Big) e^{\mathbf{Q}^{(k)}(t_{i_2}-t_{i_1})} \\
&\hspace{2cm}\dots \Big(\sum_{J_{i_{p-1}} \in \Gamma_{i_{p-1}}^c} \mathbf{e}_{J_{i_{p-1}}} \mathbf{e}_{J_{i_{p-1}}}^{\top} \Big) e^{\mathbf{Q}^{(k)}(t_{i_p}-t_{i_{p-1}})}
\Big(\sum_{J_{i_p}\in \Gamma_{i_p}^c} \mathbf{e}_{J_{i_p}}  \mathbf{e}_{J_{i_p}}^{\top} \Big)\mathbb{1},
\end{align*}
leading to $\overline{F}_{j,t}(t_{i_1},\dots,t_{i_p})$ on account of (\ref{eq:St}), the fact that $\mathbf{H}_{i_k}=\sum\limits_{J_{i_k}\in\Gamma_{i_k}^c} \mathbf{e}_{J_{i_k}}\mathbf{e}_{J_{i_k}}^{\top}$ and after applying block partition (\ref{eq:blockpartisi}) to exponential matrices $e^{\mathbf{Q}^{(k)}t}$. \exit
\end{proof}

\medskip

Notice that the conditional joint probability distribution (\ref{eq:maincdf1}) forms a non-stationary function of time $t$ with the ability to capture heterogeneity and path dependence when conditioning on all previous and current information $\mathcal{F}_{t,j}$ of the mixture process $X$. These features are removed when $\mathbf{B}^{(k)}=\mathbf{B}$, in which case, the result reduces to the multivariate phase-type distribution (\ref{eq:MPH}) for $t=0$.

\begin{prop}
Let $t_{i_p}\geq \dots \geq t_{i_1}\geq t_{i_0}=t\geq 0$ be the time ordering of $(t_1,...,t_p)\in\mathbb{R}_+^p$. The $\mathcal{G}_t-$conditional joint distribution of $\{\tau_k\}$ (\ref{eq:MultiPH}) is given by
\begin{equation}\label{eq:mainmph2}
\begin{split}
\overline{F}_{t}(t_{i_1},...,t_{i_p})=&\mathbb{P}\big\{\tau_{i_1}>t_{i_1},...,\tau_{i_p}>t_{i_p} \big\vert \mathcal{G}_{t}\big\}\\=&\sum_{k=1}^m \boldsymbol{\pi}^{\top}(t) \mathbf{S}^{(k)}(t)\prod_{l=1}^p e^{\mathbf{B}^{(k)}(t_{i_l}-t_{i_{l-1}})}\mathbf{H}_{i_l}\mathbb{1},\\
&\hspace{-0.3cm}\mathrm{with} \quad \sum\limits_{k=1}^m \mathbf{S}^{(k)}(t)=\mathbf{I},
\end{split}
\end{equation}
\end{prop}
\begin{proof}
It follows from (\ref{eq:relation}) that $\overline{F}_{t}(t_{i_1},...,t_{i_p})=\sum\limits_{j\in E} \pi_j(t) \overline{F}_{j,t}(t_{i_1},...,t_{i_p})$. \exit
\end{proof}

\begin{cor}
Set $\mathbf{B}^{(k)}=\mathbf{B}$ and $t=0$ in (\ref{eq:mainmph2}). The distribution of $\{\tau_k\}$,
\begin{equation}\label{eq:main1}
\begin{split}
\mathbb{P}\big\{\tau_{i_1}>t_{i_1},...,\tau_{i_p}>t_{i_p})=&\boldsymbol{\pi}^{\top} \prod_{k=1}^p e^{\mathbf{B}(t_{i_k}-t_{i_{k-1}})}\mathbf{H}_{i_k} \mathbb{1},
\end{split}
\end{equation}
which coincides with the unconditional multivariate phase-type distribution \cite{Assaf1984}.
\end{cor}

The absolutely continuous component of the distribution $\overline{F}_{i,t}\big(t_{i_1},\dots,t_{i_p}\big)$ (respectively, $\overline{F}_{t}\big(t_{i_1},\dots,t_{i_p}\big)$) has a density given by the following theorem.
\begin{theo}\label{theo:maintheomph}
Let $t_{i_p}\geq \dots \geq t_{i_1} > t_{i_0}=t\geq 0$ be the time ordering of $(t_1,...,t_p)\in\mathbb{R}_+^p$. The conditional joint density function of $\{\tau_k\}$ (\ref{eq:MultiPH}) is given by
\begin{align*}
f_{i,t}\big(t_{i_1},\dots,t_{i_p}\big)=&(-1)^p \sum_{k=1}^m \mathbf{e}_i^{\top} \mathbf{S}^{(k)}(t) \prod_{l=1}^{p-1} e^{\mathbf{B}^{(k)}(t_l-t_{l-1})}[\mathbf{B}^{(k)},\mathbf{H}_{i_l}] e^{\mathbf{B}^{(k)}(t_p-t_{p-1})}\mathbf{B}^{(k)}\mathbf{H}_{i_p}\mathbb{1}, \\[8pt]
f_{t}\big(t_{i_1},\dots,t_{i_p}\big)=&(-1)^p \sum_{k=1}^m \boldsymbol{\pi}^{\top}(t) \mathbf{S}^{(k)}(t) \prod_{l=1}^{p-1} e^{\mathbf{B}^{(k)}(t_l-t_{l-1})}[\mathbf{B}^{(k)},\mathbf{H}_{i_l}] e^{\mathbf{B}^{(k)}(t_p-t_{p-1})}\mathbf{B}^{(k)}\mathbf{H}_{i_p}\mathbb{1}.
\end{align*}
\end{theo}
\begin{proof}
The proof follows from taking $p-$times partial derivative to $F_{t}\big(t_{i_1},\dots,t_{i_p}\big)$:
\begin{align*}
f_{t}\big(t_{i_1},\dots,t_{i_p}\big)=&(-1)^p \frac{\partial^p \overline{F}_{t}}{\partial t_{i_p}\dots \partial t_{i_1}}\big(t_{i_1},\dots,t_{i_p}\big).
\end{align*}
To establish the result, it is enough to show the following partial derivative holds
\begin{align*}
\frac{\partial^{p}}{\partial t_{i_{p}}\dots\partial t_{i_1}}\prod_{l=1}^p e^{\mathbf{B}^{(k)}(t_{i_l}-t_{i_{l-1}})}\mathbf{H}_{i_l} = \prod_{l=1}^{p-1} e^{\mathbf{B}^{(k)}(t_{i_l}-t_{i_{l-1}})}[\mathbf{B}^{(k)},\mathbf{H}_{i_l}] e^{\mathbf{B}^{(k)}(t_{i_p}-t_{i_{p-1}})} \mathbf{B}^{(k)} \mathbf{H}_{i_p}.
\end{align*}
To justify the claim, we use induction argument. For this purpose, recall that
\begin{equation}\label{eq:perkalian}
\begin{split}
\prod_{l=1}^p e^{\mathbf{B}^{(k)}(t_{i_l}-t_{i_{l-1}})}\mathbf{H}_{i_l}=&e^{\mathbf{B}^{(k)}(t_{i_1}-t_{i_{0}})}\mathbf{H}_{i_1}e^{\mathbf{B}^{(k)}(t_{i_2}-t_{i_{1}})}\mathbf{H}_{i_2} \\
&\times \prod_{k=3}^p e^{\mathbf{B}^{(k)}(t_{i_k}-t_{i_{k-1}})}\mathbf{H}_{i_k}.
\end{split}
\end{equation}

\noindent Hence, by (\ref{eq:dervexpm}) and applying integration by part as we did before, we have
\begin{align*}
\frac{\partial}{\partial t_{i_1}}\prod_{k=1}^p e^{\mathbf{B}^{(k)}(t_{i_k}-t_{i_{k-1}})}\mathbf{H}_{i_k}=&
e^{\mathbf{B}^{(k)}(t_{i_1}-t_{i_{0}})}[\mathbf{B}^{(k)},\mathbf{H}_{i_1}]\prod_{k=2}^p e^{\mathbf{B}^{(k)}(t_{i_k}-t_{i_{k-1}})}\mathbf{H}_{i_k},
\end{align*}
from which the second order partial derivative $\frac{\partial^2}{\partial t_{i_2}\partial t_{i_1}}$ of (\ref{eq:perkalian}) is given by
\begin{align*}
\frac{\partial^2}{\partial t_{i_2}\partial t_{i_1}}\prod_{k=1}^p e^{\mathbf{B}^{(k)}(t_{i_k}-t_{i_{k-1}})}\mathbf{H}_{i_k}=&
e^{\mathbf{B}^{(k)}(t_{i_1}-t_{i_{0}})}[\mathbf{B}^{(k)},\mathbf{H}_{i_1}]\frac{\partial}{\partial t_{i_2}} \prod_{k=2}^p e^{\mathbf{B}^{(k)}(t_{i_k}-t_{i_{k-1}})}\mathbf{H}_{i_k} \\
&\hspace{-3.5cm}=e^{\mathbf{B}^{(k)}(t_{i_1}-t_{i_{0}})}[\mathbf{B}^{(k)},\mathbf{H}_{i_1}] e^{\mathbf{B}^{(k)}(t_{i_2}-t_{i_{1}})}[\mathbf{B}^{(k)},\mathbf{H}_{i_2}] \prod_{k=3}^p e^{\mathbf{B}^{(k)}(t_{i_k}-t_{i_{k-1}})}\mathbf{H}_{i_k}.
\end{align*}
After $(p-1)$steps of taking the partial derivative, one can show that
\begin{align*}
\frac{\partial^{p-1}}{\partial t_{i_{p-1}}\dots\partial t_{i_1}}\prod_{k=1}^p e^{\mathbf{B}^{(k)}(t_{i_k}-t_{i_{k-1}})}\mathbf{H}_{i_k} = \prod_{k=1}^{p-1} e^{\mathbf{B}^{(k)}(t_{i_k}-t_{i_{k-1}})}[\mathbf{B}^{(k)},\mathbf{H}_{i_k}] e^{\mathbf{B}^{(k)}(t_{i_p}-t_{i_{p-1}})} \mathbf{H}_{i_p}.
\end{align*}

The claim is established on account of (\ref{eq:dervexpm}) and the fact that
\begin{align*}
\frac{\partial^p \overline{F}_{t}}{\partial t_{i_p}\dots \partial t_{i_1}} (t_{i_1},\dots,t_{i_p})=&
\sum_{k=1}^m \boldsymbol{\pi}^{\top}(t) \mathbf{S}^{(k)}(t)\frac{\partial^p}{\partial t_{i_p}\dots \partial t_{i_1}} \prod_{l=1}^p e^{\mathbf{B}^{(k)}(t_{i_l}-t_{i_{l-1}})}\mathbf{H}_{i_l} \mathbb{1}. \exit
\end{align*}
\end{proof}

\noindent However, due to complexity of the joint distributions, the singular component of $\overline{F}_{i,t}(t_{i_1},\dots,t_{i_p})$ (resp. $\overline{F}_{t}(t_{i_1},\dots,t_{i_p})$) is more complicated to get in closed form.

\medskip

Following (\ref{eq:maincdf1}) and (\ref{eq:mainmph2}), we see that the distributions are uniquely characterized by the Bayesian update on the probability $\boldsymbol{\pi}$ of starting the process $X$ in any of the $(n+1)$ phases, the speeds of the process represented by the phase-generator matrices $\{\mathbf{B}^{(k)}\}$, and by the Bayesian update of switching probability matrix $\mathbf{S}^{(k)}$. The initial profile of the distributions form a generalized mixture of the multivariate phase-type distributions \cite{Assaf1984}. Unlike the latter, the distributions have non-stationary and path dependence property when conditioning on the available information (either full or partial) of $X$, which is non-Markov. When the process never repeatedly changes the speed, i.e., $\mathbf{B}^{(k)}=\mathbf{B}$, all these properties are removed and the initial distributions reduce to \cite{Assaf1984}. As in the univariate case, the multivariate distributions have closure and dense properties, which can be established in similar ways to the univariate analogs using matrix analytic approach \cite{Assaf1984}. We refer among others to \cite{Neuts1981}, \cite{Assaf1982}, \cite{He} and \cite{Rolski} for Markov model, and to \cite{Surya2018} for the mixture model. As a result, we have the following theorem.

\begin{theo}[Closure and dense properties]
The conditional multivariate probability distribution (\ref{eq:mainmph2}) forms a dense class of distributions on $\mathbb{R}_+^p$, which is closed under finite convex mixtures and finite convolutions.
\end{theo}

\pagebreak
\section{Some explicit and numerical examples}
This section discusses some explicit examples of the main results presented in Section \ref{sec:mainsection}, particularly on the bivariate distributions. Using the closed form density functions (\ref{eq:jointpdfit}) and (\ref{eq:jointpdft}), we discuss the mixtures of exponential distributions, Marshall-Olkin exponential distributions, and their generalization.

\begin{Ex}[Mixture of exponential distributions]

Consider the mixture process $X$ (\ref{eq:mixture}) defined on the state space $\mathbb{S}=\{1,2,3\}\cup \{\Delta\}$ with stochastically closed sets $\boldsymbol{\Gamma}_1=\{2,\Delta\}$ and $\boldsymbol{\Gamma}_2=\{3,\Delta\}$. Assume that the speed of the mixture process is represented by the following phase generator matrices:
\begin{equation*}
\mathbf{B}^{(2)} = \left(\begin{array}{ccc}
  -(b_1+b_2) & b_1 & b_2 \\
  0 & -b_2 & 0 \\
  0 & 0 & -b_1
\end{array}\right)
\;\; \mathrm{and} \;\;
\mathbf{B}^{(1)} = \left(\begin{array}{ccc}
  -(a_1+a_2) & a_1 & a_2 \\
  0 & -a_2 & 0 \\
  0 & 0 & -a_1
\end{array}\right).
\end{equation*}

It is straightforward to derive from the state space representation that
\begin{align*}
\mathbf{H}_1= \left(\begin{array}{ccc}
 1 & 0 & 0 \\
  0 & 0 & 0 \\
   0 & 0  & 1
\end{array}\right)
\quad \textrm{and} \quad
\mathbf{H}_2= \left(\begin{array}{ccc}
 1 & 0 & 0 \\
  0 & 1 & 0 \\
   0 & 0  & 0
\end{array}\right).
\end{align*}
After some calculations, the matrices $[\mathbf{B}^{(1)},\mathbf{H}_k]$ and $\mathbf{B}^{(1)}\mathbf{H}_k$, $k=1,2$, are
\begin{align*}
[\mathbf{B}^{(1)},\mathbf{H}_1]=
\left(\begin{array}{ccc}
 0 & -a_1 & 0 \\
  0 & 0 & 0 \\
   0 & 0  & 0
\end{array}\right)
\quad \textrm{and}& \quad
[\mathbf{B}^{(1)},\mathbf{H}_2]=
\left(\begin{array}{ccc}
 0 & 0 & -a_2 \\
  0 & 0 & 0 \\
   0 & 0  & 0
\end{array}\right) \\[8pt]
\mathbf{B}^{(1)}\mathbf{H}_1=
\left(\begin{array}{ccc}
 -(a_1+a_2) & 0 & a_2 \\
  0 & 0 & 0 \\
   0 & 0  & -a_1
\end{array}\right)
\quad \textrm{and}& \quad
\mathbf{B}^{(1)}\mathbf{H}_2=
\left(\begin{array}{ccc}
 -(a_1+a_2) & a_1 & 0 \\
  0 & -a_2 & 0 \\
   0 & 0  & 0
\end{array}\right).
\end{align*}
Similarly defined for $[\mathbf{B}^{(2)},\mathbf{H}_k]$ and $\mathbf{B}^{(2)}\mathbf{H}_k$, for $k=1,2$.
Set the matrix $\mathbf{S}=\textrm{diag}(p_1,p_2,p_3)$, with $0<p_k<1$, for $k=1,2,3$, whilst the initial probability $\boldsymbol{\pi}$ has mass one on the state $1$, i.e., $\boldsymbol{\pi}=\mathbf{e}_1$. It is straightforward to check that the condition (\ref{eq:singularcond}) is clearly satisfied implying that the joint density function (\ref{eq:jointpdfit}) has zero singular component. Hence, following (\ref{eq:jointpdfit}) we have for $t_1,t_2\geq 0$
\begin{align*}
f_{\tau_1,\tau_2}(t_1,t_2)=p_1 b_1 e^{-b_1 t_1}b_2 e^{-b_2 t_2} + (1-p_1)a_1 e^{-a_1 t_1}a_2 e^{-a_2 t_2},
\end{align*}

The marginal distribution of $\tau_1$ and $\tau_2$ are given respectively by
\begin{align*}
f_{\tau_1}(t_1)=&p_1 b_1 e^{-b_1 t_1} + (1-p_1)a_1 e^{-a_1 t_1}\\
f_{\tau_2}(t_2)=&p_1 b_2 e^{-b_2 t_2} + (1-p_1) a_2 e^{-a_2 t_2}.
\end{align*}

Hence, clearly, as $f_{\tau_1,\tau_2}(t_1,t_2)\neq f_{\tau_1}(t_1)f_{\tau_2}(t_2)$, it follows that the exit times $\tau_1$ and $\tau_2$ are not independent under the mixture model. They are independent if and only if $a_1=b_1=b_2=a_2$, in which case the mixture corresponds to a simple Markov jump process. See the example on p. 691 in \cite{Assaf1984} and p. 59 in \cite{He}.

Furthermore, when conditioning on the information set $\mathcal{F}_{t,i}$ with $i=1$, the conditional joint density function $f_{1,t}(t_1,t_2)$ is given for $t_1,t_2\geq t\geq 0$  by
\begin{equation}\label{eq:ex1}
\begin{split}
f_{1,t}(t_1,t_2)=&s_1(t) e^{(b_1+b_2)t} b_1 e^{-b_1 t_1} b_2 e^{-b_2t_2} \\
&+ (1-s_1(t))  e^{(a_1+a_2)t} a_1 e^{-a_1 t_1} a_2 e^{-a_2t_2},
\end{split}
\end{equation}
where the switching probability $s_1(t)$ is defined for $\mathcal{F}_{t-}=\emptyset$ and $t\geq 0$ by
\begin{align*}
s_1(t)=\frac{p_1e^{-(b_1+b_2)t}}{p_1e^{-(b_1+b_2)t} + (1-p_1)e^{-(a_1+a_2)t}}.
\end{align*}
Observe that, on the event $\{\textrm{min}\{\tau_1,\tau_2\}>t\}$, one can check that $s_1(t)\rightarrow 0$ (resp. $1$) as $t\rightarrow \infty$ if $b_1+b_2 > \; (\textrm{resp. $<$})\; a_1+a_2$, implying that the mixture $X$ moves as a Markov process at the slow speed $\mathbf{B}^{(1)}$ (resp. $\mathbf{B}^{(2)}$) in the long run.

Given that $\Gamma_1^c\cap\Gamma_2^c=\{1\}$, we have $\pi_1(t)=1$ for all $t\geq 0$. Hence, the density function $f_t(t_1,t_2)$ (\ref{eq:jointpdft}) has therefore the same expression as (\ref{eq:ex1}).
\end{Ex}

\begin{Ex}[Mixture of Marshall-Olkin distributions]

Consider the mixture process $X$ (\ref{eq:mixture}) with the same state space $\mathbb{S}$ and stochastically closed sets $\Gamma_1$ and $\Gamma_2$ as defined above. Let the speed of the mixture process be given by
\begin{align*}
\mathbf{B}^{(2)} =& \left(\begin{array}{ccc}
  -(b_1+b_2+b_3) & b_1 & b_2 \\
  0 & -(b_2+b_3) & 0 \\
  0 & 0 & -(b_1+b_3)
\end{array}\right)\\[8pt]
\mathbf{B}^{(1)} =& \left(\begin{array}{ccc}
  -(a_1+a_2+a_3) & a_1 & a_2 \\
  0 & -(a_2+a_3) & 0 \\
  0 & 0 & -(a_1+a_3)
\end{array}\right).
\end{align*}
Set the matrix $\mathbf{S}=\textrm{diag}(p_1,p_2,p_3)$, with $0<p_k<1$, for $k=1,2,3$, while the initial distribution has mass one on the state $1$, i.e., $\pi_1=1$. Following (\ref{eq:singularcond}), the joint density $f_{1,t}(t_1,t_2)$ has singular part on the set $\{(t_1,t_2):t_2=t_1\}$.By Theorem \ref{theo:maintheo} and Corollary \ref{cor:jointCDFit}, the absolutely continuous parts are given by
\begin{align*}
f_{1,t}^{(1)}(t_1,t_2)=&s_1(t)b_2(b_1+b_3)e^{-b_1(t_1-t)} e^{-b_2(t_2-t)} e^{-b_3(t_1-t_2)}
\\&+ \big(1-s_1(t)\big)a_2(a_1+a_3)e^{-a_1(t_1-t)} e^{-a_2(t_2-t)} e^{-a_3(t_1-t_2)}\\[8pt]
f_{1,t}^{(2)}(t_1,t_2)=&s_1(t)b_1(b_2+b_3)e^{-b_1(t_1-t)}e^{-b_2(t_2-t)} e^{-b_3(t_2-t_1)}\\
&+ \big(1-s_1(t)\big)a_1(a_2+a_3)e^{-a_1(t_1-t)}e^{-a_2(t_2-t)} e^{-a_3(t_2-t_1)},
\end{align*}
whereas the singular component $f_{t}^{(0)}(t_1,t_2)$ is given by the function:
\begin{align*}
f_{1,t}^{(0)}(t_1,t_1)=&s_1(t) b_3 e^{-(b_1+b_2+b_3)(t_1-t)}
+ \big(1-s_1(t) \big) a_3 e^{-(a_1+a_2+a_3)(t_1-t)}.
\end{align*}
Note that the switching probability $s_1(t)$ is given for $\mathcal{F}_{t-}=\emptyset$ and $t\geq 0$ by
\begin{align*}
s_1(t)=\frac{p_1e^{-(b_1+b_2+b_3)t}}{p_1e^{-(b_1+b_2+b_3)t} + (1-p_1)e^{-(a_1+a_2+a_3)t}}.
\end{align*}
\end{Ex}

\subsection{General explicit identity}

In order to take advantage of the structure of the generator matrices, let
\begin{equation}\label{eq:intensityex}
\mathbf{B}^{(2)} = \left(\begin{array}{ccc}
 \mathbf{B}_{11} & \mathbf{B}_{12} & \mathbf{B}_{13} \\
  \mathbf{0} & \mathbf{B}_{22} & \mathbf{0} \\
   \mathbf{0} & \mathbf{0}  & \mathbf{B}_{33}\
\end{array}\right)
\quad \textrm{and} \quad
\mathbf{B}^{(1)} = \left(\begin{array}{ccc}
 \mathbf{A}_{11} & \mathbf{A}_{12} & \mathbf{A}_{13} \\
  \mathbf{0} & \mathbf{A}_{22} & \mathbf{0} \\
   \mathbf{0} & \mathbf{0}  & \mathbf{A}_{33}
\end{array}\right).
\end{equation}
The generator matrices $\mathbf{B}^{(1)}$ and $\mathbf{B}^{(2)}$ are nonsingular if and only if $\mathbf{A}_{11}$, $\mathbf{A}_{22}$, $\mathbf{A}_{33}$, $\mathbf{B}_{11}$, $\mathbf{B}_{22}$ and $\mathbf{B}_{33}$ are all nonsingular. The matrices $\mathbf{H}_1$ and $\mathbf{H}_2$ are
\begin{align*}
\mathbf{H}_1= \left(\begin{array}{ccc}
 \mathbf{I} & \mathbf{0} & \mathbf{0} \\
  \mathbf{0} & \mathbf{0} & \mathbf{0} \\
   \mathbf{0} & \mathbf{0}  & \mathbf{I}
\end{array}\right)
\quad \textrm{and} \quad
\mathbf{H}_2= \left(\begin{array}{ccc}
 \mathbf{I} & \mathbf{0} & \mathbf{0} \\
  \mathbf{0} & \mathbf{I} & \mathbf{0} \\
   \mathbf{0} & \mathbf{0}  & \mathbf{0}
\end{array}\right).
\end{align*}
After some calculations the matrix $[\mathbf{B}^{(1)},\mathbf{H}_k]$ and $\mathbf{B}^{(1)}\mathbf{H}_k$, $k=1,2$, are given by
\begin{align*}
[\mathbf{B}^{(1)},\mathbf{H}_1]=&
\left(\begin{array}{ccc}
 \mathbf{0} & -\mathbf{A}_{12} & \mathbf{0} \\
  \mathbf{0} & \mathbf{0} & \mathbf{0} \\
   \mathbf{0} & \mathbf{0}  & \mathbf{0}
\end{array}\right)
\quad \textrm{and} \quad
[\mathbf{B}^{(1)},\mathbf{H}_2]=
\left(\begin{array}{ccc}
 \mathbf{0} & \mathbf{0} & -\mathbf{A}_{13} \\
  \mathbf{0} & \mathbf{0} & \mathbf{0} \\
   \mathbf{0} & \mathbf{0}  & \mathbf{0}
\end{array}\right) \\[8pt]
\mathbf{B}^{(1)}\mathbf{H}_1=&
\left(\begin{array}{ccc}
 \mathbf{A}_{11} & \mathbf{0} & \mathbf{A}_{13} \\
  \mathbf{0} & \mathbf{0} & \mathbf{0} \\
   \mathbf{0} & \mathbf{0}  & \mathbf{A}_{33}
\end{array}\right)
\quad \textrm{and} \quad
\mathbf{B}^{(1)}\mathbf{H}_2=
\left(\begin{array}{ccc}
 \mathbf{A}_{11} & \mathbf{A}_{12} & \mathbf{0} \\
  \mathbf{0} & \mathbf{A}_{22} & \mathbf{0} \\
   \mathbf{0} & \mathbf{0}  & \mathbf{0}
\end{array}\right).
\end{align*}
Similarly defined for $[\mathbf{B},\mathbf{H}_k]$ and $\mathbf{B}\mathbf{H}_k$, for $k=1,2$. A rather long calculations using infinite series representation of exponential matrix shows following (\ref{eq:jointpdfit}),
\begin{eqnarray} \label{eq:jointpdfitex}
f_{i,t}(t_1,t_2)=
\begin{cases}
f_{i,t}^{(1)}(t_1,t_2), &\; \textrm{if $t_1\geq t_2 > t \geq 0$} \\[7pt]
f_{i,t}^{(2)}(t_1,t_2), &\; \textrm{if $t_2\geq t_1 > t \geq 0$}\\[7pt]
f_{i,t}^{(0)}(t_1,t_1), &\; \textrm{if $t_1= t_2 > t \geq 0$}, \\[7pt]
1-\mathbf{e}_i^{\top}\mathbf{H}_2\mathbf{H}_1\mathbb{1},
&\; \textrm{if $t_1= t_2=t\geq 0$},
\end{cases}
\end{eqnarray}
for $i\in \Gamma_1^c \cap \Gamma_2^c$, with the absolutely continuous parts $f_{i,t}^{(1)}(t_1,t_2)$ and $f_{i,t}^{(2)}(t_1,t_2)$:
 \begin{align*}
 f_{i,t}^{(1)}(t_1,t_2)=&-\mathbf{e}_i^{\top}\Big\{\mathbf{S}_{11}(t)e^{\mathbf{B}_{11}(t_2-t)}\mathbf{B}_{13}e^{\mathbf{B}_{33}(t_1-t_2)}\mathbf{B}_{33} \\
 &\hspace{2cm}+\big[\mathbf{I}-\mathbf{S}_{11}(t)\big]e^{\mathbf{A}_{11}(t_2-t)}\mathbf{A}_{13}e^{\mathbf{A}_{33}(t_1-t_2)}\mathbf{A}_{33}\Big\}\mathbb{1},\\[8pt]
f_{i,t}^{(2)}(t_1,t_2)=&-\mathbf{e}_i^{\top}\Big\{\mathbf{S}_{11}(t)e^{\mathbf{B}_{11}(t_1-t)}\mathbf{B}_{12}e^{\mathbf{B}_{22}(t_2-t_1)}\mathbf{B}_{22} \\
 &\hspace{2cm}+\big[\mathbf{I}-\mathbf{S}_{11}(t)\big]e^{\mathbf{A}_{11}(t_1-t)}\mathbf{A}_{12}e^{\mathbf{A}_{22}(t_2-t_1)}\mathbf{A}_{22}\Big\}\mathbb{1},
 \end{align*}
 whereas the singular component $f_{i,t}^{(0)}(t_1,t_1)$ is defined by the function
 \begin{align*}
  f_{i,t}^{(0)}(t_1,t_1)=&-\mathbf{e}_i^{\top}\Big\{\mathbf{S}_{11}(t) e^{\mathbf{B}_{11}(t_1-t)}\big(\mathbf{B}_{11} \mathbb{1}+ \mathbf{B}_{12} \mathbb{1}+ \mathbf{B}_{13}\mathbb{1}\big)\\
 &\hspace{2cm}+\big[\mathbf{I}-\mathbf{S}_{11}(t)\big]e^{\mathbf{A}_{11}(t_1-t)}\big(\mathbf{A}_{11}\mathbb{1} + \mathbf{A}_{12}\mathbb{1} + \mathbf{A}_{13}\mathbb{1}\big)\Big\}.
 \end{align*}

 \pagebreak

\noindent Note that $\mathbf{S}_{11}(t)$ denotes the switching probability matrix of $X$ on $\Gamma_1^c\cap \Gamma_2^c$.

It is straightforward to see that the absolutely continuous parts of $f_{i,t}(t_1,t_2)$ vanishes when $\mathbf{A}_{12}=\mathbf{B}_{12}=\mathbf{0}$ and $\mathbf{A}_{13}=\mathbf{B}_{13}=\mathbf{0}$, in which case $\Gamma_1$ and $\Gamma_2$ are non overlapping. Moreover, $f_{i,t}(t_1,t_2)$ has no singular component $f_{i,t}^{(0)}(t_1,t_2)$ iff
 \begin{align*}
 \mathbf{A}_{11}\mathbb{1}+\mathbf{A}_{12}\mathbb{1}+\mathbf{A}_{13}\mathbb{1}=\mathbf{0}=\mathbf{B}_{11}\mathbb{1}+\mathbf{B}_{12}\mathbb{1}+\mathbf{B}_{13}\mathbb{1}.
 \end{align*}

\noindent Denote by $\boldsymbol{\alpha}$ and $\boldsymbol{\gamma}$ the restriction of the probability $\boldsymbol{\pi}$ on  the set $\Gamma_1^c\cap \Gamma_2^c$ and $E\backslash\{ \Gamma_1^c\cap \Gamma_2^c\}$, respectively, s.t. $\boldsymbol{\pi}=\big(\boldsymbol{\alpha},\boldsymbol{\gamma}\big)$. The Bayesian updates $\boldsymbol{\pi}(t)$ on $\Gamma_1^c\cap \Gamma_2^c$ is defined by $\boldsymbol{\alpha}(t)$. The conditional density $f_{t}(t_1,t_2)$ of $\tau_1$ and $\tau_2$ is given by
\begin{eqnarray} \label{eq:jointpdfitex2}
f_{t}(t_1,t_2)=
\begin{cases}
f_{t}^{(1)}(t_1,t_2), &\; \textrm{if $t_1\geq t_2 >  t \geq 0$} \\[7pt]
f_{t}^{(2)}(t_1,t_2), &\; \textrm{if $t_2\geq t_1 >  t \geq 0$}\\[7pt]
f_{t}^{(0)}(t_1,t_1), &\; \textrm{if $t_1= t_2 >  t \geq 0$}, \\[7pt]
1-\boldsymbol{\alpha}^{\top}(t)\mathbb{1},
&\; \textrm{if $t_1= t_2=t\geq 0$}
\end{cases}
\end{eqnarray}
where the subdensity functions $f_{t}^{(1)}(t_1,t_2)$, $f_{t}^{(2)}(t_1,t_2)$ and $f_{t}^{(0)}(t_1,t_2)$ are
 \begin{align*}
 f_{t}^{(1)}(t_1,t_2)=&-\boldsymbol{\alpha}^{\top}(t)\Big\{\mathbf{S}_{11}(t)e^{\mathbf{B}_{11}(t_2-t)}\mathbf{B}_{13}e^{\mathbf{B}_{33}(t_1-t_2)}\mathbf{B}_{33} \\
 &\hspace{2cm}+\big[\mathbf{I}-\mathbf{S}_{11}(t)\big]e^{\mathbf{A}_{11}(t_2-t)}\mathbf{A}_{13}e^{\mathbf{A}_{33}(t_1-t_2)}\mathbf{A}_{33}\Big\}\mathbb{1},\\[8pt]
f_{t}^{(2)}(t_1,t_2)=&-\boldsymbol{\alpha}^{\top}(t)\Big\{\mathbf{S}_{11}(t)e^{\mathbf{B}_{11}(t_1-t)}\mathbf{B}_{12}e^{\mathbf{B}_{22}(t_2-t_1)}\mathbf{B}_{22} \\
 &\hspace{2cm}+\big[\mathbf{I}-\mathbf{S}_{11}(t)\big]e^{\mathbf{A}_{11}(t_1-t)}\mathbf{A}_{12}e^{\mathbf{A}_{22}(t_2-t_1)}\mathbf{A}_{22}\Big\}\mathbb{1}, \\[8pt]
 f_{t}^{(0)}(t_1,t_1)=&-\boldsymbol{\alpha}^{\top}(t)\Big\{\mathbf{S}_{11}(t) e^{\mathbf{B}_{11}(t_1-t)}\big(\mathbf{B}_{11}\mathbb{1} + \mathbf{B}_{12}\mathbb{1} + \mathbf{B}_{13}\mathbb{1}\big)\\
 &\hspace{2cm}+\big[\mathbf{I}-\mathbf{S}_{11}(t)\big]e^{\mathbf{A}_{11}(t_1-t)}\big(\mathbf{A}_{11} \mathbb{1}+ \mathbf{A}_{12}\mathbb{1} + \mathbf{A}_{13}\mathbb{1}\big)\Big\}.
 \end{align*}

The marginal probability density functions $f_{\tau_k}^{(i)}(s\vert t):=-  \partial_s \mathbb{P}\{\tau_k>s \big\vert \mathcal{F}_{t,i}\}$ and $f_{\tau_k}(s\vert t):=-  \partial_s \mathbb{P}\{\tau_k>s \big\vert \mathcal{G}_t\}$ of $\tau_k$, $k=1,2$, can be deduced from $\overline{F}_{i,t}(t_1,t_2)$ and $\overline{F}_{t}(t_1,t_2)$. They are given for $s\geq t\geq 0$ and $i\in \Gamma_k^c$ by the following:
 \begin{eqnarray}
f_{\tau_k}^{(i)}(s\vert t)&=& - \mathbf{e}_i^{\top} \Big(\mathbf{S}(t) e^{\mathbf{B}_k(s-t)} \mathbf{B}_k+ \big[\mathbf{I}-\mathbf{S}(t)\big] e^{\mathbf{A}_k(s-t)} \mathbf{A}_k  \Big)\mathbb{1},\\[8pt] \label{eq:marginal1}
f_{\tau_k}(s\vert t)&=& - \boldsymbol{\pi}^{\top}(t) \Big(\mathbf{S}(t) e^{\mathbf{B}_k(s-t)} \mathbf{B}_k + \big[\mathbf{I}-\mathbf{S}(t)\big] e^{\mathbf{A}_k(s-t)} \mathbf{A}_k  \Big)\mathbb{1}, \label{eq:marginal2}
  \end{eqnarray}
where the phase-generator matrices $\mathbf{B}_k$ and $\mathbf{A}_k$, for $k=1,2$, are defined by
  \begin{align*}
\mathbf{B}_1=  \left(\begin{array}{cc}
 \mathbf{B}_{11} & \mathbf{B}_{13} \\
  \mathbf{0} &  \mathbf{B}_{33}
\end{array}\right)
\quad \textrm{and}& \quad
\mathbf{A}_1=  \left(\begin{array}{cc}
 \mathbf{A}_{11} & \mathbf{A}_{13} \\
  \mathbf{0} &  \mathbf{A}_{33}
\end{array}\right), \\[8pt]
\mathbf{B}_2=  \left(\begin{array}{cc}
 \mathbf{B}_{11} & \mathbf{B}_{12} \\
  \mathbf{0} &  \mathbf{B}_{22}
\end{array}\right)
\quad \textrm{and}& \quad
\mathbf{A}_2=  \left(\begin{array}{cc}
 \mathbf{A}_{11} & \mathbf{A}_{12} \\
  \mathbf{0} &  \mathbf{A}_{22}
\end{array}\right).
\end{align*}

In the next section we discuss some numerical examples of the main results presented in Section \ref{sec:mainsection}, in particular on the conditional bivariate distributions, taking the advantage of the structure of phase-generator matrices given in (\ref{eq:intensityex}).

\subsection{Numerical examples}

\begin{figure}
\begin{center}
  \begin{tikzpicture}[font=\sffamily]

        \tikzset{node style/.style={state,
                                    minimum width=1.5cm,
                                    line width=1mm,
                                    fill=gray!20!white}}

        \node[node style] at (2, 0)     (s1)  {$J_1$};
        \node[node style] at (5, 0)     (s2)  {$\dots$};

        \node[node style] at (8, 0)     (s3)  {$J_n$};
        \node[node style] at (10, -2)  (s4)  {$J_{n+2}$};

        \node[node style] at (10, 2)   (s5)  {$J_{n+1}$};

        \node[node style] at (12,0)    (s6)  {$\Delta$};

        \draw[every loop,
              auto=right,
              line width=1mm,
              >=latex,
              draw=orange,
              fill=orange]

            (s1)  edge[bend right=20, auto=right]  node {$\beta_1$} (s2)
            (s2)  edge[bend right=20, auto=right]  node {$\alpha_1$}  (s1)

          (s2)  edge[bend right=20, auto=right]    node {$\beta_{n-1}$} (s3)
            (s3)  edge[bend right=20, auto=right]  node {$\alpha_{n-1}$}  (s2)

             (s3) edge[auto=left] node {$\gamma_1$} (s5)
             (s4) edge[auto=right] node {$\delta_2$} (s6)
             (s5) edge[auto=left] node {$\delta_1$} (s6)
              (s3) edge[auto=left] node {$\delta_3$} (s6)
             (s3) edge[auto=right] node {$\gamma_2$} (s4);

 \end{tikzpicture}
 \caption{State diagram of birth-death process.}\label{fig:birthdeathp}
\end{center}
\end{figure}
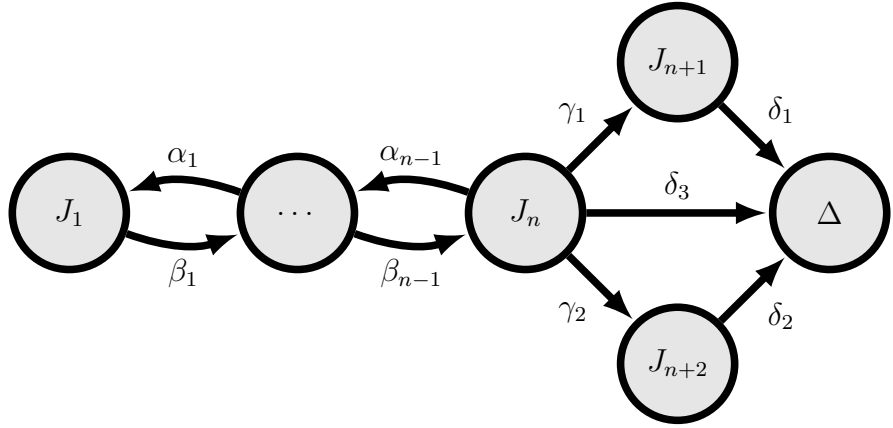

Consider a mixture of birth-death processes with state diagram described in Figure \ref{fig:birthdeathp}. The birth-death process has been widely used in many places such as, among others, in queueing theory, performance engineering, see \cite{Bolch}, demography, epidemiology and biology \cite{Nowak}. For simplicity, we set $\mathbb{S}=\{1,2,3,4,5\}\cup\{\Delta\}$ with $\Gamma_1=\{4,\Delta\}$ and $\Gamma_2=\{5,\Delta\}$. The intensity matrix $\mathbf{Q}^{(1)}$ of $X^{(1)}$ is given by
  \begin{align*}
\mathbf{Q}^{(1)}=  \left(\begin{array}{cccccc}
 -\beta_1 & \beta_1 & 0 & 0 & 0 & 0 \\
  \alpha_1 & -(\beta_2+\alpha_1) & \beta_2 & 0 & 0 & 0 \\
  0 & \alpha_2 & - (\alpha_2 + \gamma_1 + \gamma_2 + \delta_3)  &\gamma_1 & \gamma_2 & \delta_3 \\
  0 & 0 & 0 & -\delta_1 & 0 & \delta_1 \\
  0 & 0 & 0 & 0 & -\delta_2 & \delta_2\\
  0 & 0 &0 &0 &0 &0
\end{array}\right),
\end{align*}
while the intensity matrix of $X^{(2)}$ is defined by $\mathbf{Q}^{(2)}=\boldsymbol{\Psi}\mathbf{Q}^{(1)}$. Following \cite{Frydman2008} and \cite{Surya2018} we choose $\boldsymbol{\Psi}=\psi\mathbf{I}$, with $\psi\geq0$, whilst the initial switching probability matrix is defined by $\mathbf{S}=0.5\mathbf{I}$. For numerical purposes, we set $\beta_1=\beta_2=2$, $\alpha_1=\alpha_2=0.5$, $\gamma_1=\gamma_2=1$ and $\delta_i=1$, $i=1,2,3$. The initial probability $\boldsymbol{\pi}$ of starting the process $X$ at any of the $5$ states is given by $\boldsymbol{\pi}=(0.6,0.3,0.1,0,0)^{\top}$.

Numerical results on getting various shapes of conditional bivariate density functions (\ref{eq:jointpdfitex}) and (\ref{eq:jointpdfitex2}) as function of time $t$ are presented in Figures 3 - 6.

\begin{figure}[ht!]
\centering
\begin{tabular}{cc}
\subf{\includegraphics[width=72.5mm]{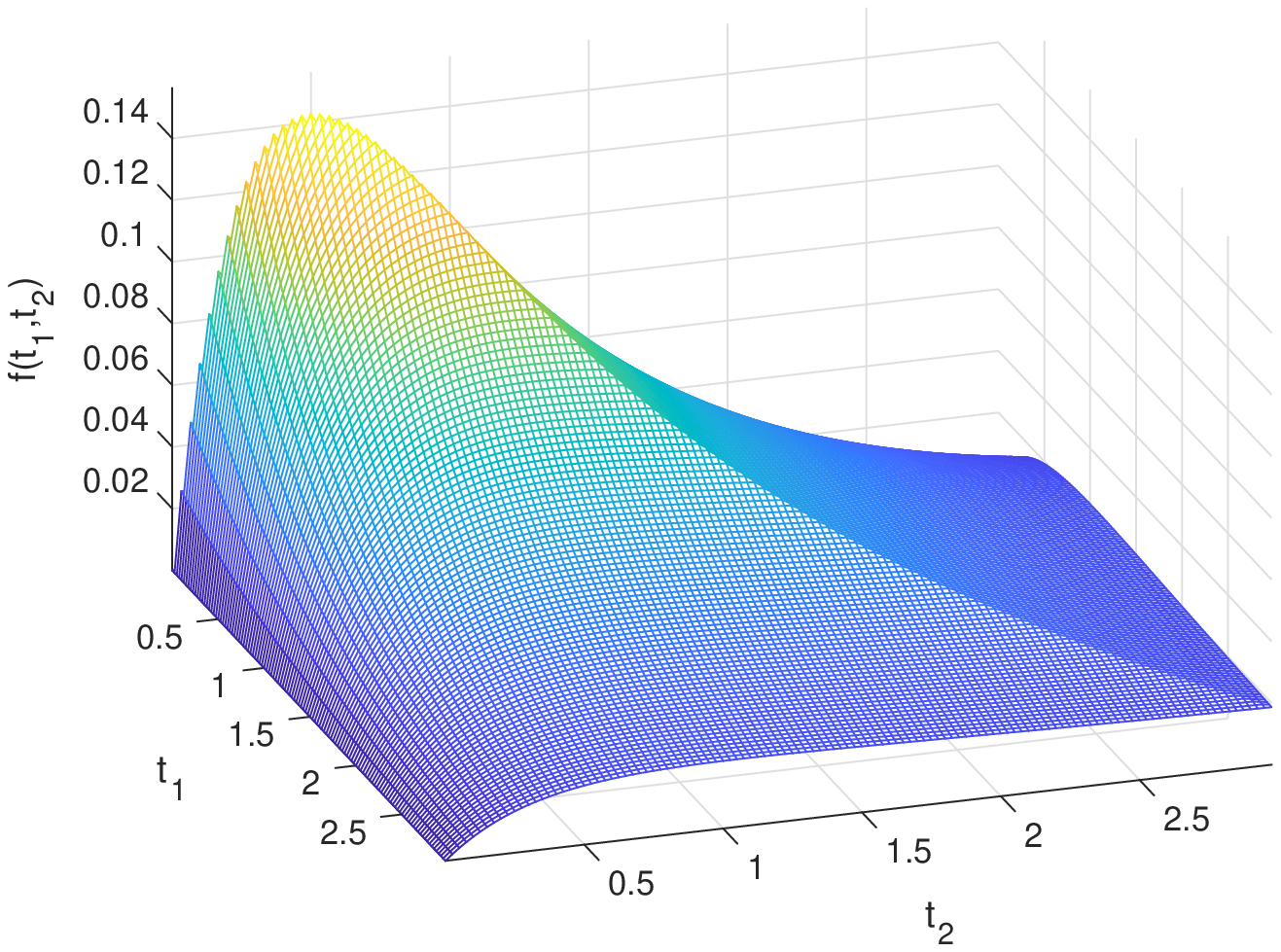}}
       {$f_{i,t}(t_1,t_2)$ with $i=2$, $\psi=0.5$, $t=0$.}

&
\subf{\includegraphics[width=72.5mm]{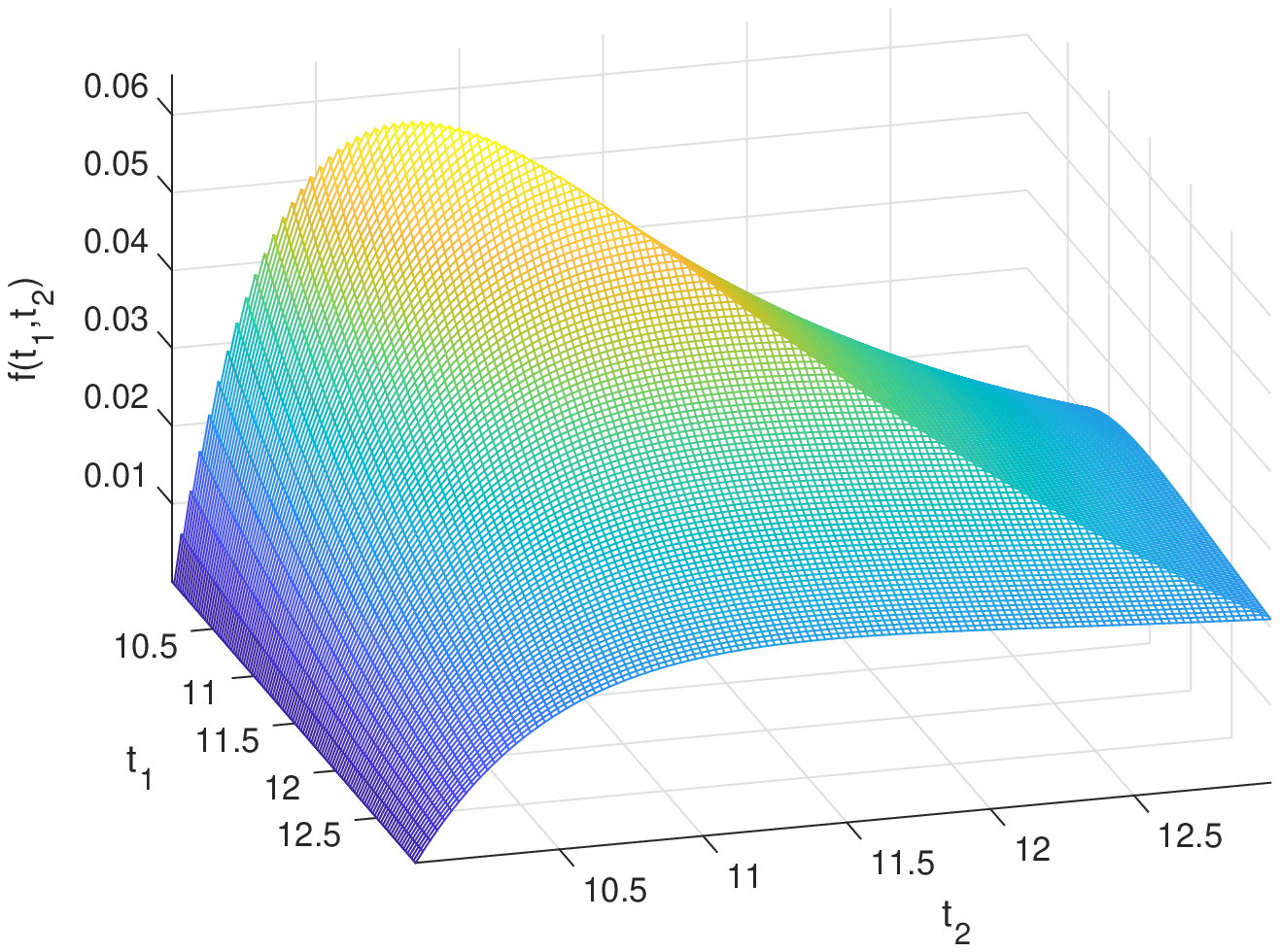}}
     {$f_{i,t}(t_1,t_2)$ with $i=2$, $\psi=0.5$, $t=10$.}
\\
\subf{\includegraphics[width=72.5mm]{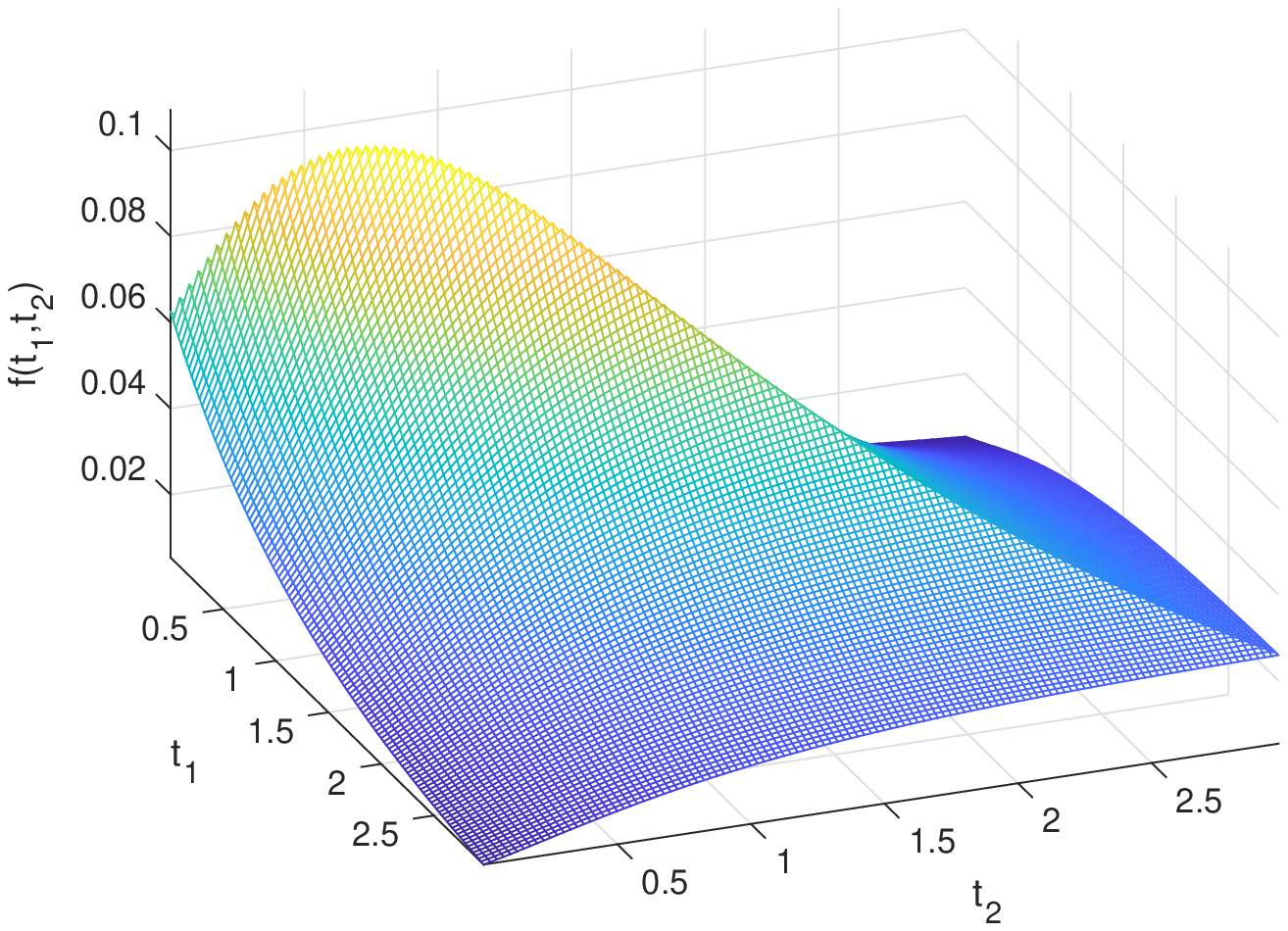}}
       {$f_{t}(t_1,t_2)$ with $\psi=0.5$, $t=0$.}
&
\subf{\includegraphics[width=72.5mm]{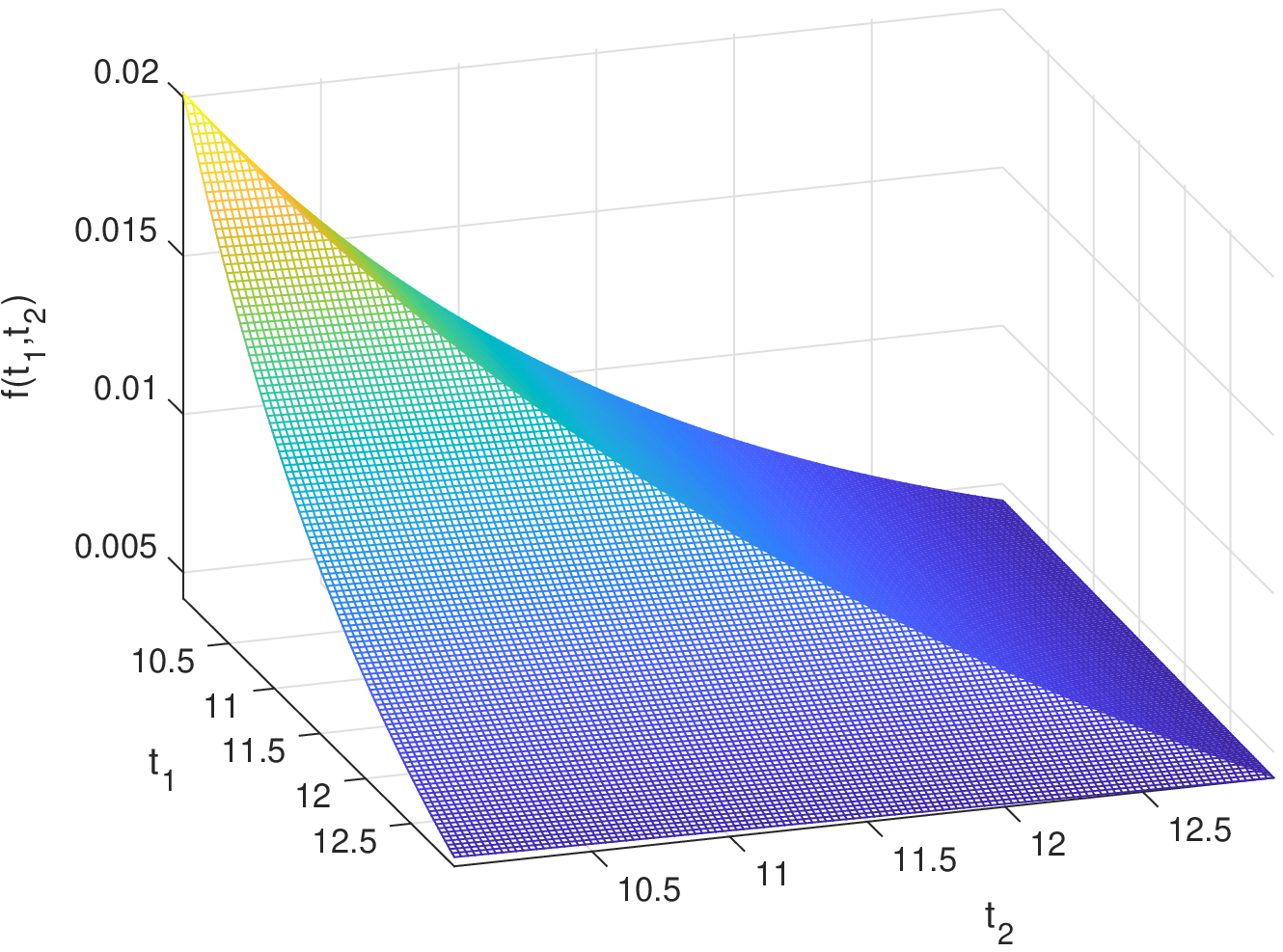}}
     {$f_{t}(t_1,t_2)$ with $\psi=0.5$, $t=10$.}
\\
\end{tabular}
\caption{ Plots of $f_{i,t}(t_1,t_2)$ (\ref{eq:jointpdfitex}) and $f_{t}(t_1,t_2)$ (\ref{eq:jointpdfitex2}) and their stationary shapes.}\label{fig:birthdeathgam}
\end{figure}

The shape of the density functions $f_{i,t}(t_1,t_2)$ (\ref{eq:jointpdfitex}) and $f_t(t_1,t_2)$ (\ref{eq:jointpdfitex2}) are displayed in Figure 3. The first plot in the top pictures exhibits the initial shape of $f_{i,t}(t_1,t_2)$ when $X_t$ starts in state $i=2$ at time zero, whereas the second plot presents the shape of stationary probability density function of $\tau_1$ and $\tau_2$ given that the process starts in the same state $i=2$ at time $t=10$. The picture clearly shows that the function has zero value at initial time and left skewed. The two pictures below which represent the function $f_t(t_1,t_2)$ (\ref{eq:jointpdfitex2}) when the process starts at a random initial state in $E$ at time $t=0$ and $t=10$, respectively. The probability of starting the process at any given time $t\geq 0$ is given by $\boldsymbol{\pi}(t)$. Note that we have used $\psi=0.5$, by which $X^{(1)}$ moves two times faster than $X^{(2)}$ does. The function has nonzero value at initial time. However, unlike the two pictures above which, the function losses its hump shape in the long run. We can see this more detailed in Figure 6 in terms of the marginal density function of $\tau_1$ and $\tau_2.$.

We observe that the joint density function $f_{i,t}(t_1,t_2)$ changes its shape as time $t$ increases, a feature that lacks in the Markov model $(\psi=1)$. Given that $\mathbf{S}=0.5\mathbf{I}$, the initial profile of density function $f_t(t_1,t_2)$ (for $t=0$) forms a mixture of bivariate phase-type distributions $f_{\boldsymbol{\pi},\mathbf{B}^{(1)}}(t_1,t_2)$ and $f_{\boldsymbol{\pi},\mathbf{B}^{(2)}}(t_1,t_2)$, i.e., $f_t(t_1,t_2)=0.5f_{\boldsymbol{\pi},\mathbf{B}^{(1)}}(t_1,t_2) + 0.5f_{\boldsymbol{\pi},\mathbf{B}^{(2)}}(t_1,t_2)$, where $f_{\boldsymbol{\pi},\mathbf{B}^{(k)}}(t_1,t_2)$, $k=1,2$, is obtained by setting $\mathbf{B}^{(1)}=\mathbf{B}^{(2)} $ in (\ref{eq:jointpdfitex2}), see e.g. \cite{Assaf1984}. In contrary to \cite{Assaf1984}, the distribution $f_t(t_1,t_2)$ changes its shape as $t$ increases, as depicted in Figure 3. 

\begin{figure}[ht!]
\centering
\begin{tabular}{cc}
\subf{\includegraphics[width=72.5mm]{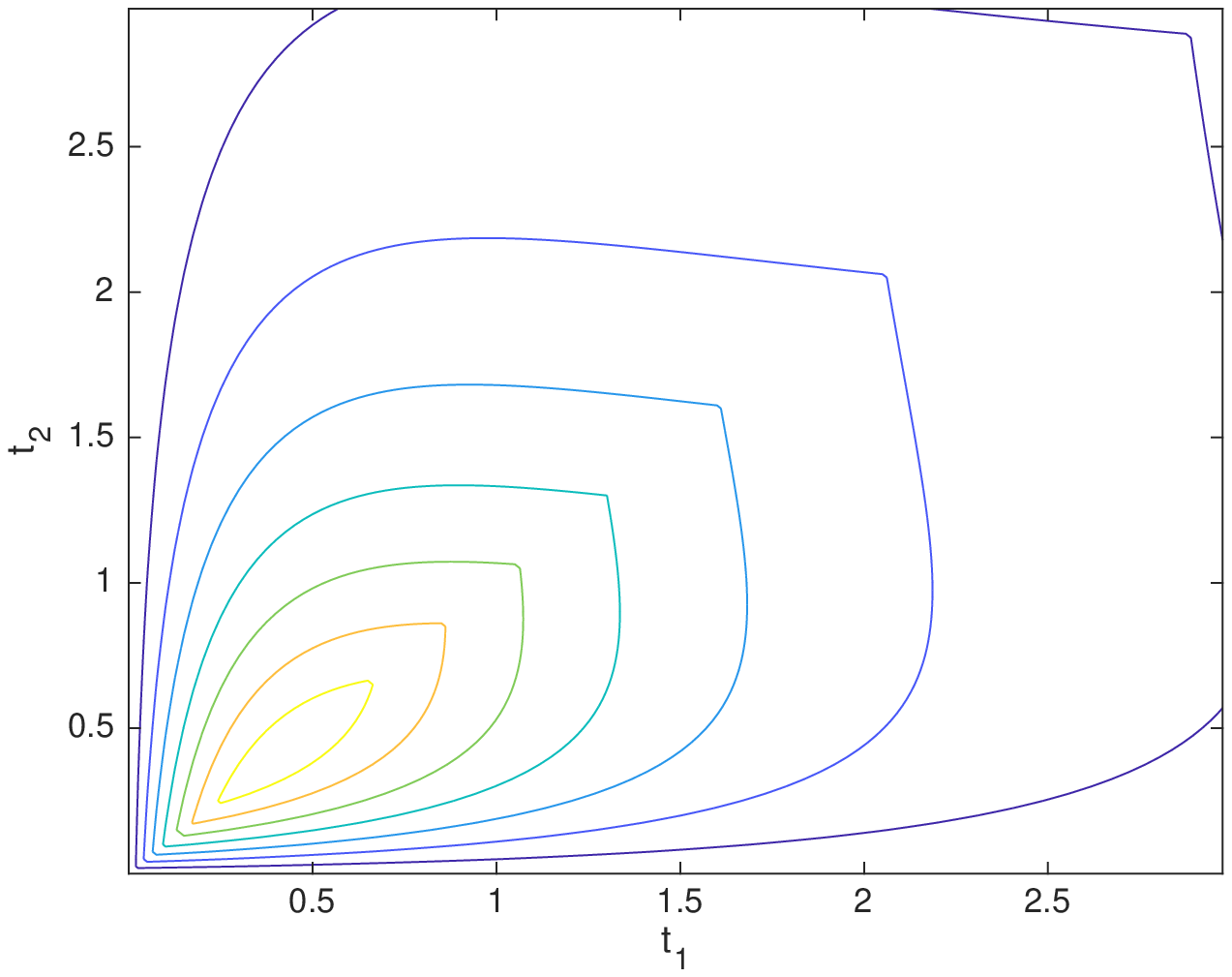}}
       {$f_{i,t}(t_1,t_2)$ with $i=2$, $t=0$.}

&
\subf{\includegraphics[width=72.5mm]{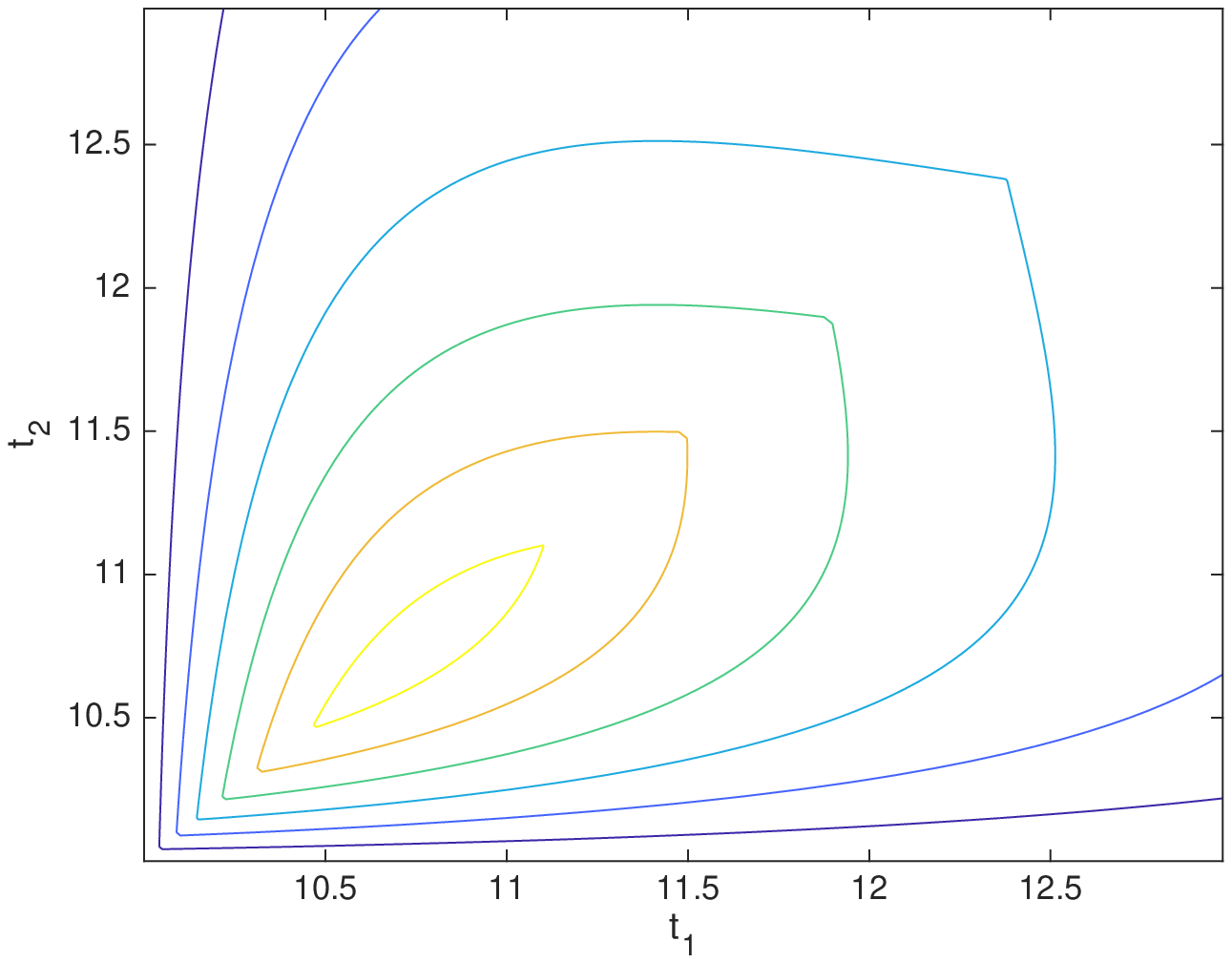}}
     {$f_{i,t}(t_1,t_2)$ with $i=2$, $t=10$.}
\\
\subf{\includegraphics[width=72.5mm]{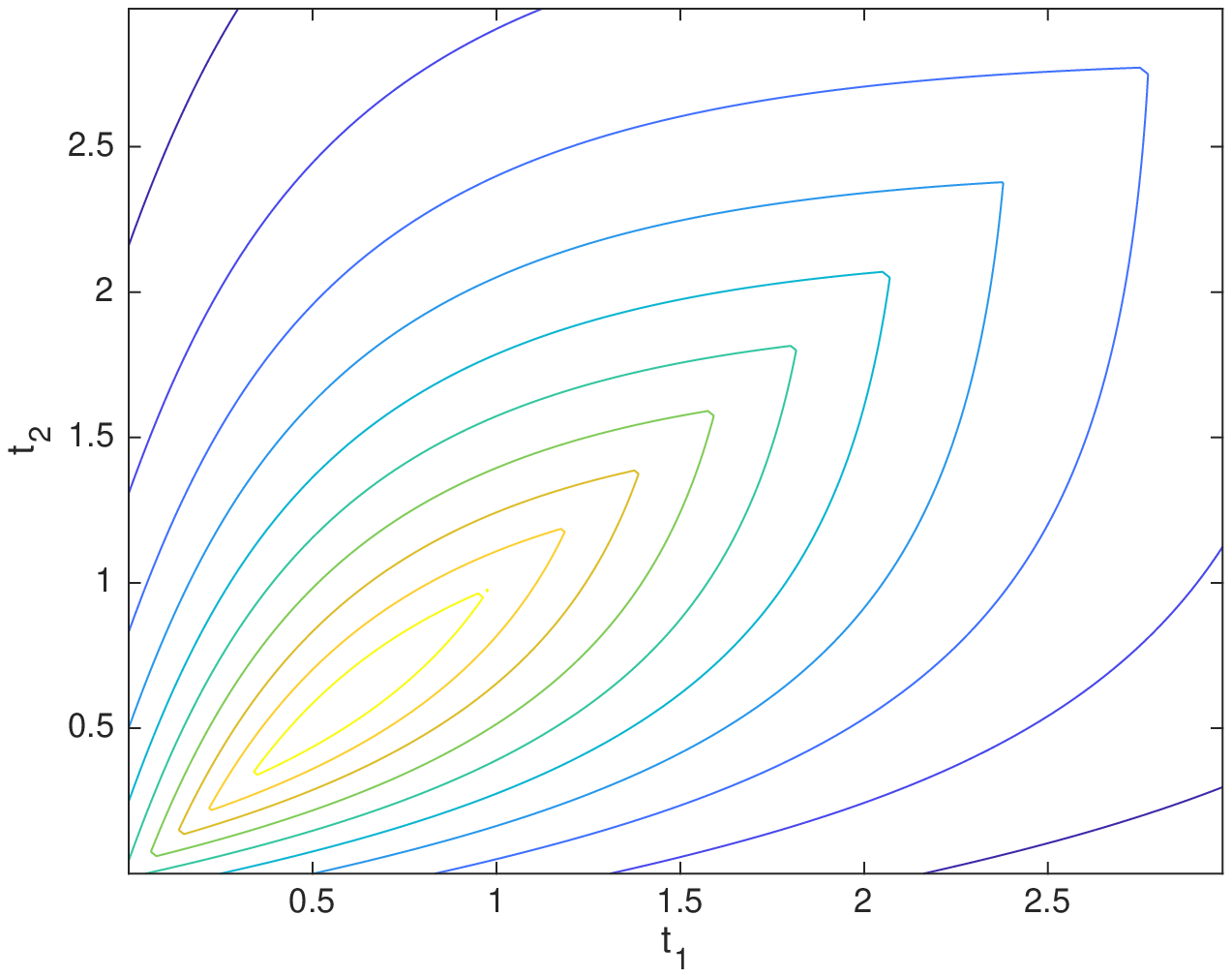}}
       {$f_{t}(t_1,t_2)$, $t=0$}
&
\subf{\includegraphics[width=72.5mm]{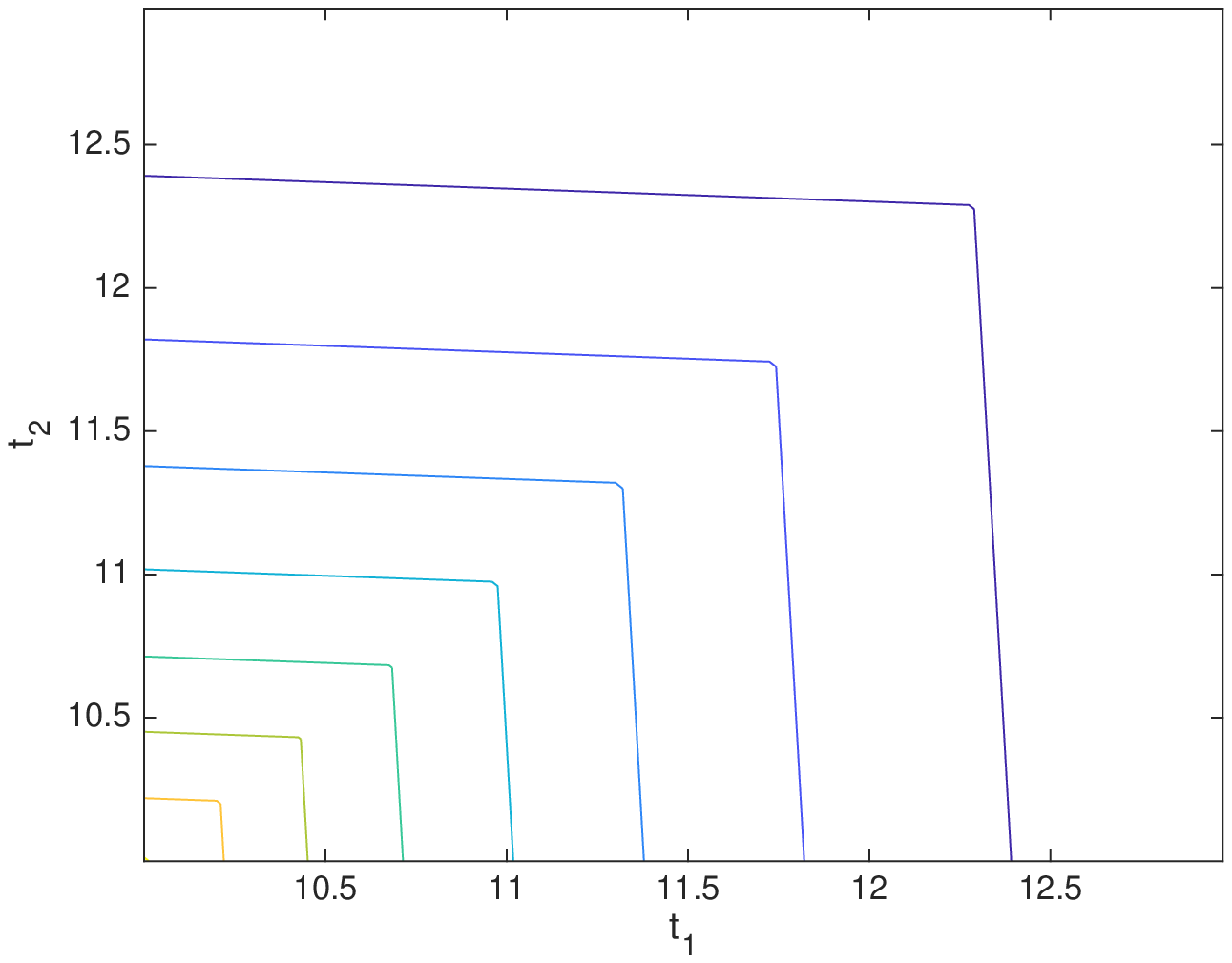}}
     {$f_{t}(t_1,t_2)$, $t=10$.}
\\
\end{tabular}
\label{fig:contour}\caption{Contour plots of $f_{i,t}(t_1,t_2)$ and $f_{t}(t_1,t_2)$ and their stationary shapes.}
\end{figure}

The stationary values of $\mathbf{S}_{11}(t)$ and $\boldsymbol{\alpha}(t)$ are given as $t\rightarrow \infty$ respectively by

  \begin{align*}
\mathbf{S}_{11}(\infty)=
\left(\begin{array}{ccc}
 1 & 0 & 0  \\
  0  & 1 & 0  \\
  0 & 0 & 1 
\end{array}\right) \;\; \textrm{and} \;\;
\boldsymbol{\alpha}(\infty)=
\left(\begin{array}{c}
 0.0245 \\
  0.0468 \\
  0.0381 \\
\end{array}\right),
\end{align*}
%

\noindent from which it follows that, conditional it is still alive in the long run, $X$ moves according to the Markov process $X^{(2)}$. Despite $\mathbb{1}^{\top}\boldsymbol{\alpha}(0)=1$, we have for $t>0$ that $0<\mathbb{1}^{\top}\boldsymbol{\alpha}(t)<1$. In all cases, the density has symmetry property for the values of parameters chosen. The contour plot in Figure 4 confirms this observation. 

The shape of marginal distributions $f_{\tau_1}^{(i)}(t_1\vert t)$ and $f_{\tau_1}(t_1\vert t)$ of $\tau_1$ are presented in Figure 5 for different values of speed parameter $\psi$. By symmetry, the marginal distributions of $\tau_2$ also share the same shape. Despite changing its shapes as $t$ increases, the pictures strongly suggest that the marginal pdf is left skewed and has zero value at zero for $f_{\tau_1}^{(i)}(t_1\vert t)$, and positive value for $f_{\tau_1}(t_1\vert t)$. They both decay to zero as $t_1$ increases as shown in more details in Figure 6. We also notice from the latter that the marginal probability density functions of $\tau_1$ and $\tau_2$ do not have common shape when the exit parameter $\delta_2$ changes its value. 

\begin{figure}[ht!]
\centering
\begin{tabular}{cc}
\subf{\includegraphics[width=72.5mm]{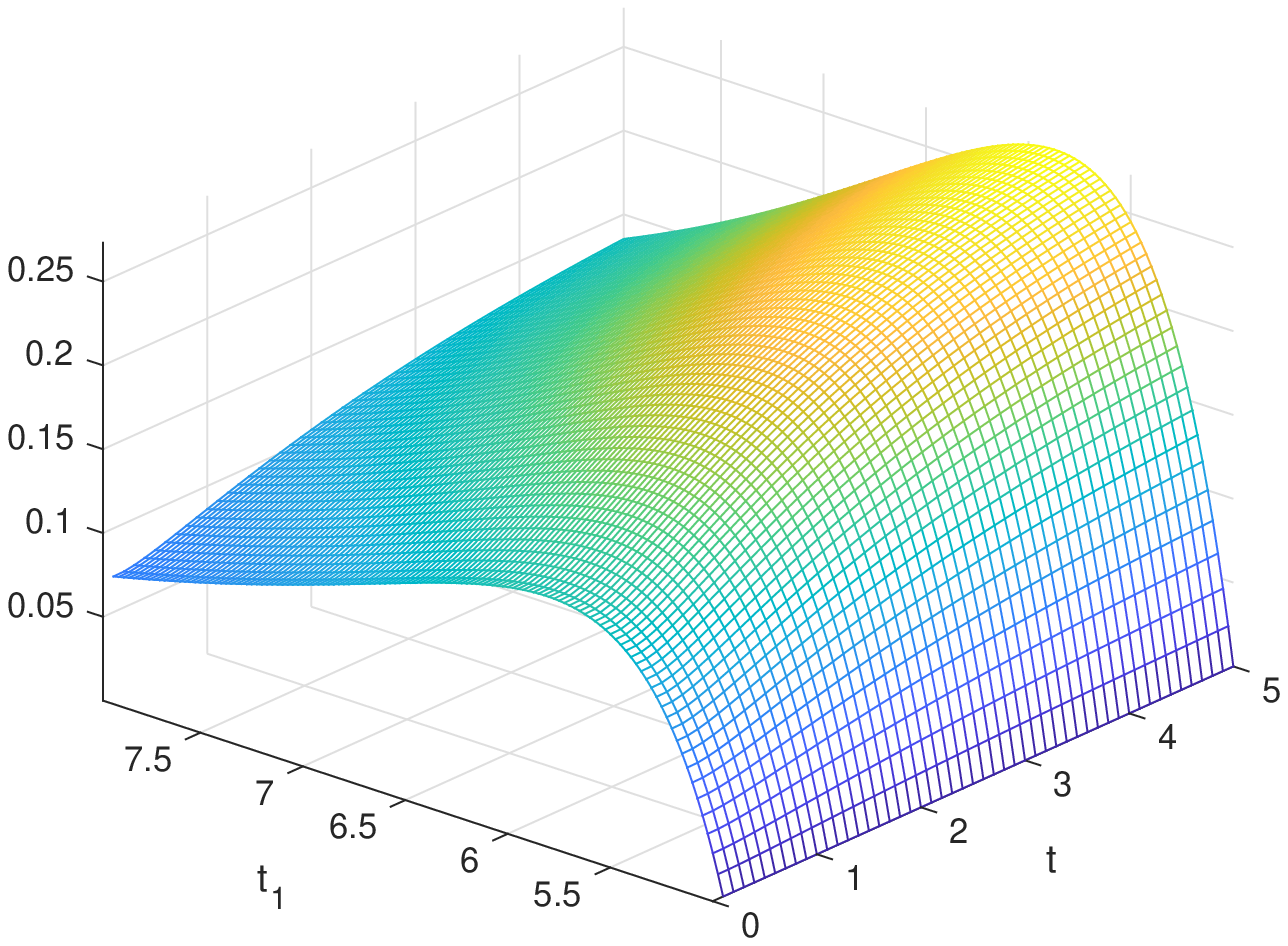}}
       { $f_{\tau_1}^{(i)}(t_1\vert t)$ with $i=2$, $\psi=0.5$.}

&
\subf{\includegraphics[width=72.5mm]{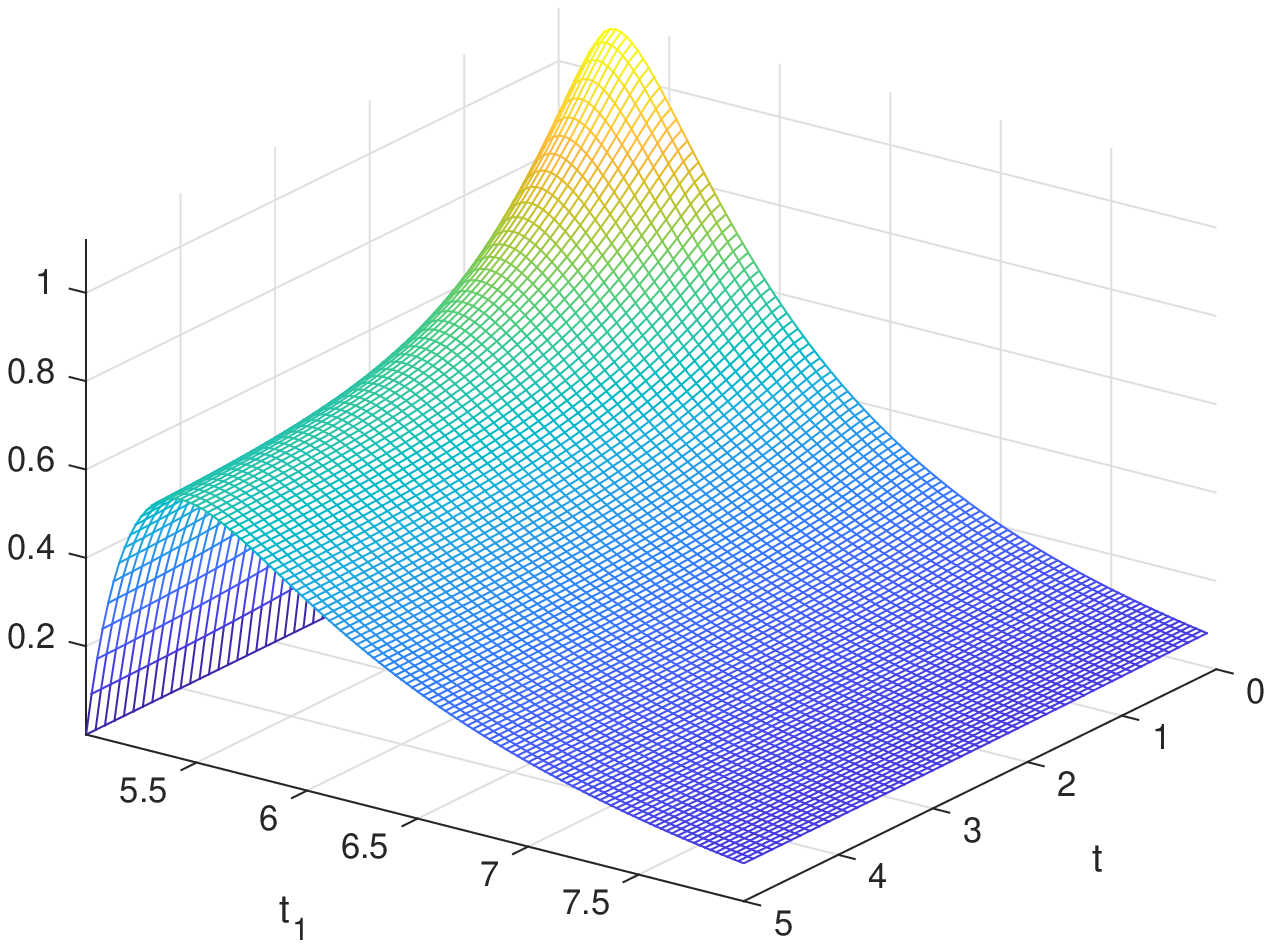}}
     { $f_{\tau_1}^{(i)}(t_1\vert t)$ with $i=2$, $\psi=2$.}
\\
\subf{\includegraphics[width=72.5mm]{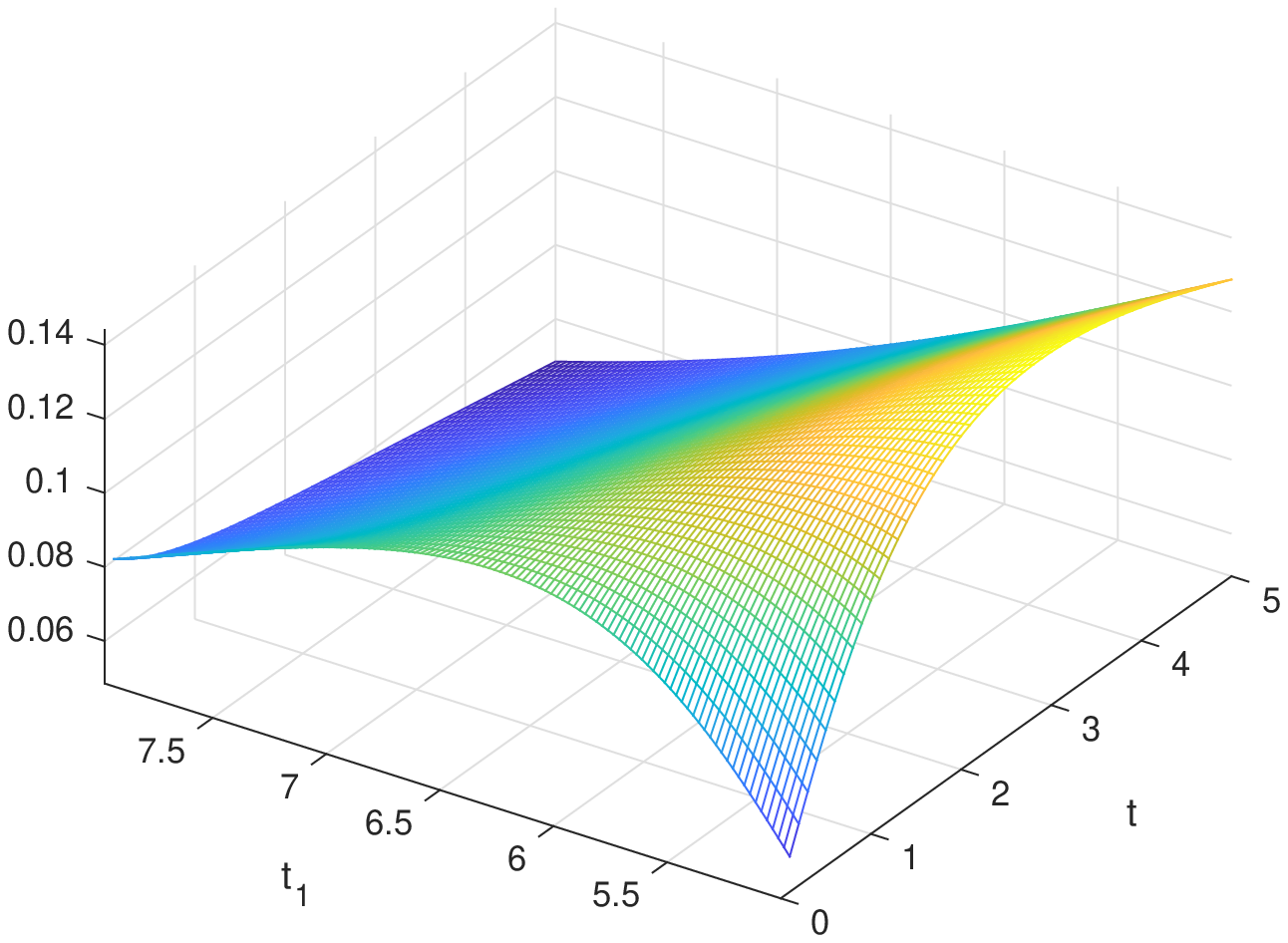}}
       { $f_{\tau_1}(t_1\vert t)$ with $\psi=0.5$.}
&
\subf{\includegraphics[width=72.5mm]{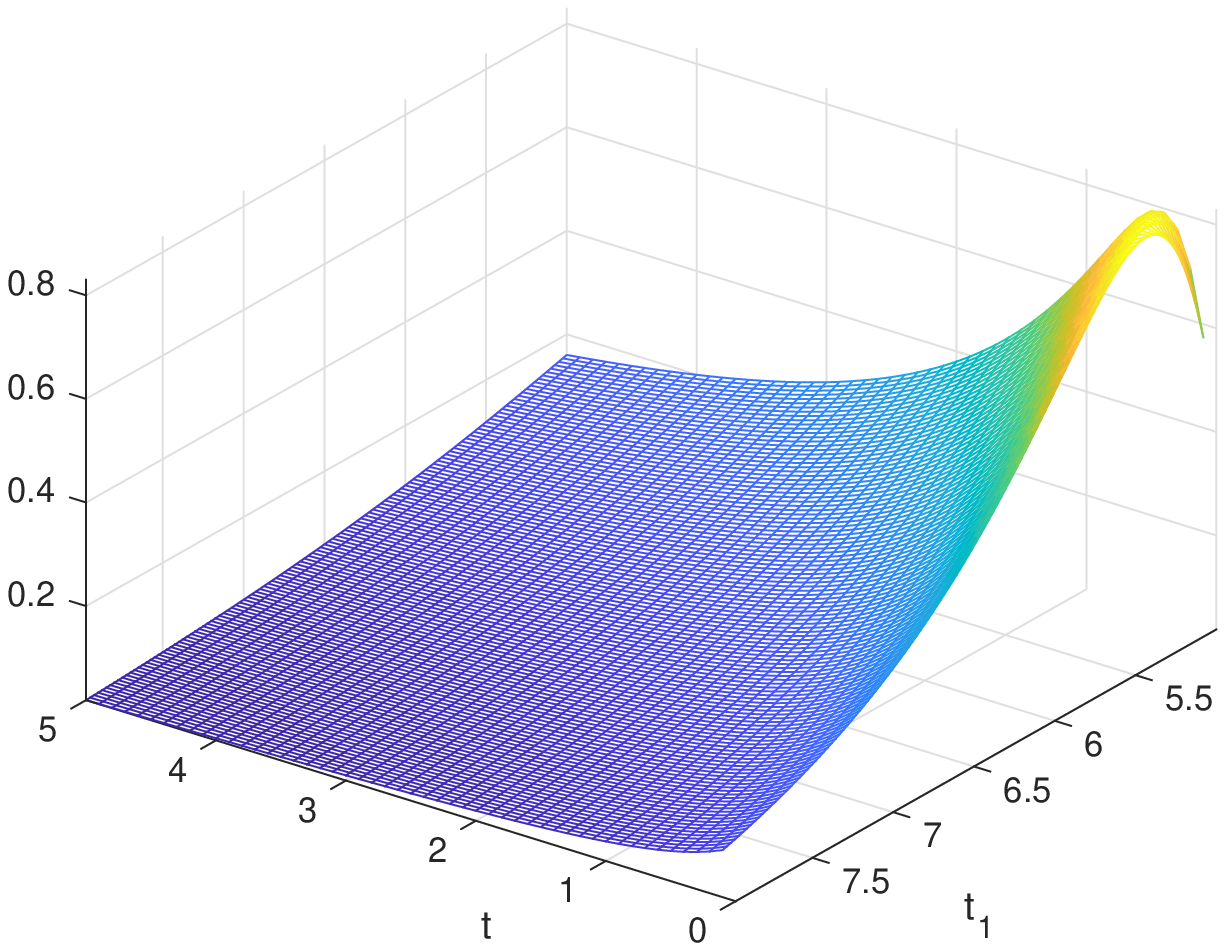}}
     { $f_{\tau_1}(t_1\vert t)$ with $\psi=2$.}
\\
\end{tabular}
\label{fig:marginalpdf}\caption{ Marginal pdfs  $f_{\tau_1}^{(i)}(t_1\vert t)$ and  $f_{\tau_1}(t_1\vert t)$ of $\tau_1$ as function of $t_1$ and $t$.}
\end{figure}

\begin{figure}[ht!]
\centering
\begin{tabular}{cc}
\subf{\includegraphics[width=72.5mm]{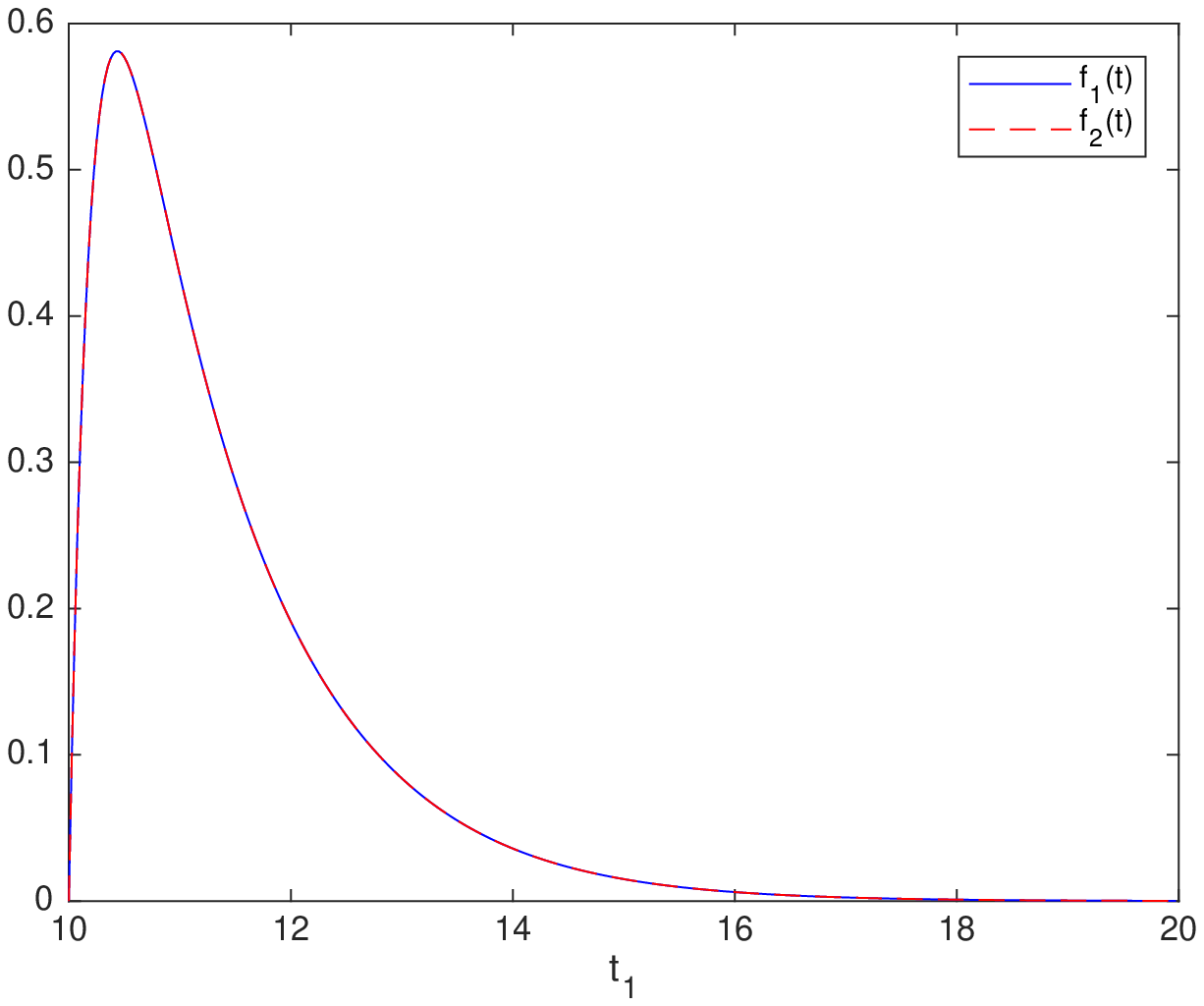}}
       { $i=2$, $\psi=2$, $\delta_2=1$.}

&
\subf{\includegraphics[width=72.5mm]{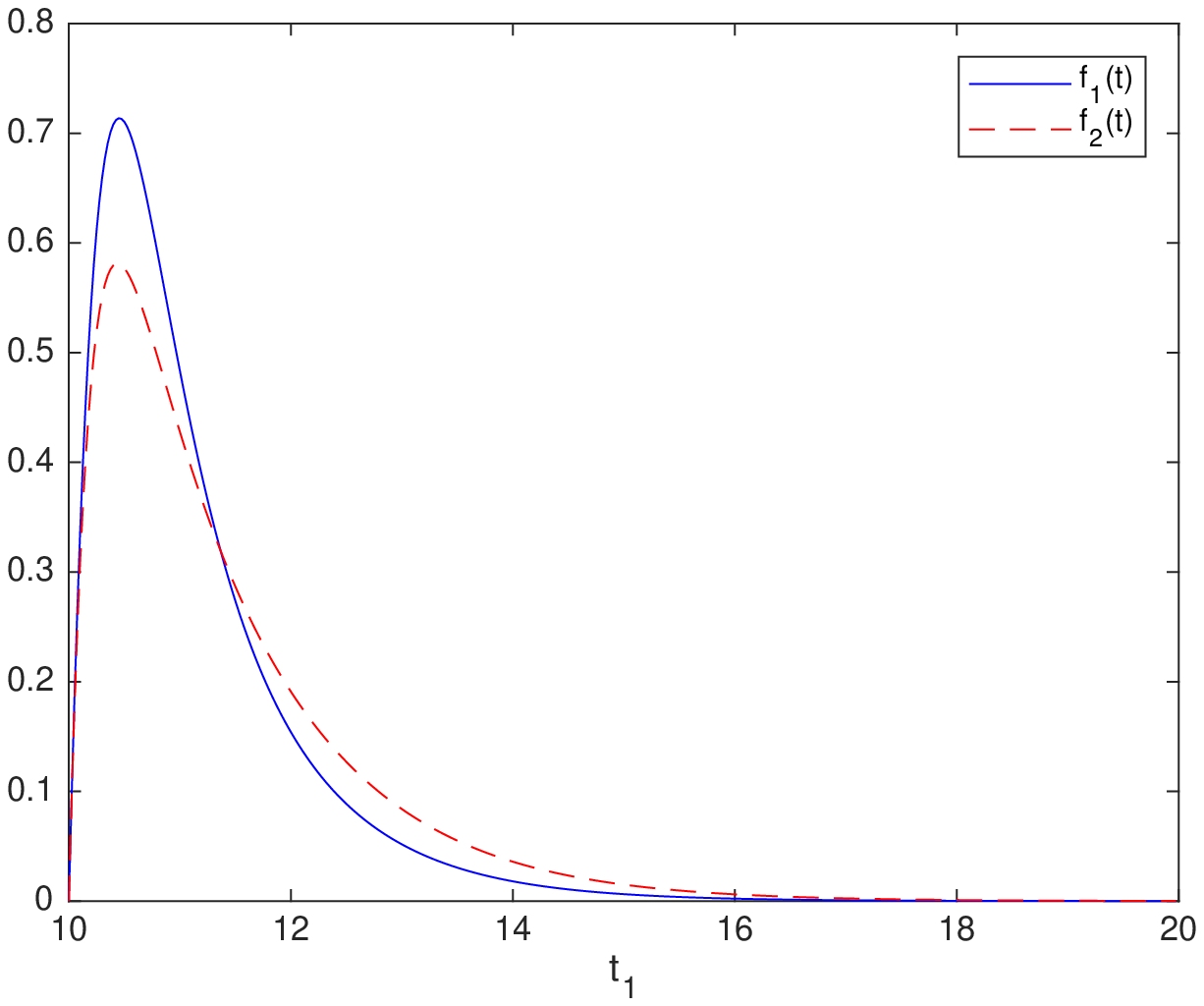}}
     {$i=2$, $\psi=2$, $\delta_2=5$.}
\\
\subf{\includegraphics[width=72.5mm]{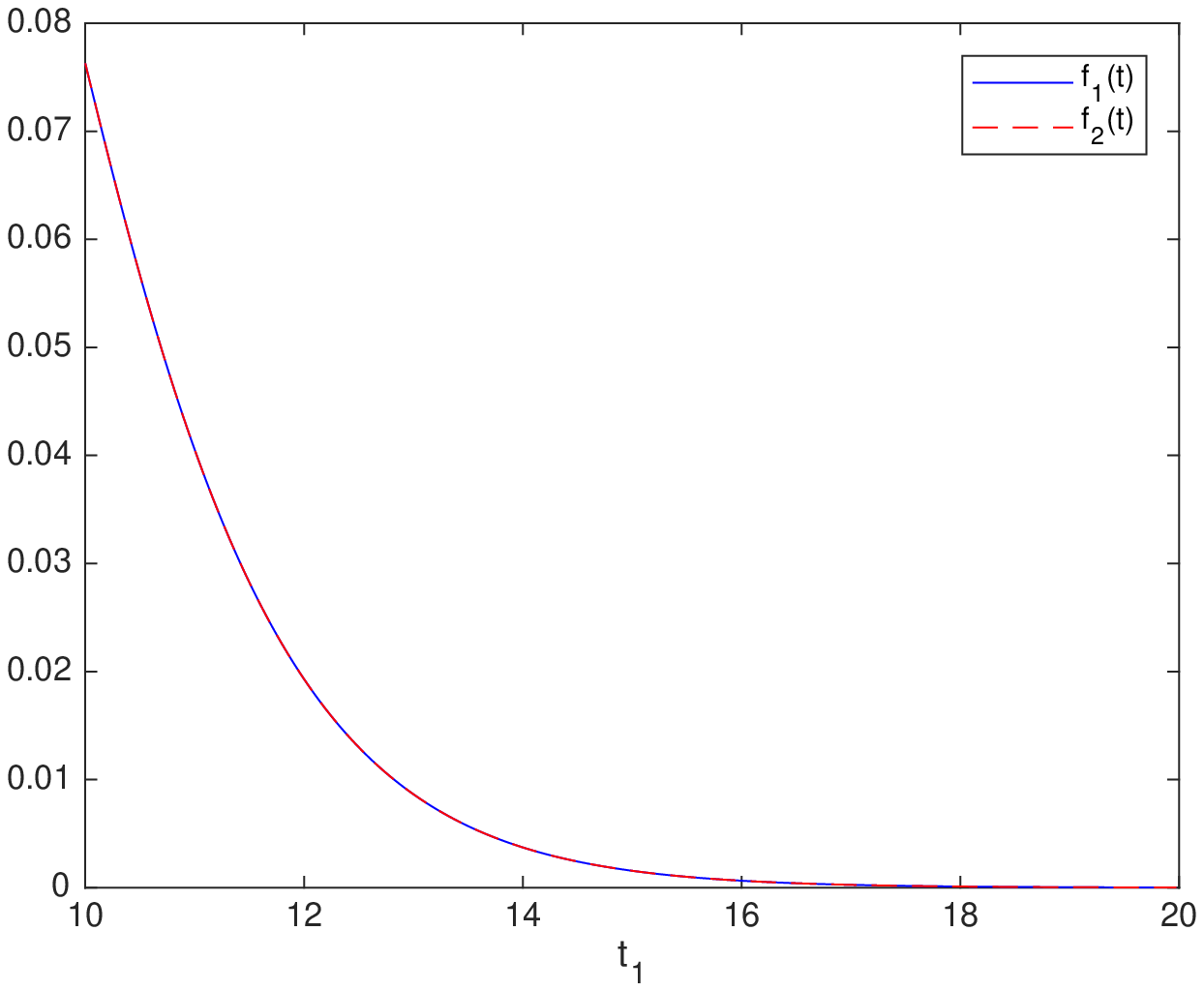}}
       { $\psi=2$, $\delta_2=1$.}
&
\subf{\includegraphics[width=72.5mm]{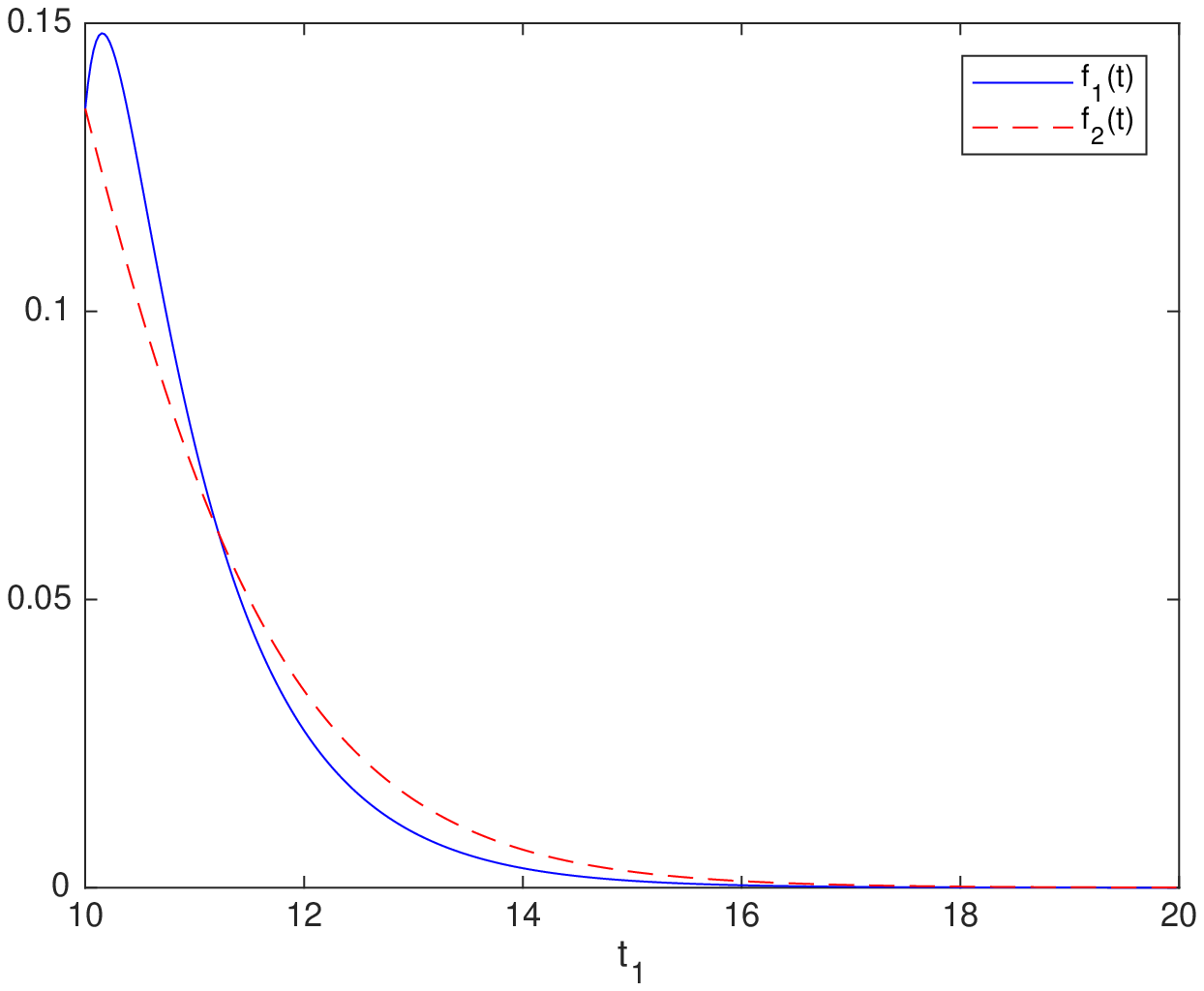}}
     { $\psi=2$, $\delta_2=5$.}
\\
\end{tabular}
\label{fig:marginalpdf2}\caption{ Stationary distribution of $f_{\tau_k}^{(i)}(t_1\vert t)$ and $f_{\tau_k}(t_1\vert t)$, $k=1,2$, for $t=10$.}
\end{figure}

\section{Conclusions}

We have introduced a new class of conditional joint probability distributions of first exit times of a continuous-time stochastic process defined as a finite mixture of right-continuous Markov jump processes, with overlapping absorbing sets, moving at different speeds on the same finite state space, while the mixture occurs at a random time. Distributional properties of the mixture process were discussed in general case, in particular the Bayesian update on the probability of starting the process in any phase of the state space at a given time, based on past observation of the process. The results presented in this paper generalizes that of given in \cite{Frydman2005}, \cite{Frydman2008} and \cite{Surya2018}. The new distributions form non-stationary functions of time and have the ability to capture heterogeneity and path dependence when conditioning on the available information (either full or partial) of the process. The attribution of path dependence is due to non-Markov property of the process.

Distributional identities are presented explicitly in terms of the intensity matrices of the underlying Markov processes, the Bayesian updates of switching probability and of the probability of starting the process in any of the phases in the state space, despite the fact that the mixture itself is non-Markov. In particular, the initial distributions form of a generalized mixture of the multivariate phase-type distributions of Assaf et al. \cite{Assaf1984}. When the underlying processes move at the same speed, in which case the mixture becomes a simple Markov jump process, heterogeneity and path dependence are removed and the initial distributions reduce to \cite{Assaf1984}. As in the univariate case, the probability distributions have dense and closure properties under finite convex mixtures and finite convolutions. These properties emphasize the additional importance of the new distributions.

As we have shown in this paper, the Markov mixture process forms a tractable construction of a continuous-time stochastic process having non-Markov property. Given their availability in explicit form and tractability, the Markov mixture process and the new conditional multivariate probability distributions should be able to offer appealing features for variety of applications, in which the Markov chains and the (multivariate) phase-type distributions have played central role.

 \section{Acknowledgments}

 The author acknowledges some inputs and suggestions from participants of Risk and Stochastic Seminar of London School of Economics, Hugo Steinhaus Center of Mathematics Seminar of Wroclaw University of Science and Technology, and Insurance: Mathematics and Economics Conference during 16-18 July 2018 at UNSW Sydney, at which part of the results of this work were presented. Financial support from Victoria University of Wellington \# 218772 is also acknowledged.

\end{document}